\documentclass[11pt]{amsart}
\usepackage{amsfonts, amsmath, enumerate, amssymb, xypic}

\newtheorem{theorem}{Theorem}[section]

\newtheorem{lemma}[theorem]{Lemma}
\newtheorem{corollary}[theorem]{Corollary}
\newtheorem{proposition}[theorem]{Proposition}
\theoremstyle{definition}
\newtheorem{definition}[theorem]{Definition}
\theoremstyle{notation}

\newtheorem{terminology}[theorem]{Terminology}

\newtheorem{examples}[theorem]{Examples}

\theoremstyle{remark}
\newtheorem{remark}[theorem]{Remark}
\newtheorem{remarks}[theorem]{Remarks}

\numberwithin{equation}{section}
\usepackage{amscd}
\usepackage{xypic}
\usepackage{amsmath, amssymb}

\numberwithin{equation}{subsection}
\newcommand{\be}%
  {\protect\setcounter{equation}{\value{subsubsection}}}
  \newcommand{\ee}%
   {\protect\setcounter{subsubsection}{\value{equation}}}

  {\protect\setcounter{subsubsection}{\value{equation}}}

\def \A{{\mathcal A}}

\def \B{\mathcal B}
\def \rmB{\rm B}
\def \BG{\rm {BG}}
\def \rmBW{\rm {BW}}
\def \rmBWT{\rm {BWT}}
\def \compl{\, \, {\widehat {}}}
\def \rmC{\rm C}
\def \bC{\mathbf C}

\def \Cl{\mathbb C}

\def \colimn{\underset {n \rightarrow \infty}  {\hbox {lim}}}
\def \colimI{\underset I {\hbox {colim}}}
\def \colimJ{\underset J {\hbox {colim}}}

\def \colimK.{\underset {\underset K^.  \rightarrow}  {\hbox {lim}}}

\def \colimU.{\underset {\underset U_.  \rightarrow}  {\hbox {lim}}}

\def \codim{\rm {codim}}

\def \rmD{\rm D}
\def \diag{{\rm diag}}
\def \diagNT{{\rm {diag}}({\rm {N}})}
\def \diagT{{\rm {diag}}({\rmT})}

\def \rmEG{\rm {EG}}
\def \rmEK{\rm {EK}}

\def \rmE{{\rm E}}
\def \rmET{\rm {ET}}
\def \EGx{{\rm {EG}}{\underset {\rm G}  \times}}
\def \ETx{{\rm {ET}}{\underset {\rm T}  \times}}
\def \EWx{\rmEW_{\underset {\rmW} \times}}

\def \EH\rmX{EH{\underset H  \times}\rmX}
\def \EM\rmX{EM{\underset M  \times}\rmX}

\def \EG\rmXx{EG{\underset \rmX  \times}}
\def \EG1{E{(G \times {\mathbb C}^*)}{\underset {G\times {\mathbb C}^*} 
\times}}

\def \EZ(s)1{E{(Z(s) \times {\mathbb C}^*)}{\underset {(Z(s)\times {\mathbb
C}^*)}  \times}}

\def \ET\rmX{{\rm {ET}}{\underset T  \times }\rmX}
\def \EG\rmX{{\rm {EG}}{\underset G  \times}\rmX}

\def \EGX{{\rm EG}{\underset {\rmG} \times} \rmX}

\def \eps{\in}
\def \EM(u){EM(u){\underset {M(u)}  \times}}
\def \EM(us){EM(u,s){\underset {M(u, s)}  \times}}

\def \EGm{{\rm E}{\rm G}^{{\rm {gm}},{\it m}}}

\def \EGm+1{{{\rm E}{\rm G}^{{\rm {gm}},{\it m+1}}{\underset {\rm G} \times}}}
\def \rmEW{\rm {EW}}
\def \rmEWT{\rm {EWT}}
\def \exp{{\rm exp}}
\def \itF{\it F}

\def \F{\mathbf F}

\def \bF{\mathbb F}

\def \gm{\rm {gm}}

\def \rmG{{\rm G}}
\def \GL{\rm {GL}}
\def\rmGL{\rm {GL}}
\def \rmH{\rm  H}
\def \H{\mathbb H}

\def \Hom{\it {Hom}}

\def \ICG\rmXp{{\mathcalI}{\mathcal C}^{G,\rmX}({\underline {KU}})}

\def \invlim1{\underset {\infty \leftarrow q}  {\hbox {lim}}^1}

\def \rmIC{{\rm I}{\rm C}}
\def \rmI{\rm I}
\def \rmJ{\rm J}
\def \rmK{\rm K}

\def \oK{\rm K}

\def \itK{\it K}
\def \rmL{\rm L}
\def \itL{\it L}

\def \L3{\Lambda \times \Lambda \times \Lambda}
\def \L2{\Lambda \times \Lambda}
\def \limi{\underset i  {\hbox {lim}}}

\def \lim{\underset \leftarrow  {\hbox {lim}}}
\def \limm{\underset {\infty \leftarrow m}  {\hbox {lim}}}

\def \longright2arrow{{\overset \longrightarrow  {\overset {} 
\longrightarrow}}}

\def \limnu{\mathop{\textrm{lim}}\limits_{ \infty \leftarrow \nu}}

\def \L{\mathcal L}

\def \rmres{{\rm res}}
\def \rmL{\rm L}
\def \Lie{{\rm Lie}}
\def \rmN{\rm N}
\def\rmNT{\rm {N(T)}}

\def \O{{\mathcal O}}

\def \rmp{\rm p}

\def \rmPGL{\rm {PGL}}
\def \Q{{\mathbb  Q}}

\def \ra{\rightarrow}

\def \RG^{R(G)^{\hat {}}\ }

\def \res{respectively}
\def \RHom{\it {RHom}}

\def \rmT{\rm T}
\def \rmR{\rm R}

\def \Sh{\rm {Sh}}

\def \rmS{\rm S}

\def \S{\mathcal S}

\def \oSpec{\rm {Spec}}

\def \STop{\rm {STop}}
\def \rmSL{\rm {SL}}
\def \topGcoh*{^{top, *} _{\rmG}}
\def \topGho*{ _{top,*} ^{\rmG}}

\def \bT{{\mathbf T}}
\def \T{\mathbf T}
\def \sT{\mathcal T}
\def \bT{{\mathbf T}}

\def \Top{\rm {Top}}
\def \rmU{\rm U}
\def \rmV{\rm V}
\def \rmWT{{\rm W}{\rm T}}
\def \rmW{\rm W}

\def \wY{\rm {wY}}
\def \wY '{\rm {wY '}}
\def \itw{\it w}
\def \itx{\it x}

\def \rmX{\rm X}
\def \rmY{\rm Y}

\def \Z(s){Z(s) \times {\mathbb C}^*}
\def \Z{\mathbb Z}

\def \rmZ{\rm Z}
\begin{document}
\title{Equivariant  derived categories for Toroidal Group Imbeddings}
%\author{Michel Brion}
%  Address of record for the research reported here
%\address{Institut Fourier, Grenoble, France.}
%\email{Michel.Brion@univ-grenoble-alpes.fr}
%\thanks{  }  
\author{Roy Joshua}
\address{Department of Mathematics, Ohio State University, Columbus, Ohio,
43210, USA.}
%   Current address
%\curraddr{Max Planck Institut f\"ur Mathematik,
 %         P.O. Box 7280, Bonn 53072
  %        GERMANY}
\email{joshua.1@math.osu.edu}
    %\thanks will become a 1st page footnote.
\thanks{The author was supported by a grant
from the NSF. }
\begin{abstract} 
Let $\rmX$ denote a projective variety over an algebraically closed field on which a linear algebraic group acts with finitely
many orbits. Then,
a conjecture of Soergel and Lunts in the setting of Koszul duality and Langlands' philosophy, postulates that the equivariant
derived category of bounded complexes with constructible equivariant cohomology  sheaves on $\rmX$ is 
equivalent to a full subcategory of 
the derived category of modules over a graded ring defined as a suitable graded $Ext$. Only special cases
of this conjecture have been proven so far.
{\it The purpose of this paper is to provide a proof of  this conjecture for all projective toroidal imbeddings of complex reductive groups.} In fact, we show that the methods used by Lunts for a proof in the case of
toric varieties can be extended with suitable modifications to handle the toroidal imbedding case.
{\it Since every equivariant imbedding of a complex reductive group is dominated by a toroidal imbedding, the class of varieties for which our proof
applies is quite large.} 
\newline\indent
We also show that, in general, there exist a countable number of obstructions for this conjecture
to be true and that half of these vanish when the odd dimensional equivariant intersection cohomology sheaves on the
orbit closures vanish. This last vanishing  condition had been proven to be true in many cases of spherical varieties by Michel Brion and 
the author in prior work.

\end{abstract}
\maketitle 
\centerline{\bf Table of contents}
\vskip .2cm 
1. Introduction and statement of results.
\vskip .2cm
2. Toroidal Group Imbeddings and the Proof of Theorem ~\ref{mainthm.2.1.0}.
\vskip .2cm
3. Proof of Theorem ~\ref{mainthm.2}
%4. Horospherical varieties and the Proof of Theorem ~\ref{mainthm.3.0}.
\vskip .2cm
4. A  General Obstruction Theory and Conclusions: Proof of Theorem ~\ref{mainthm.4}.
\vskip .2cm
5. Comparison of Equivariant Derived Categories. 
\vskip .2cm
6. Appendix.
\markboth{Roy Joshua}{{Equivariant  Derived Categories for Toroidal Group Imbeddings}}
\input xypic
\vfill \eject
%%%%%%%%%%%%%%%%%%%%%%%%%%%%%%%%%%%%%%%%%%%%%%%%%%%%%%%%%%%%%%%%%%%%%
%%%%%%%%%%%%%%%%%%%%%%%%%%%%%%%%%%%%%%%%%%%%%%%%%%%%%%%%%%%%%%%%%%%%%
%% INTRODUCTION
%%%%%%%%%%%%%%%%%%%%%%%%%%%%%%%%%%%%%%%%%%%%%%%%%%%%%%%%%%%%%%%%%%%%%

\section{\bf Introduction} 
This paper concerns a variant of a conjecture attributed to Soergel and Lunts (see \cite{So98}, \cite{So01}, \cite{Lu95}) for the action of linear algebraic groups on
projective varieties with finitely many orbits and over any algebraically closed field $k$. 
Let $\rmX$ denote such a projective variety provided with the action of a linear algebraic group $\rmG$. Recall
that a sheaf $F$ on $\rmX$ is equivariant, if it satisfies the following condition: let $\mu, pr_2:\rmG \times \rmX \ra \rmX$
denote the group-action and projection to the second factor, \res. Then 
there is  given an  isomorphism 
$\phi: \mu^*(F) \ra pr_2^*(F)$ satisfying a co-cycle condition on further pull-back to $\rmG \times \rmG \times \rmX$,
and which reduces to the identity on pull-back to $\rmX$ by the degeneracy map $x \mapsto (e, x)$. 
The conjecture asserts that the equivariant derived category $\rmD_{\rmG,c}^b(\rmX)$ of complexes of sheaves with bounded, 
equivariant and constructible
cohomology sheaves is equivalent to a full subcategory of the derived 
category of differential graded-modules over a certain differential graded algebra defined as a graded $Ext$ ring. 
\vskip .2cm
More precisely, the conjecture states
the following. 
Let ${\rm C}_{\rmG.c}^b(\rmX)$ denote the
category of complexes on a suitable Borel construction $\EG\rmX$ associated to $\rmX$, and with bounded constructible 
$\rmG$-equivariant cohomology sheaves. Let $L_i$ denote the  the equivariant intersection cohomology complex on
the  closure of a $\rmG$-orbit on $\rmX$, obtained by perverse  extension of an irreducible $\rmG$-equivariant local systems on the corresponding orbit. 
{\it We will next replace each such  $L_i$ by a complex of injective sheaves in ${\rm C}_{\rmG.c}^b(\rmX)$ upto quasi-isomorphism and fixed throughout the paper, but
will continue to denote the corresponding complex of injective sheaves by $L_i$ itself.} 
Let $L=\oplus_iL_i$ and let $\B=\B_{\rmG}(\rmX) = \Hom(\itL, \itL)$ ($=\RHom(\itL, \itL)$) denote the differential graded algebra where the multiplication is given by composition, and where the $Hom$ is computed in ${\rm C}_{\rmG,c}^b(\rmX)$.  (It is straightforward to verify that a different 
choice of a complex of injective sheaves replacing the $L_i$ up to quasi-isomorphism will provide a quasi-isomorphic differential graded algebra.) Let $Mod(\B)$
denote the category of differential graded modules over the differential graded algebra $\B$.
Then each $E_i = Hom(L, L_i)$ is an object of $Mod(\B)$, and the conjecture 
says that $\rmD_{\rmG,c}^b(\rmX)$ is equivalent to the full subcategory of $\rmD Mod(\B)$ generated by the  $E_i$ and 
that $\B$ is formal as a differential graded algebra, that is, it is quasi-isomorphic
as a differential graded algebra to its cohomology algebra. 
\vskip .2cm
When $k$ is the field of complex numbers, one can take the equivariant derived category to be made up of complexes of sheaves of 
$\Q$-vector spaces and in positive characteristics, one could consider instead the $\ell$-adic derived category considered in \cite{Jo93}.
Only special cases of this conjecture have been proven so far, all of them for certain classes of complex spherical varieties: 
the case of toric varieties was considered in \cite{Lu95}, the case of
smooth complete symmetric $\rmG$-varieties for the action of a  semi-simple adjoint group $\rmG$ was considered
in \cite{Gu05} and \cite{Sc11} considered the case of complex flag varieties.  
\vskip .2cm
{\it The main goal of this paper is 
to  prove this conjecture 
for an important sub-class of  complex projective spherical varieties, namely, toroidal imbeddings of complex reductive groups.}
\vskip .1cm \noindent
$\bullet$ Spherical varieties associated to reductive groups, forms a large class of algebraic varieties that includes as special cases,
both flag varieties as well as toric varieties, as two extreme cases. Moreover, spherical varieties are constructed using the combinatorial data of colored fans, similar to
how toric varieties are constructed from fans. As pointed out above, the conjecture has been already verified for both projective toric varieties and
also for flag varieties. Therefore, as the next step, it is important to consider the validity of the conjecture for other important sub-classes of
spherical varieties.
\vskip .1cm \noindent
$\bullet$ 
 Recall,  (see [BK05, Prop. 6.2.5]), that given a complex connected reductive group $\rmG$ and any projective $\rmG \times \rmG$-equivariant imbedding $X$ of $\rmG$, there exists a projective 
toroidal imbedding $\tilde{\rm X}$ of $\rmG$ together with a $\rmG \times \rmG$-equivariant birational map $\tilde{\rmX} \to \rmX$.
The above observation shows that the class of projective toroidal imbeddings is, in fact, an important subclass of spherical varieties.
\vskip .1cm \noindent
$\bullet$   Moreover,
toroidal imbeddings of reductive groups forms a sub-class of spherical varieties that is closest to toric varieties, in the
sense that they are also classified by fans and the local structure is that of toric varieties: see \cite[29.1]{Ti11} or \cite[Sec. 6.2.2]{BK05}. 
(More details on toroidal imbeddings, along with some examples may be found in section 2.2.) Thus, our main result, that is, Theorem ~\ref{mainthm.2}, shows that the above mentioned conjecture is indeed true for an important large class of spherical varieties.
\vskip .2cm
Along the way, the same techniques allow us to prove part of the conjecture for horospherical varieties. However, we have decided
it may be preferable, for a variety of reasons, to discuss the case of horospherical varieties separately elsewhere.
It may be also important to point out that the above conjecture also holds in positive characteristics, though hardly any results are known in this
case at present. The reason we restrict to complex spherical varieties in this paper, is largely because we make use of a reduction to
the action by a compact group as in Lemma ~\ref{phiK.acyclic.fibers}, though most of our remaining arguments seem to extend to positive 
characteristics using $\ell$-adic \'etale cohomology in the place of singular cohomology. Moreover, a proof of the conjecture for
projective toric varieties in positive characteristics seems to be in place, using certain stack-theoretic machinery in the place of the arguments in \cite{Lu95}.
 {\it In view of these,  we provide a discussion of equivariant derived categories in section 5 that works in all characteristics, but will restrict to complex
 algebraic varieties for the rest of the paper.}
\vskip .2cm
Next assume that  $\rmX$ is a variety on which a linear algebraic group $\rmG$ acts with only finitely many orbits. 
For each orbit $\O$, let $\L_{\O}$ denote an irreducible
$\rmG$-equivariant local system on $\O$. One may identify each $\L_{\O}$ with its extension by zero to an equivariant sheaf on all of $\rmX$. 
Since the cohomology sheaves of any complex $K \eps \rmD^b_{\rmG,c}(\rmX)$
are $\rmG$-equivariant, it is clear that the set $\{\L_{\O}|\O\}$ of all such $\rmG$-equivariant local systems as one varies over the $\rmG$-orbits, forms a set of generators for the derived category $\rmD^b_{\rmG,c}(\rmX)$ in the
sense that the smallest triangulated subcategory containing all of $\{\L_{\O}|\O\}$ and closed under finite sums is $\rmD^b_{\rmG,c}(\rmX)$.\footnote{The equivariant derived
 categories we consider in this paper will always be defined making use of the simplicial Borel construction. The reason for this choice,
 as well as a comparison with other models of equivariant derived categories, such as those considered in \cite{BL94} is discussed
  towards the end of this introduction.} 
  Let $\rmIC^{\rmG}(\L_{\O}[dim \O])$ denote the equivariant intersection cohomology complex on $\bar \O$ obtained by starting with the 
$\rmG$-equivariant local system $\L_{\O}$ on $\O$. Since the restriction of each $\rmIC^{\rmG}(\L_{\O}[dim \O])$ to the orbit $\O$ is
the local system $\L_{\O}$,  the set $\{\rmIC^{\rmG}(\L_{\O}[dim \O])|\O, \L_{\O}\}$  also  generates
 $\rmD^b_{\rmG,c}(\rmX)$.  Let $L_1, \cdots, L_n$ denote the above collection of equivariant intersection 
cohomology complexes. {\it We will next replace each such  $L_i$ by a complex of injective sheaves in ${\rm C}_{\rmG.c}^b(\rmX)$, but
will continue to denote the corresponding complex of injective sheaves by $L_i$ itself.} Let $L= \oplus _i L_i$.
\vskip .2cm
Next let $\B_{\rmG}(\rmX) = {\it {Hom}}(\itL, \itL)$. Let $\rmD(Mod(\B_{\rmG}(\rmX)))$  denote the
derived category of differential graded modules over the differential graded algebra $\B_{\rmG}(\rmX)$. 
Then we begin with the following theorem, which may be deduced readily using derived Morita theory. (See also 
\cite[(0.3.1) Proposition]{Lu95} for a variant.) Therefore, we skip its proof.
\begin{theorem}
 \label{mainthm.1}
Let $\rmX$ denote a variety on which a linear algebraic group acts with finitely many orbits.
 Then, sending $K \eps \rmD^b_{\rmG, c}(\rmX) \mapsto \RHom(\itL, \itK)$ defines a fully-faithful imbedding of  the equivariant 
derived category $\rmD^b_{\rmG, c}(\rmX)$  into the derived category $\rmD(Mod(\B_{\rmG}(\rmX)))$.
If $E_i$ denotes the image of $L_i$ under this imbedding and if $\rmD^f(Mod(\B_{\rmG}(\rmX)))$ denotes the triangulated
full sub-category of $\rmD(Mod(\B_{\rmG}(\rmX)))$ generated by the $E_i$, $i=1, \cdots, n$, then $\rmD^b_{\rmG, c}(\rmX)$ is
equivalent to $\rmD^f(Mod(\B_{\rmG}(\rmX)))$.
\end{theorem}
 With the above theorem in place, it remains to prove that $\B_{\rmG}(\rmX)$ is formal as a differential graded algebra, when $\rmX$ is a projective
 variety satisfying further assumptions. 
Sections 2 and  3 are devoted to a detailed discussion of toroidal group imbeddings in characteristic $0$, where we prove 
the formality of the
corresponding differential graded algebra thereby settling the conjecture  in this case and resulting in the following theorem.
\begin{theorem}
 \label{mainthm.2}
Let $\bar \rmG$ denote
a projective toroidal imbedding of the complex connected reductive group  $\rmG$ and let $\B_{\rmG \times \rmG}(\bar \rmG)$ denote the differential graded algebra
 $Hom( L, L)$ ($=\RHom(L, L)$) considered above.  Then $\B_{ \rmG \times \rmG}(\bar \rmG)$ is formal as a differential graded algebra.
\end{theorem}
 Let $\bar \rmG$ denote such a projective toroidal imbedding of $\rmG$  and let $\bar \rmT$ denote the closure of a
maximal torus $\rmT$ in $\bar \rmG$.  Our strategy here is to reduce to the case of the
 toric variety $\bar \rmT$: the picture is however, much more involved than the toric case considered in \cite{Lu95}, since one also has an action of the Weyl group
 on $\bar \rmT$ that needs to be considered, and considerable effort is needed to separate out the Weyl group action from the
torus action. We break the entire argument into the following key steps:
\vskip .2cm

\subsection {\it Step 1}
\label{step1}
Let $\rmN$ denote the normalizer of $\rmT$ in $\rmG$ and let $\diagT$ (resp. $\diagNT$) denote the image of $\rmT$ (resp. $\rmN$) in $\rmG \times \rmG$ under the diagonal imbedding
of $\rmG$.  Then, the $\rmG \times \rmG$-action on $\bar \rmG$ induces  an action of $(\rmT \times \rmT)\diagNT$ on $\bar \rmT$. Restricting the $\rmG \times \rmG$-action on $\bar \rmG$ to
an action by $(\rmT \times \rmT)\diagNT$ on $\bar \rmT$ induces a fully-faithful functor
\be \begin{equation}
 \label{eq.toroidal.0}
\rmD_{\rmG \times \rmG, c}^b(\bar \rmG) {\overset {\rmres} \ra}  \rmD_{(\rmT \times \rmT)\diagNT ,c}^b(\bar \rmT).
\end{equation} \ee
 Moreover, if ${\rmD}_{(\rmT \times \rmT)\diagNT ,c}^{b,o}(\bar \rmT)$ denotes the full subcategory of $\rmD_{(\rmT \times \rmT)\diagNT ,c}^{b}(\bar \rmT)$
generated by the  
\newline \noindent
$(\rmT \times \rmT)\diagNT$-equivariant sheaves that are constant along the orbits of $(\rmT \times \rmT)\diagNT $ on $\bar \rmT$, then 
the same functor induces an equivalence 
\be \begin{equation}
 \label{eq.toroidal.0.1}
\rmD_{\rmG \times \rmG, c}^b(\bar \rmG) {\overset {\rmres} \ra}  { \rmD}_{(\rmT \times \rmT)\diagNT ,c}^{b,o}(\bar \rmT). 
\end{equation} \ee
(See Theorem ~\ref{eq.toroidal.1} for further details.)
\vskip .2cm
\subsection {\it Step 2} 
\label{step2}
Next we observe that the subgroup $\diagT \subseteq (\rmT \times \rmT) \diagNT$ acts trivially on $\bar \rmT$. The quotient
$(\rmT \times \rmT)\diagNT/\diagT$ identifies with $\rmWT$, which is the semi-direct product of $\rmW$ and $\rmT$, where the $\rmW$ action on $\rmT$ is
induced from the action of $\rmNT$ on $\rmT$. For the following discussion, we will abbreviate $(\rmT \times \rmT) \diagNT)$ to $\tilde \rmN$.
Since $\diagT$ acts trivially on $\bar \rmT$, one obtains an induced action of $\rmWT$ on $\bar \rmT$. Let 
\be \begin{equation}
    \label{psi}
\psi: \rmE\tilde \rmN{\underset {\tilde \rmN} \times} \bar \rmT \ra \rmE \rmWT {\underset {\rmWT} \times}\bar \rmT
    \end{equation} \ee
denote the map induced by the identity on $\bar \rmT$ and the quotient map $\tilde \rmN \ra \rmWT$. Since $\diagT$ acts trivially on $\bar \rmT$, the fibers of the map at every point on $\bar \rmT$ can be identified
with $\rmB \diagT$ . Let $\rmG^{\bullet}$ denote the canonical Godement resolution. 
Then, clearly the functor $R\psi_*=\psi_*\rmG^{\bullet}$ sends complexes in ${ \rmD}^{b,o}_{\tilde \rmN, c}(\bar \rmT )$ to complexes of dg-modules over the sheaf of dg-algebras
$R\psi_*(\Q)$, that is, to objects in the derived category $\rm D_{\rmWT }(\bar \rmT, R\psi_*(\Q))$. (Here we observe that any complex $K \eps { \rmD}^{b,o}_{\tilde \rmN, c}(\bar \rmT )$
 comes equipped with a pairing $\Q \otimes K \ra K$, which induces the pairing $R\psi_*\Q) \otimes R\psi_*(K) \ra R\psi_*(K)$, and hence the structure of a dg-module over $R\psi_*(\Q)$ on $R\psi_*(K)$.)
 \vskip .2cm
 We let ${ \rmD}^{+,o}_{\rmWT, c}(\bar \rmT, R\psi_*(\Q))$
denote the full subcategory generated by the objects $R\psi_*(j_!\Q)$ where $j:\O \ra \bar \rmT$ denotes the immersion associated with an 
$\tilde \rmN$-orbit and we vary over such $\tilde \rmN$-orbits. Let
\[L\psi^*: { \rmD}^{+,o}_{\rmWT,c}(\bar \rmT,R\psi_*(\Q)) \ra { \rmD}^{b,o}_{\tilde \rmN, _c}(\bar \rmT)\]
denote the functor defined by sending a dg-module $M$ to $\Q{\overset L {\underset {\psi^{-1}R\psi_*(\Q)} \otimes}} \psi^{-1}(M)$.
Then the next key step is that the functor
\[R\psi_*: { \rmD}^{b,o}_{\tilde \rmN, c}(\bar \rmT) \ra {\rmD}^{+,o}_{\rmWT, c}(\bar \rmT, R\psi_*(\Q))\]
 is an equivalence of categories
with  $L\psi^*$ its inverse. (See Proposition ~\ref{Rpsi*} for further details.)
\vskip .2cm
\subsection {\it Step 3}
\label{step3}
The next step is to separate out the $\rmW$ and $\rmT$-actions. Since there is no obvious map between the
simplicial varieties $ \rmE \rmWT{\underset {\rmWT} \times} (\bar \rmT) $ and $\rmE \rmW{\underset {\rmW} \times} (\ETx \bar \rmT)$, we need to consider
the maps $p_1: \rmE \rmWT{\underset {\rmWT} \times} (\rmE \rmT \times \bar \rmT) \ra \rmE \rmW{\underset {\rmW} \times} (\ETx \bar \rmT) \mbox{ and }
p_2:\rmE \rmWT{\underset {\rmWT} \times} (ET \times \bar \rmT) \ra \rmE \rmWT{\underset {\rmWT} \times} (\bar \rmT) $. Let
 $\pi: \rmE \rmT{\underset {\rmT} \times} \bar \rmT \ra \bar \rmT/\rmT$ denote the obvious map as well as the induced map
$\rmE \rmW{\underset {\rmW} \times} (\rmE\rmT{\underset {\rmT} \times} \bar \rmT) \ra \rmE \rmW{\underset {\rmW} \times}(\bar \rmT/\rmT)$. (Here $\bar \rmT/\rmT$ denotes the
set of all $\rmT$-orbits on $\bar \rmT$ provided with the topology, where the closed sets
are the unions of $\rmT$-orbit closures. Then $\pi: \bar \rmT \ra \bar \rmT/\rmT$ is continuous.)
Then we prove the following theorem.
\begin{theorem} 
\label{mainthm.2.1.0}
 The restriction functor $\rmD^b_{\rmG \times \rmG,c}(\bar \rmG) \ra  \rmD^{b,o}_{\tilde \rmN,c}(\bar \rmT)$ and the 
functors $R\psi_*, p_2^*, Rp_{1*}$ and $R\pi_*$ induce equivalences of derived  categories:
\[ \rmD_{\rmG \times \rmG, c}^b(\bar \rmG) \simeq { \rmD}_{cart,c}^{+,o}(\rmE \rmW{\underset {\rmW} \times}(\bar \rmT/\rmT), R\pi_*(\A)), \mbox{ where } \A= Rp_{1*}p_2^*R\psi_*(\Q).
\]
Here ${ \rmD}^{+,o}_{cart,c}(\rmE \rmW{\underset {\rmW} \times} (\bar \rmT/\rmT), R\pi_*(\A))$ denotes the full 
subcategory of ${ \rmD}^+_{}(\rmEW{\underset W \times} (\bar \rmT/{\rmT}), R\pi_*(\A))$ 
generated by applying the functor $R\pi_*Rp_{1*}p_2^*$ 
to the generators of the subcategory ${ \rmD}^{+,o}_{\rmWT,c}(\bar \rmT, R\psi_*(\Q))$ (as in Step 2).
\end{theorem}
 Observe that the $t$-structure on the derived
category on the left is the $t$-structure whose heart consists of $\rmG \times \rmG$-equivariant perverse sheaves on $\bar \rmG$, while the corresponding $t$-structure on the derived category on the right is obtained by transferring the
$t$-structure from the derived category on the left. (See Proposition ~\ref{transf.t.struct}, for example.)
\vskip .2cm
\subsection {\it Step 4 (Final Step)}
\label{step4}
Our {\it next goal} is to show, making use of the equivalences of derived categories provided by the last three steps, 
that the dg-algebra $\B_{\rmG \times \rmG}(\bar \rmG)$ is quasi-isomorphic as a dg-algebra to
a dg-algebra defined in terms of the toric variety $\bar \rmT$, and provided with a compatible action by $\rmW$. 
We will adopt the following convention henceforth: if $\{K_i|i\}$ denotes a finite collection of complexes in an abelian category (with enough injectives),
$\RHom(\oplus K_i, \oplus _i K_i) = Hom(\oplus _i \hat K_i, \oplus \hat K_i)$ where $\hat K_i$ is a fixed replacement of $K_i$ by a complex of 
injectives up to quasi-isomorphism. Then we observe the following:
\begin{enumerate}[\rm(i)]
\item 
The $\rmG \times \rmG$-equivariant local systems on the $\rmG\times \rmG$-orbits are constant as observed in \cite[Lemma 3.6]{BJ04}.
If $\O$ denotes an orbit for the $\rmG \times \rmG$-action on $\bar \rmG$, and $\rmIC^{\rmG}(\O)$ denotes the corresponding equivariant intersection cohomology
complex extending the constant sheaf $\Q$ on the orbit $\O$, then it corresponds to $\rmIC^{\tilde \rmN}(\O_{\tilde \rmN})$ under 
the equivalence of derived categories provided by Step 1.
 Here $\O_{\tilde \rmN}$ is the $\tilde \rmN$-orbit on $\bar \rmT$ corresponding to the $\rmG \times \rmG$-orbit $\O_{\rmG}$. 
Since the equivalence of derived categories in Step 1 is provided by derived functors as in ~\eqref{eq.toroidal.0.1} relating
the two derived categories, they preserve the corresponding $RHom$ as well as the composition pairing on the $RHom$s. Therefore, the dga 
$\B_{\rmG \times \rmG}$ is quasi-isomorphic to the dga $\B_{\tilde \rmN}(\bar \rmT) = \RHom(\oplus _{\O} \rmIC^{\tilde \rmN}(\O_{\tilde \rmN}), \oplus _{\O} \rmIC^{\tilde \rmN}(\O_{\tilde \rmN}))$, where the sum varies
 over all the $\tilde \rmN$-orbits on $\bar \rmT$.
\item 
Under the equivalence of equivariant derived categories provided by Step 2,
the equivariant intersection cohomology complex $\rmIC^{\tilde \rmN}(\O_{\tilde \rmN})$ corresponds to $\rmIC^{\rmWT}(\O_{\rmWT })  \otimes R\psi_*(\Q)$, where
$\O_{\rmWT }$ denotes the same orbit $\O_{\tilde \rmN}$ of $\tilde \rmN$, but viewed as an orbit for the induced action of $\rmWT$.
Therefore, the equivalence of derived categories in Step 2 similarly provides a quasi-isomorphism between the dgas: $\B_{\tilde \rmN}(\bar \rmT)$ and
$\B_{\rmWT}(\bar \rmT) = \RHom (\oplus _{\O} \rmIC^{\rmWT}(\O_{\rmWT}) \otimes R\psi_*(\Q), \oplus _{\O} \rmIC^{\rmWT}(\O_{\rmWT}) \otimes R\psi_*(\Q))$.
\item 
Under the equivalence of derived categories provided by Step 3, the  complex
$\rmIC^{\rmWT }(\O_{\rmWT})\otimes R\psi_*(\Q)$ corresponds to $R\pi_*(\rmIC^{\rmW, \rmT}(\O_{\rmW,T }) \otimes \A)$. 
Here $\rmIC^{\rmW, \rmT}(\O_{\rmW,T })$ is the same complex as $\rmIC^{\rmWT }(\O_{\rmWT})$, but with the action of $\rmW$ and $\rmT$ separated out: 
see the proof of Corollary ~\ref{equiv.2}.
This similarly provides a quasi-isomorphism between the dgas: $\B_{\rmWT}(\bar \rmT) \simeq \B_{\rmW, \rmT}(\bar \rmT) =
\RHom( \oplus _{\O} R\pi_*(\rmIC^{\rmW, \rmT}(\O_{\rmW, \rmT}) \otimes \A), \oplus _{\O}R\pi_*( \rmIC^{\rmW, \rmT}(\O_{\rmW, \rmT}) \otimes \A))$
\end{enumerate}
The last $RHom$ is
taken in the category $ {\rm C}_{cart,c}^{+,o}(\rmE \rmW{\underset {\rmW} \times}(\bar \rmT/\rmT), R\pi_*(\A))$, which is
the category of complexes whose derived category is ${ \rmD}_{cart,c}^{+,o}(\rmE \rmW{\underset {\rmW} \times}(\bar \rmT/\rmT), R\pi_*(\A))$.
\vskip .1cm
Thus, in order to show the dg-algebra $\B_{ \rmG \times \rmG}(\bar \rmG)$ is formal, it suffices to prove that the
dg-algebra $  \B_{\rmW, \rmT}(\bar \rmT)$ is formal. After having done all of the above work to separate out the $\rmW$ and $\rmT$-actions,
the remaining part of this proof is similar to the proof in the toric case: see \cite{Lu95} and Theorem ~\ref{main.thm.5}. 
A key step in this proof is that the dg-algebra 
$R\pi_*(\A) = R\pi_*(Rp_{1*}p_2^*R\psi_*(\Q))$ is formal. 
\begin{remarks}
1. The toric variety $\bar \rmT$, with the action of $\rm\rmT \times \rmT$, is a non-trivial example of a {\it toric stack}, since the diagonal torus acts
trivially on $\bar \rmT$. Thus, perhaps surprisingly, the theory of toric stacks shows up in the analysis of the equivariant derived categories for toroidal group imbeddings.
\vskip .1cm
2. One may want to compare the above steps with their counterparts in the toric case. There, step 1 is absent and
steps 2 and 3 are combined into one step where one proves the result in Theorem ~\ref{mainthm.2.1.0} for projective toric
varieties: here none of the complications coming from the maps $p_1$ and $p_2$ occur in the toric case, as there is no need to consider $p_1$ or $p_2$ in this case.
Finally in the analogue of Step 4, one produces a dg-algebra $  \B_{\rmT}(\bar \rmT)$ from the complexes $R\pi_*(L_i)$,  and proves that $ \B_{\rmT}(\bar \rmT)$ is formal as a dga.
\end{remarks}
\vskip .2cm
The discussion of the proof of Theorem 1.2 in the body of the paper is also broken up, with each major step occupying a clearly marked subsection. See for example, subsection 2.3, which discusses Step 1
of the proof, subsection 2.4, which discusses Step 2, and subsection 2.5 which discusses Step 3. Finally the last Step (namely Step 4), occupies
all of section 3 and completes the proof of Theorem 1.2. Since this is a rather long section, we have further broken this up into Steps 4.1 through 4.3.
\vskip .1cm
In section 4, we discuss the following general result which sheds some light on the validity  of the conjecture that the dg-algebra $\B_{\rmG}(\rmX)$ is formal in general. 
We recall that a well-known theorem (see \cite{Ka80}) shows that there are a countable number of obstructions $m_i$, $i=3,4, \cdots,$ that
need to vanish for an $A_{\infty}$-dg-algebra to be formal. Implicit in the
 following theorem is the fact that we make use of the cohomology notation for intersection cohomology complexes, where
 we start with a local system $\L$ or the constant sheaf $\Q$ on the smooth part of a stratified variety: in \cite{BBD82}, they start with $\L[d]$ ($\Q[d]$, \res) where $d$ is the dimension of the
 open stratum of the stratified variety.
\begin{theorem}
 \label{mainthm.4}
(i) Suppose the $A_{\infty}$-dg-algebra $\B$ has the property that $H^i(\B) =0$ for all {\it odd } $i$. Then all the obstructions $m_i$, for
$i$ odd and $i \ge 3$ vanish.
\vskip .1cm
(ii) Suppose $\rmX$ is a projective $\rmG$-spherical variety for a connected reductive group $\rmG$ over a field of characteristic $0$.  Assume further
that $\rmX$ is a {\it simply-connected spherical variety} in the sense of \cite{BJ04}, that is, the stabilizers in $\rmG$ and $\rmB$ are connected at all points on $\rmX$, where
$\rmB$ is a Borel subgroup of $\rmG$. Then the conclusions of (i) hold for the dg-algebra $\B_{\rmG}(\rmX)$ considered in Theorem ~\ref{mainthm.1}. 
\end{theorem}
\vskip .2cm
We begin section 2, with a quick self-contained review of the simplicial model of equivariant derived categories in 2.1, so that the remainder of the paper
could be read effortlessly. Some of the more technical details are left to 
section 5, which is also devoted a detailed comparison of the two well-known models of equivariant derived categories.
We provide this comparison for the following two reasons:
\vskip .1cm
(i) The simplicial Borel construction sending a $\rmG$-space to $\EG\rmX$ produces a simplicial resolution of $\rmX$, which is
clearly functorial in $X$. Other geometric models of the Borel construction always involve the choice of such a resolution
and can often make the situation quite complicated. For example, in section 3, in the context of the analysis of the toroidal imbedding case, we run into
situations where it becomes necessary to relate the equivariant derived categories associated
to two groups, $\rmG$ and $ \rmH$, where one is provided with a surjective homomorphism $\rmG \ra  \rmH$.
This is quite difficult, and nearly impossible if one uses the geometric approach where one produces
approximations of classifying spaces by starting with representations of the groups. In contrast,
 this is totally effortless in the simplicial setting, so that this forces us to adopt the
simplicial methods for defining classifying spaces and associated equivariant derived categories.
\vskip .1cm
(ii) On the other hand, the geometric models of classifying spaces have certain advantages in that
their approximations are indeed varieties, and therefore, much of the machinery from the non-equivariant
framework adapts easily. In addition, the geometric models seem to be more popular
in the literature. Moreover, the conjecture we are considering in this paper is originally stated in terms
of the geometric models for classifying spaces: see \cite{Lu95}. \footnote{The reason we have left this discussion still in section 5, is that 
the results in this section are  not used anywhere
 else in the paper, except to show that the conjecture in question, which was originally stated in the framework of the geometric model of 
 equivariant derived
 categories carries over to the simplicial model of equivariant derived categories.}
\vskip .2cm
We provide this comparison in as general a context as possible so as to be of use in a wide variety of contexts, that is,
 for schemes of finite type over perfect fields $k$ of finite $\ell$-cohomological dimension for some prime 
$\ell \ne char (k)$. 
Here $\BG ^{\gm, {\it m}}$ is a degree-$m$ approximation to the classifying space for $\rmG$ and 
$\rmEG ^{\gm, {\it m}}$ denotes  the universal principal
$\rmG$-bundle over $\BG^{\gm, {\it m}}$. Recall this means $\rmU_{\it m}=\rmEG ^{\gm, {\it m}}$  is an open $\rmG$-stable
subvariety of a representation $\rmW_{\it m}$ of $\rmG$, so that (i) $\rmG$ acts freely on $\rmU_{\it m}$ and a geometric quotient
$\rmU_{\it m}/\rmG$ exists as a variety and
(ii) so that in the family $\{(\rmW_{\it m}, \rmU_{\it m})|{\it m} \eps {\mathbb N}\}$, the codimension of $\rmW_{\it m} - \rmU_{\it m}$ in $\rmW_{\it m}$ goes to $\infty$ as 
$m$ approaches $\infty$. (See Definition ~\ref{defn:Adm-Gad} for more precise details.)
 $\BG$ will denote a simplicial model for
the classifying space of $\rmG$ and $\rmEG$ will denote its universal principal $\rmG$-bundle.
\begin{theorem}
 \label{mainthm.5}
For each fixed $m \ge 0 $, we obtain  the diagram of simplicial varieties (where $p_1$ is induced by the
projection $\rmEG^{\gm, {\it m}} \times \rmX \ra \rmX$ and $p_2$ is induced by the projection 
$\rmEG \times (EG^{\gm, {\it m}} \times \rmX) \ra \rmEG^{\gm, {\it m}} \times \rmX$):
 \be \begin{equation}
  \xymatrix{&{\rmEG \times_{{\rmG}} (\rmEG^{\gm, {\it m}} \times \rmX)} \ar@<1ex>[ld]_{p_1} \ar@<-1ex>[rd]^{p_2}\\
{\rmEG \times_{{\rmG}} \rmX}   && {\rmEG^{\gm, {\it m}} {\underset {G} \times}\rmX}}
     \end{equation} \ee
\vskip .2cm \noindent
(i) For each finite interval $\rmI=[a, b]$ of the integers, with $2m-2 \ge b-a $, 
\[p_1^*:\rmD^{\rmI}_{\rmG}(\rmEG \times_{\rmG} \rmX) \ra \rmD^{\rmI}_{\rmG}(\rmEG \times_{\rmG} (\rmEG^{\gm, {\it m}} \times \rmX)) \mbox{ and }\]
\[p_2^*: \rmD^{\rmI}_{\rmG}(\rmEG ^{\gm, {\it m}} \times_{\rmG}\rmX) \ra \rmD^{\rmI}_{\rmG}(\rmEG\times_{\rmG} (\rmEG^{\gm, {\it m}} \times \rmX))\]
are equivalences of categories. (Here the superscript $I$ denotes the full subcategory of complexes
whose cohomology sheaves vanish outside of the interval $I$ and subscript $\rmG$ denotes the full subcategories of complexes
 whose cohomology sheaves are $\rmG$-equivariant.) Moreover, for complexes of $\bar \Q_{\ell}$-adic sheaves (where $\bar \Q_{\ell}$ denotes an algebraic closure of $\Q_{\ell}$), both the functors 
$p_1^*$ and $p_2^*$ send complexes that are mixed and pure to complexes that are mixed and pure. \footnote{Recall that a complex of $\bar \Q_{\ell}$-adic sheaves is
mixed and pure if it has a finite increasing filtration whose successive quotients have cohomology sheaves that are pure: see \cite[5.1.5]{BBD82}.}
There exists an equivalence of derived categories:
\[ \rmD^{b}_{\rmG}(\rmEG^{\gm, {\it m}} \times_{\rmG}\rmX) \simeq \rmD^b_{\rmG}(\rmEG \times_{\rmG}\rmX)\]
which is natural in $\rmX$ and $\rmG$. The above equivalences hold in all characteristics with the derived categories of complexes of $\ell$-adic sheaves
on the \'etale site and hold in characteristic $0$ with the derived categories of complexes of sheaves of $\Q$-vector spaces.
\vskip .1cm
(ii) Moreover, both the maps $p_i$, $i=1, 2$, induce isomorphisms on the fundamental groups completed away from the characteristic.
\end{theorem}
\vskip .2cm
We devote all of section 5 to a detailed discussion of equivariant derived categories leading up to the above theorem.
 The Appendix  provides a supplementary discussion of geometric classifying spaces, concluding with a result that shows how to transfer $t$-structures under
equivalences of triangulated categories.
\vskip .2cm \noindent
{\bf Conventions}. 
\begin{itemize}
\item In view of the reasons explained above, {\it the equivariant derived categories we consider will always be defined using the simplicial
 construction of the classifying spaces of linear algebraic groups}. The equivalence of derived categories
 stated in the last Theorem, then shows that all our results carry over to the equivariant derived categories defined by means of the procedures discussed
in \cite{BL94} or \cite{To99}. 
\item
The derived push-forwards we consider in various contexts in the paper, will almost 
always be defined using a functorial resolution such as that given by the Godement resolution. This
has the advantage of preserving any extra structure that is associated to a given sheaf. 
\item
A dga (short for  differential graded algebra or a sheaf of such algebras) will mean a 
chain complex $A^{\bullet}$, provided with an associative pairing $A^{\bullet}\otimes A^{\bullet} \ra A^{\bullet}$
of chain complexes. The dga will be  commutative if the above pairing is  graded commutative. (A chain complex will always have
differentials of degree $+1$.)
\end{itemize}
\vskip .2cm
 The approach we take is similar in principle to that of
 \cite{Lu95} (and \cite{Gu05}),  where  the formality of the dg-algebra $\B$ is proven for projective complex toric varieties (
 for complex complete symmetric varieties, \res). Both these need 
to  relate the equivariant
derived category $\rmD^b_{\rmG, c}(\rmX)$ with the category of sheaves of dg-modules over a sheaf of
dg-algebras on a {\it space} whose points correspond to all the $\rmG$-orbits on $\rmX$. 
\vskip .1cm
This intermediate equivalence, strongly needs the existence of {\it attractive slices} at each point
of each $\rmG$-orbit. Therefore, such a hypothesis is not always satisfied, in general. In the case of 
 toroidal group imbeddings, we show that
 these hypotheses are met to a large extent.
\vskip .1cm
We define {\it slices} in a somewhat more general context as
follows. 
\begin{definition}
\label{slice.def}
Let a linear algebraic group $\rmG$ act on a variety $\rmX$ and let $x \eps \rmX$. 
A locally closed subvariety ${\mathcal  S}$ of $\rmX$ containing $x$, stable under the action of the isotropy group $\rmG_x$ 
and satisfying the
following two conditions is called a {\it slice} at $x$: 
\vskip .1cm
(i) There exists a ${\mathbb G}_m$-action 
on ${\mathcal S}$ commuting with the action of $\rmG_x$.
\vskip .1cm
(ii) The map $\rmG{\underset {\rmG_{\it x}}  \times} {\mathcal  S} \ra \rmX$ sending $(g, x)$ to $g.x$ is 
an open immersion at $(e,x)$, and the dimension of ${\mathcal  S}$ is
the codimension of the orbit $\rmG\cdot x$ in $\rmX$. 
\vskip .2cm
Let ${\mathbb G}_m$  act on a variety $\rmX$ with a
fixed point $x$. Let $\lambda: {\mathbb G}_m \times \rmX\ra \rmX$ denote this action. 
We say $x$ is {\it attractive} if for all $y$
in a Zariski neighborhood of $x$, we have 
$lim_{t \to  0}ty=x$. Equivalently, all weights of
$\lambda$ acting on the Zariski tangent space at $x$ are contained in an
open half-space. 
\vskip .1cm
 Assume the situation of Definition ~\ref{slice.def}. Let $x \eps \rmX$ and let
${\mathcal  S}$ denote a slice at $x$. We say that ${\mathcal  S}$ is an 
{\it attractive slice}, if $x$ is an attractive fixed point for the
given action of ${\mathbb G}_m$ on ${\mathcal  S}$. (See \cite[appendix]{BJ01} for further details on
attractive fixed points.) 
\end{definition}
{\bf Acknowledgment}. The present paper had its origin in a question posed to the author by Michel Brion several years ago and it
evolved as a joint project between the author and Brion. In fact Brion has contributed substantially to the paper: for example,
Lemma ~\ref{phiK.acyclic.fibers} is entirely due to him. The author would like to thank Michel Brion for his considerable input into
the paper. He would also like to acknowledge the strong influence of \cite{Lu95}: in fact, as one can see, the main contribution of
the present paper is to show that the methods of \cite{Lu95} for toric varieties can be extended suitably to handle the class of all toroidal imbeddings
of connected complex reductive groups. The author also thanks both the referees for their detailed review, and constructive comments that have substantially improved
 the exposition.

\section{\bf Toroidal Group Imbeddings and the Proof of Theorem ~\ref{mainthm.2.1.0}}
As pointed out earlier, it is more convenient for us to make use of the simplicial model of the Borel construction and
the resulting equivariant derived categories. As this may not be that familiar, we begin this section with a quick review of
the simplicial model of the equivariant derived category. A detailed comparison with other geometric models of the
equivariant derived category appears in section ~\ref{equiv.dercat.3}.\footnote{We feel this approach would provide enough information of
the constructions used, so that a reader can easily follow the main arguments in the paper, without having to suffer through all the technical fine points.
The more technical aspects of the simplicial construction, along
 with a detailed comparison with other geometric constructions are left to section 5, which can be used as a reference, if the reader so wishes.}
\subsection{\bf The simplicial model of Equivariant Derived Categories}
\label{simpl.der.cat}
This is the simplicial model discussed in detail in \cite{De74},  and also in \cite[section 6]{Jo93} or \cite{Jo02}. The main 
advantage of this model comes from the functoriality of the simplicial Borel construction. In view of the various applications, we have decided to make the discussion in 
this section general enough
so that it applies to actions of linear algebraic groups defined over fields $k$ that are perfect and of finite $\ell$ cohomological dimension for some $\ell \ne char (k)$.
Therefore, all objects in the following discussion will be defined and of finite type over such a field $k$.
\vskip .1cm
Given a linear algebraic group
$\rmG$ acting on a variety $\rmX$, $\rmEG{\underset {G} \times}\rmX$ will now denote the simplicial variety defined by 
letting $(\rmEG{\underset {\rmG} \times}\rmX)_n = {\rmG}^{\times n} \times \rmX$ with the face maps $d_i:(\rmEG{\underset {\rmG} \times}\rmX)_n \ra (\rmEG{\underset {\rmG} \times}\rmX)_{n-1}$, $i=0, \cdots n$,
induced by the group action 
$\mu: \rmG \times \rmX \ra \rmX$,
the group multiplication $ \rmG \times \rmG \ra \rmG$ and the  projection $\rmG \times \rmX \ra \rmX$. The $i$-th degeneracy 
$s_i:(\rmEG{\underset {\rmG} \times}\rmX)_{n-1} \ra (\rmEG{\underset {\rmG} \times}\rmX)_n$, $i=0, \cdots, n-1$ is induced by inserting
 the identity element of the group $\rmG$ in the $i$-th place. This construction is functorial: if $f: \rmX \ra \rmY$ is any $\rmG$-equivariant map
  between varieties with $\rmG$-action, one obtains an induced map $\rmEG{\underset {G} \times}\rmX \ra \rmEG{\underset {G} \times}\rmY$.
 \vskip .1cm
 Given a Grothendieck topology, $\Top$, on varieties over $k$, one defines an induced
Grothendieck topology $\Top(\EGx \rmX )$ whose objects are $\rmU_n \ra (\EGx \rmX )_n$ in $\Top((\EGx \rmX )_n)$ for some $n \ge 0$. The maps between two such objects and
coverings for this topology are defined as in \cite{De74}. When one chooses the \'etale topology, this site will be denoted $Et(\EGx \rmX)$.
\vskip .2cm
\subsubsection{}
\label{sheaves.simp.1} 
Given a simplicial variety $\rmX_{\bullet}$ (for example, $\rmEG{\underset {G} \times}\rmX$),
a sheaf $F$ on the site  $\Top(\rmX_{\bullet})$ is a collection of
sheaves $\{F_m|m \ge 0\}$ with $F_m $ on the transcendental site (in case $k = \Cl$) or the \'etale site of $\rmX_m$, provided with structure maps $\alpha^*(F_m) \ra F_n$
for each structure map $\alpha: \rmX_n \ra \rmX_m$, and satisfying certain compatibility conditions: see \cite[(5.6.6)]{De74}.
We say that a sheaf $F$ has {\it descent} (or is {\it cartesian}) if the above maps $\alpha^*(F_m) \ra F_n$ are all isomorphisms. 
$\rmD(\EGx \rmX)$ will denote the derived category of complexes of sheaves of $\Q$-vector spaces on the simplicial variety $\EGx \rmX$ when everything is defined
over the complex numbers,  and will denote the derived category of $\ell$-adic sheaves on the \'etale site $Et(\EGx \rmX)$, in general.
In this framework, 
$\rm \rmD_{\rmG}(\rmX)$ will denote the full subcategory of $\rmD(\rmEG{\underset {G} \times}\rmX)$ consisting of complexes of sheaves
so that the cohomology sheaves  have descent. Moreover, for each finite  interval $\rmI=[a,b]$ with $-\infty < a \le b < \infty$, $\rmD^{\rmI}_{\rmG}(\rmX)$ will denote the
full subcategory of $\rmD_{\rmG}(\rmX)$ consisting of complexes $K$ for which ${\mathcal H}^i(K)=0$ for all $ i \notin I$.
\vskip .2cm
\begin{remark}
 It may be important to point out that, in order to construct the equivariant derived category, 
 one needs to begin with the category of all sheaves on the simplicial variety $\EGX$: that is the only way to ensure that the category of sheaves
 have important properties like having enough injectives. Then one restricts to the {\it full subcategory} of all complexes of sheaves whose cohomology sheaves are  equivariant (that is, cartesian) to
  obtain the equivariant derived category: the fact that the full subcategory of equivariant sheaves is closed under extensions is needed to ensure
  the equivariant derived category, so defined, is a triangulated category as shown in \cite[p. 38]{Hart66}. One needs to adopt such a construction in
  the geometric models of equivariant derived categories as well: see \cite[Chapter 1.8]{BL94}, for example. The only exception to this is when the group $\rmG$ is discrete:
   in this case, the category of equivariant sheaves have enough injectives as shown by Grothendieck (see \cite{Groth57}, \cite[Chapter 1.8]{BL94}): therefore, in this case, we may
   work with the category of complexes which are equivariant in each degree. We make use of this observation in {\it Step 4.1} in the proof of Theorem 
   ~\ref{mainthm.2}.
\end{remark}

\subsubsection{}
\label{sheaves.simp.2}
Observe that if $f_{\bullet}: \rmX_{\bullet} \ra \rmY_{\bullet}$ is a map of simplicial varieties, then the induced push-forward $f_{\bullet, *} = \{f_{n,*}|n \ge \}$,
which is not a single functor, but a collection of functors, indexed by $n \ge 0$. This issue is rather technical, and becomes relevant only where
it is important to compute the cohomology of the fibers of the simplicial map $f_{\bullet}$ as a simplicial scheme. The solution is to use the 
simplicial topology as
 in  \cite{Jo02}, and we invoke that in section ~\ref{simpl.der.functors}, as well as in 
the proof of Theorem ~\ref{mainthm.5} in section 5. 

\begin{terminology}
 We will adopt the following terminology throughout the rest of the paper. If $\rmG$ is a linear algebraic group acting on a variety $\rmX$,
$\EGx \rmX$ will always denote the simplicial variety defined above. In particular, $\BG$ will denote the corresponding simplicial variety
when $\rmX = Spec \, k$. The geometric model for $\EGx \rmX$ considered in  ~\eqref{eqdercat.1}, {\it which is an ind-scheme}, will always be denoted $\{\rmEG^{\gm, {\it m}}{\underset {G} \times}\rmX|m\}$.
\end{terminology}

\subsection{\bf Toroidal imbeddings of connected reductive groups: basic definitions and examples}
Throughout the rest of this section we will assume the base field is the field of complex numbers.
\label{tor.grp.imbed}
Let $\rmG$ denote a connected reductive group. Viewing $\rmG$ as a $\rmG \times \rmG$-homogeneous space for the
action of $\rmG \times \rmG$ by left and right multiplication, $\rmG \simeq (\rmG \times \rmG)/ \diag(\rmG)$. (Here
 $\diag(\rmG)$ denotes the group $\rmG$ imbedded diagonally in $\rmG \times \rmG$.)
\vskip .2cm
{\it $\rmG$-spherical varieties} may be defined as normal varieties $\rmX$, equipped with a $\rmG$-action so that 
a Borel subgroup $\rmB$ has an open dense $\rmB$-orbit. It is known (see \cite[Remark 2.2]{Kn91}) that then
$\rmX$ contains only finitely many $\rmB$-orbits as well as $\rmG$-orbits. Moreover, such a spherical variety $\rmX$
may be viewed as a partial compactification of the homogeneous space $\rmG/\rmH$, which denotes the open $\rmG$-orbit on $\rmX$.
When the group $\rmG$ is replaced by $\rmG \times \rmG$ and $\rmH$ by the diagonal imbedding of $\rmG$ in $\rmG \times \rmG$
(with the diagonal $\rmG$ acting on $\rmG \times \rmG$ by both left and right-multiplication), we obtain {\it spherical 
imbeddings of the group $\rmG$}.
\vskip .2cm
Choose a maximal torus $\rmT$ of $\rmG$, and denote by $\rmN$ its normalizer in $\rmG$; the quotient
$\rmN/\rmT$ is the Weyl group $\rmW$. Let $\rmB$ denote a fixed Borel subgroup containing $\rmT$.
$\rmG$-spherical varieties are classified by {\it colored} fans (that is,  fans with the extra structure of colors) in the valuation cone associated to $\rmG/\rmH$ (see \cite{Kn91}).
The {\it colors} of $\rmX$ correspond to $\rmB$-stable prime divisors in $\rmX$ that are  not $\rmG$-stable, and contain a $\rmG$-orbit. Equivalently, the colors correspond to
the closures in $\rmX$ of the $\rmB$-stable prime divisors in the open $\rmG$-orbit, so that the closure in $\rmX$ contains a $\rmG$-orbit. 
Spherical varieties that are {\it $\rmG\times \rmG$-equivariant imbeddings of $\rmG$} are classified by colored fans in $\rmX_*(\rmT) \otimes {\mathbb R}$ with support in the 
negative Weyl chamber: see \cite[6.2.4 Proposition]{BK05}. (Here $\rmX_*(\rmT)$ denotes the weight-lattice.) 

\vskip .1cm
{\it Toroidal imbeddings} form an important  special class
of $\rmG$-spherical varieties, defined as follows. 
\begin{definition} (Toroidal imbeddings) 
An equivariant imbedding of the connected reductive group $\rmG$ is {\it toroidal}, precisely when there are no colors, that is, every $\rmB$-stable prime divisor in $\rmX$ is either $\rmG$-stable or
does not contain a $\rmG$-orbit. 
Toroidal imbeddings of the group $\rmG$ are classified by fans in the
negative Weyl chamber: see \cite[6.2.4 Proposition]{BK05}.
\end{definition}

Therefore, given a $\rmG \times \rmG$ equivariant imbedding $\rmX$ of $\rmG$, one may find a 
$\rmG \times \rmG$ equivariant toroidal imbedding $\tilde {\rm X}$ provided with a birational $\rmG \times \rmG$-equivariant map $\tilde \rmX \ra \rmX$, by replacing the given colored
fan with the fan obtained by removing all the colors.

Moreover, such a toroidal imbedding is {\it complete}, if and only if the corresponding fan has support the
whole Weyl chamber. By considering $\rmW$-translates of the above fans, one obtains a $\rmW$-invariant fan in $\rmX_*(\rmT) \otimes {\mathbb R}$. Then a toroidal imbedding will be {\it smooth} if and only if the
corresponding $\rmW$-invariant fan has the property that every cone is generated by part of a basis of the free abelian group 
$\rmX_*(\rmT)$.
 
 Before proceeding further, we will next discuss a few examples of
toroidal group imbeddings.
\begin{examples}
\begin{enumerate}[\rm(i)]
\item All toric varieties are toroidal imbeddings of the open dense torus.
 \item The simplest example of group imbeddings, other than toric varieties, are that of $\rmGL_2$. In this case there are at least $4$ distinct 
 $\rmGL_2 \times \rmGL_2$-equivariant imbeddings of $\rmGL_2$, namely, ${\mathbb A}^4$, ${\mathbb P}^4$, and
 $\widetilde {\mathbb A}^4$, $\widetilde {\mathbb P}^4$, where the last two are blow-ups of the first two at the
 origin in ${\mathbb A}^4$. Only the latter two are toroidal imbeddings, and clearly  $\widetilde {\mathbb P}^4$
 is the only toroidal imbedding that is projective. (See \cite[8.2]{AKP}.)
 \item Let $\rmG$ be  semi-simple of adjoint type. Then $\rmX_*(\rmT)$ has
a basis consisting of the fundamental weights. Therefore, the fan consisting
of the Weyl chambers and their faces is smooth. The corresponding
(smooth) toroidal imbedding is the {\it wonderful compactification} of $\rmG$. (See \cite[section 6.1]{BK05}.)
In fact, an alternate definition of a toroidal group imbedding of the reductive group $\rmG$ is as a group imbedding $\rmX$ of $\rmG$,
 so that the quotient map $\rmG \ra \rmG_{ad}$ (which denotes the adjoint group) extends to a morphism of schemes
 $\rmX $ to the wonderful compactification of $\rmG_{ad}$: see \cite[Definition 6.2.2]{BK05}.
\item 
If $\rmG$ is semi-simple but no longer of adjoint type, then we may consider
the same fan, but now it is almost never smooth. For example, if
$\rmG = \rmSL_3$ then we get a singular toroidal compactification of $\rmG$. It can
be constructed geometrically as the normalization of the wonderful
compactification of $\rmPGL_3$ in the function field $k(\rmSL_3)$. This
construction works more generally for any semi-simple $\rmG$: the normalization
in $k(\rmG)$ of the wonderful compactification of the adjoint group yields
a canonical imbedding of $\rmG$, which is toroidal and projective but (again)
almost never smooth. (See \cite[section 6.2.A]{BK05} for additional details.)
\end{enumerate}
\end{examples}
For the rest of the paper, we will adopt the terminology introduced in the outline of the proof of Theorem ~\ref{mainthm.2}
in the introduction.
\vskip .1cm
\subsection{\bf Step 1 of the proof of Theorem ~\ref{mainthm.2}.  }
Let $\bar \rmG$ denote a normal $\rmG \times \rmG$-variety which contains $\rmG$ as an open orbit. We will
assume that $\bar{\rmG}$ is a toroidal imbedding of $\rmG$. 
Note that the normalizer (or stabilizer) of $\diagT$ in $\rmG \times \rmG$ 
equals $(\rmT \times \rmT) \diag (\rmN)$; moreover, the centralizer of $\diagT$ in $\rmG \times \rmG$ equals
$\rmT \times \rmT$, since $\rmT$ is its own centralizer in $\rmG$. We also have the exact sequence
\[ 1 \longrightarrow \diagT \longrightarrow (\rmT \times \rmT) \diag(\rmN) \longrightarrow \rmWT
\longrightarrow 1, \]
where $\rmWT$ denotes the semi-direct product of $\rmT$ with $\rmW$ (acting naturally on $\rmT$).

Let $\bar \rmT$ denote the closure of $\rmT$ in $\bar {\rmG}$; then $(\rmT \times \rmT)\diag (\rmN)$ acts on 
$\bar \rmT$ via its quotient $\rmWT$, where $\rmW$ acts by conjugation and $\rmT$ by left multiplication. 
This yields a (restriction) functor
\label{toroidal.1}
\be \begin{equation}
 \label{eq.toroidal.0}
res : \rmD_{\rmG \times \rmG, c}^b(\bar \rmG) \longrightarrow  \rmD_{(\rmT \times \rmT)\diag (\rmN) ,c}^b(\bar \rmT).
\end{equation} \ee
The first main result of this section is the following. 
\begin{theorem}
\label{eq.toroidal.1}
(i) The above functor is fully-faithful. 
\vskip .2cm \noindent
(ii) Moreover, if ${ \rmD}_{(\rmT \times \rmT)\diag (\rmN) ,c}^{{\rm b},o}(\bar \rmT)$ 
denotes the full subcategory of $\rmD_{(\rmT \times \rmT)\diag (\rmN) ,c}^b(\bar \rmT)$
generated by the  $(\rmT \times \rmT)\diag (\rmN)$-equivariant sheaves that are constant along the 
orbits of $(\rmT \times \rmT)\diag (\rmN) $ on $\bar \rmT$, then the above functor  induces an equivalence 
\[\rmD_{\rmG \times \rmG, c}^b(\bar \rmG) {\overset {res} \ra}  { \rmD}_{(\rmT \times \rmT) \diag(N) ,c}^{b,o}(\bar \rmT). \]
\vskip .2cm \noindent
(iii) It sends the $\rmG \times \rmG$-equivariant intersection cohomology complex $\rmIC^{\rmG \times \rmG}(\Q_{\O})$ on
a $\rmG \times \rmG$-orbit $\O$ to the corresponding $(\rmT \times \rmT)\diag (\rmN)$-equivariant intersection cohomology
complex $\rmIC^{(\rmT \times \rmT)\diag (\rmN)}(\Q_{\O'})$ where $\O'$ is the $(\rmT \times \rmT)\diag (\rmN)$-orbit corresponding to
the $\rmG \times \rmG$-orbit $\O$. 
\end{theorem}

\begin{proof}
Observe that 
\[ \rmD_{\rmG \times \rmG, c}^b(\bar \rmG) =
\rmD_{\rmG \times \rmG, c}^b(EG \times E G {\underset {\rmG \times \rmG} \times }{\bar \rmG}) \mbox{ while}, \]
\[ \rmD_{(\rmT \times \rmT)\diag (\rmN), c}^b(\bar \rmT) = 
\rmD_{\rmG \times \rmG, c}(E G \times E G {\underset {\rmG \times \rmG} \times}
(\rmG \times \rmG{\underset {(\rmT \times \rmT )\diag (\rmN)} \times} (\bar {\rmT}))). \]
Therefore, the map denoted $res$ is the pull-back map $\phi^*$ induced by the map 
\[\phi : \rmEG \times \rmEG {\underset {\rmG \times \rmG} \times}
(\rmG \times \rmG{\underset {(\rmT \times \rmT )\diag (\rmN)} \times}(\bar \rmT)) \ra 
EG \times E G {\underset {\rmG \times \rmG} \times} \bar \rmG,\]
where $\phi$ itself is induced by the natural map $\rmG \times \rmG {\underset {(\rmT \times \rmT) \diag (\rmN)} \times} \bar {\rmT} \to \bar{\rmG}$. 
\vskip .2cm
In order to prove $\phi^*$ is fully-faithful, it is enough to show that the natural map 
$M \ra R\phi_* \phi^*(M)$ is a quasi-isomorphism for all $M \eps \rmD_{\rmG \times \rmG, c}^b(\bar \rmG)$. 
Unfortunately, the map $\phi$ is neither proper nor smooth, which makes it difficult to prove 
this directly. Therefore, we adopt a rather indirect technique. We first replace all the linear algebraic groups 
by their maximal compact subgroups. We may choose a maximal compact subgroup $\rmK \subset \rmG$ such that
$\rmN_{\rmK} := \rmK \cap \rmN$ is a maximal compact subgroup of $\rmN$; then $\rmT_{\rmK} := \rmT \cap \rmK$ is the largest
compact sub-torus of $\rmT$. Lemma ~\ref{max.compact} below shows then that it suffices to prove that 
the induced restriction 
\[ res_{\rmK} : \rmD_{\rmK \times \rmK, c}^b(\bar \rmG) \longrightarrow 
\rmD_{(T_K \times T_K) diag(N_K),c}^b(\bar \rmT) \]
is fully-faithful. This functor identifies with $\phi_{\rmK}^*$ where $\phi_{\rmK}$ is the natural map
\be \begin{equation}
\label{phi.K}
\rmEK \times \rmEK {\underset {\rmK \times \rmK} \times}
(\rmK \times \rmK{\underset {(\rmT_{\rmK} \times \rmT_{\rmK}) \diag (\rmN_{\rmK})} \times}(\bar \rmT)) \ra 
\rmEK \times \rmEK {\underset {\rmK \times \rmK} \times} \bar{\rmG}.
\end{equation} \ee
Observe that $\phi_{\rmK}$ is proper. Therefore, one has a projection formula 
which provides the (natural) identification 
\[ R\phi_{K*}\phi_K^*(M) \simeq M \otimes R\phi_{K*}\phi_K^*({\underline \Q}), M 
\eps \rmD_{(T_K \times T_K ) \diag(N_K),c}^b(\bar \rmT).\]
Therefore, it suffices to prove that 
\be \begin{equation}
 \label{phiK}
 R\phi_{K*}({\underline \Q}) = R\phi_{K*}(\phi_K^*{\underline \Q}){\overset {\simeq} \ra} {\underline \Q}. 
\end{equation} \ee
\vskip .2cm
Secondly, the properness of $\phi_{\rmK}$ shows that one has proper-base-change, so that 
it suffices to prove that $\phi_{\rmK}$ is surjective and its fibers are $\Q$-acyclic and connected: observe that this will prove ~\eqref{phiK} and, 
 making use of the adjunction between $\phi_{\rmK}^*$ and $R\phi_{\rmK*}$, that the functor $\phi_{\rmK}^*$ is fully-faithful. 
This is worked out in Lemma ~\ref{phiK.acyclic.fibers} below, which will complete the proof that the functor $\phi^*$ is fully-faithful.
\vskip .2cm
 To prove the second statement, 
it suffices to show that $\phi^*$ induces an equivalence on the corresponding hearts,
that is, at the level of the equivariant sheaves. First observe that $\bar{\rmG}$ is an scs variety (that is, a spherical simply connected variety)
in the sense of \cite[section 1]{BJ04}. Thus the 
 isotropy subgroup of any $\rmG \times \rmG$-orbit $\O \subset \bar{\rmG}$ is connected, by [loc. cit., Lemma 3.6], 
which implies that all the $\rmG \times \rmG$-equivariant local systems on $\O$ are constant.
Any constructible $\rmG \times \rmG$-equivariant sheaf on $\bar \rmG$ has a finite filtration by 
$\rmG \times \rmG$-equivariant sheaves that are  extension by zero of the constant sheaves 
on the $\rmG \times \rmG$-orbits. Also, by \cite[Proposition 6.2.3]{BK05}, every such orbit $\O$ intersects 
$\bar {\rmT}$ along a unique orbit of $(\rmT \times \rmT) \diag (\rmN)$, and this sets up a bijection
between $\rmG \times \rmG$-orbits in $\bar{\rmG}$ and $(\rmT \times \rmT) \diag (\rmN)$-orbits in $\bar {\rmT}$,
which preserves the closure relations. 
This completes the proof of the second statement. (Observe that the stabilizers of the 
$(\rmT \times \rmT ) \diag (\rmN)$-orbits in $\bar \rmT$ are not connected in general: for example,
the stabilizer of the open orbit is $\diag (\rmN)$. Thus, one obtains 
$(\rmT \times \rmT ) \diag (\rmN)$-equivariant sheaves on these orbits that are not constant.
This is the need to restrict to the full subcategory 
$ \rmD_{(\rmT \times \rmT) \diag(N) ,c}^{{\rm b},o}(\bar \rmT)$ in order for the restriction functor
 to be an equivalence.)
\vskip .2cm
The third statement follows from the local structure of toroidal group imbeddings: see \cite[Theorem 29.1]{Ti11}. Here are the details to prove it. Let $\O'$ denote a $(\rmT \times \rmT) \diag (\rmN)$-orbit 
on $\bar \rmT$ and let $\O' \ra \bar \O'$ denote the corresponding open immersion of $\O'$  into its closure. We stratify $\bar \O'$ so that one obtains a sequence of
 open immersions $\rmU_0 = \O' {\overset {j_0'} \ra} \rmU_1' {\overset {j_1'} \ra} \rmU_2' \cdots \rmU_{n-1}'{\overset {j_{n-1}'} \ra} \rmU_n' = \bar O'$,
  so that each of the strata $\rmU_i - \rmU_{i-1}$ is a (disjoint) union of  $(\rmT \times \rmT)\diag (\rmN)$-orbits. Let
$j': \rmU _i \ra \rmU _{i+1}$ denote an open immersion of
$(\rmT \times \rmT)\diag (\rmN)$-stable subvarieties appearing in the above factorization of $\O' \ra \bar \O'$. 
Let $j:\rmV _i \ra \rmV _{i+1}$ denote the open immersion of the corresponding $\rmG \times \rmG$-stable subvarieties in the
factorization of $\O \ra \bar \O$, where $\O$ is the $\rmG \times \rmG$-orbit corresponding to $\O'$.  Recalling that the equivariant
 intersection cohomology complexes are  suitable perverse extensions,  now
it suffices to show that $\phi^* ( \sigma _{\le n} Rj_* (L)) \simeq \sigma_{\le n} Rj'_* ( \phi^*(L))$ for
any constant sheaf $L$ on $\rmV_i$, where $\sigma_{\le n}$ denotes the cohomology
truncation that kills cohomology sheaves in degrees greater than $n$. 
\vskip .2cm
For a given point $x \eps \rmV_{i+1}$, the local structure in \cite[Theorem 29.1]{Ti11} shows
that, after possibly replacing $x$ by a translate of $x$ by an element $(g_1, g_2) \eps \rmG \times \rmG$, one may take the intersection with an open neighborhood $\rmV_{\it x}$ so that $j_{|\rmV_{\it x} \cap \rmV_i}: \rmV_{\it x} \cap \rmV_i \ra \rmV_{\it x} \cap \rmV_{i+1}$
identifies with ${\rm id}_{\rm {R_u(B) \times R_u(B^-)}} \times j': {\rm {R_u(B) \times R_u(B^-)}} \times \rmU_i \ra {\rm {R_u(B) \times R_u(B^-)}} \times \rmU_{i+1}$. 
(Here  $\rmB$ is a Borel subgroup of $\rmG$ containing $\rmT$ and
$\rmB^-$ is its opposite Borel subgroup also containing $\rmT$. Then  $ {\rm R_u}(\rmB) \times {\rm R_u}(\rmB^-)$ denotes the product of the unipotent radicals of $\rmB$ and $\rmB^-$.)
Now $\phi^*$ corresponds to pull-back by the inclusion of $\rmU_i$ in ${\rm {R_u(B) \times R_u(B^-)}} \times \rmU_i$ so that $\phi^*(L)$ is the pull-back of
 the sheaf $L$ on $\rmV_{\it x} \cap \rmV_i = {\rm {R_u(B) \times R_u(B^-)}} \times \rmU_i$ to $\rmU_i$. Therefore, on this open neighborhood $\rmV_{\it x}$, $\sigma_{\le n}(Rj_*(L)) =
id\boxtimes \sigma_{\le n}Rj'_*(\phi^*(L))$ and therefore, $\phi^*(\sigma_{\le n}Rj_*(L)) = \sigma_{\le n} Rj'_*(\phi^*(L))$.
 \end{proof}
\begin{lemma}
 \label{max.compact}
 Let $\rmG$ denote a connected reductive group acting on a variety $\rmX$. Let $\rmK$ denote
a maximal compact subgroup of $\rmG$. Then the restriction functor:
\[\rmD^b_{\rmG, c}(\rmX) \ra \rmD^b_{\rmK, c}(\rmX)\]
is fully-faithful.
\end{lemma}
\begin{proof} We consider the map 
$\psi: \rmEG{\underset {\rmG} \times}(\rmG {\underset {\rmK} \times}\rmX) \ra \rmEG{\underset {\rmG} \times}\rmX$ induced by the
left-action of $\rmG$ on $\rmX$. Observe that the functor $\psi^*$ identifies with the restriction functor considered in the lemma.
Using the simplicial models for the Borel construction, one sees that the above map in simplicial degree $n$ is
the map $\psi_n:\rmG^n \times \rmG{\underset {\rmK} \times}\rmX \ra \rmG^n \times \rmX$. This is induced by the map 
$\psi_0:\rmG{\underset {\rmK} \times}\rmX \ra  \rmX$, which being locally trivial in the complex topology, is cohomologically proper.
One may verify therefore that each $\psi_n$  is cohomologically proper, so that proper-base-change and 
the projection formula hold and that therefore, it suffices to verify that the fibers of 
$\psi_0$ are acyclic with $\Q$-coefficients. But the fibers of $\psi_0$ identify with $\rmG/\rmK$, 
which is acyclic with rational coefficients. Now an argument exactly as in the case of the functor $\phi_{\rmK}^*$ in the last Theorem, proves
 that the functor $\psi^*$ is fully-faithful. 
\end{proof}
\begin{lemma}
\label{phiK.acyclic.fibers}
The fibers of $\phi_{\rmK}$ are acyclic with $\Q$-coefficients, where $\phi_{\rmK}$ is the map in ~\eqref{phi.K}. Moreover, $\phi_{\rmK}$ 
restricts to an isomorphism over an open dense subset of $\rmG$ ($\subseteq \bar \rmG$) for the
complex topology.
\end{lemma}
\begin{proof}
Recall the Cartan decomposition $\rmG = \rmK \rmT \rmK$; also, $\rmT = T_{\rmK} T_r$, where 
$\rmT_r :=  \exp(i \Lie (\rmT_{\rmK}))$ denotes the real part of $\rmT$. Thus, $\rmG= {\rmK} \rmT_r {\rmK}$. 
We will first consider fibers of $\phi_{\rmK}$ on the open orbit, that is, for the map
\[\phi_{\rmK}^o:({\rmK} \times {\rmK} ){\underset {(\rmT_{\rmK} \times \rmT_{\rmK} ).\diag (\rmN_{\rmK})} \times}\rmT \ra \rmG.\]
Since this map is ${\rmK} \times {\rmK}$-equivariant, it suffices to consider the fiber at a point 
$t \eps \rmT_r$. We claim that this fiber is isomorphic to ${\rmK}^t/\rmN_{\rmK}^t$, 
where ${\rmK}^t$ (resp. $\rmN_{\rmK}^t$) denotes the centralizer of $t$ in ${\rmK}$ (resp. $\rmN_{\rmK}$). 

\vskip .2cm
Let $x,y \in {\rmK}$ and $z \in \rmT$ such that $x z y^{-1} = t$. Write 
$z = z_{\rmK} z_r \in \rmT_{\rmK} \rmT_r = \rmT$. Using the action of 
$1 \times \rmT_{\rmK} \subset (\rmT_{\rmK} \times \rmT_{\rmK}) \diag (N_{\rmK})$, we may assume that
$z_{\rmK} = 1$, that is, $z = z_r$. Then $t = (x y^{-1})(y z y^{-1})$ and $x y^{-1} \in {\rmK}$, 
$y z y^{-1} \in \exp(i \Lie ({\rmK}))$. By the uniqueness in the decomposition
$\rmG = {\rmK} \exp(i \Lie ({\rmK}))$, we obtain $x y^{-1} = 1$ and $t = y z y^{-1}$.
So $y = x$, and $z$ is conjugate to $t$ in ${\rmK}$. As 
$z,t \in \rmT_r$, they are conjugate in $\rmN_{\rmK}$, that is, there exists $n \in 
\rmN_{\rmK}$
such that $t = n z n^{-1}$. Then $t = n x^{-1} t x n^{-1}$, that is,  $x \in {\rmK}^t n$. 
We may replace $(x,y,z)$ with $(xn^{-1},yn^{-1}, n z n^{-1})$ by using the action
of $\diag (\rmN_{\rmK})$, to get $(xn^{-1}, xn^{-1},t)$, where $xn^{-1} \in {\rmK}^t$ is unique
up to multiplication by $N_{\rmK}^t$. As 
$(\rmT_{\rmK} \times \rmT_{\rmK}) \diag( \rmN_{\rmK}) = (1 \times \rmT_{\rmK}) \diag( \rmN_{\rmK})$, this yields the claim.
 
\vskip .2cm
Since ${\rmK}^t$ is a compact connected Lie group with maximal compact torus $\rmT_{\rmK}$, and $\rmN_{\rmK}^t$ is 
its normalizer in ${\rmK}^t$, the claim implies the fibers over $\rmG$ are $\Q$-acyclic. Moreover, 
the fiber at any point $t \in \rmT$ with regular real part $t_r$ just consists of this point. 

 \vskip .2cm
Next we consider the fibers of $\phi_{\rmK}$ over an arbitrary $\rmG \times \rmG$-orbit $\O$ 
in $\bar \rmG$. It suffices to show that any such fiber is again isomorphic to ${\rmK}^t/\rmN_{\rmK}^t$ 
for some $t \in \rmT_r$. For this, we recall the structure of $\O$ (see \cite[(5.1.2), (5.1.3)]{BJ04} which
discusses projective reductive varieties, which includes group imbeddings):
there exists a Levi subgroup $\rmL$ of $\rmG$ containing $\rmT$ such that
$\rmL_{\rmK} := \rmL \cap {\rmK}$ is a maximal compact subgroup of $\rmL$, and 
\[\O = ({\rmK} \times {\rmK}) {\underset {\rmL_{\rmK} \times \rmL_{\rmK}} \times} \rmL/\rmL_{\O}, \]
where $\rmL_{\O} \subseteq \rmL$ is a central subgroup. Moreover 
$\bar \rmT \cap \O = \diag(\rmN_{\rmK}){\underset {({\rm {diag}}(\rmN_{\rmK}) \cap \rmL)} \times} \rmT/\rmL_{\O} $ 
(note that $\rmL_{\O} \subset \rmT$ as $\rmL_{\O}$ is central in $\rmL$). So we obtain 
an isomorphism
\[{({\rmK} \times {\rmK}){\underset {(\rmT_{\rmK} \times \rmT_{\rmK} ) diag(N_{\rmK})} \times}(\bar \rmT \cap \O)} 
\cong {({\rmK} \times {\rmK}){\underset {(\rmT_{\rmK} \times \rmT_{\rmK}) diag(\rmN_{\rmK} \cap \rmL)} \times} \rmT /\rmL_ {\O}}\]
and a commutative triangle
\[\xymatrix{{({\rmK} \times {\rmK}){\underset {(\rmT_{\rmK} \times \rmT_{\rmK} ) diag(\rmN_{\rmK}))} \times}(\bar \rmT \cap \O)} 
\ar@<1ex>[r] \ar@<1ex>[dr] & {\O = ({\rmK} \times {\rmK}){\underset {\rmL_{\rmK} \times \rmL_{\rmK} } \times} \rmL/\rmL_{\O}} 
\ar@<1ex>[d]\\
&{({\rmK} \times {\rmK})/{(\rmL_{\rmK} \times \rmL_{\rmK} )}}}
\] 
where the horizontal arrow is the pull-back of $\varphi_{\rmK}$ over $\O$. 
This identifies the fibers of $\phi_{\rmK}$ over the orbit $\O$ with ${\rmK} \times {\rmK}$-translates of
the fibers of the map
\[(\rmL_{\rmK} \times \rmL_{\rmK}){\underset {(\rmT_{\rmK} \times \rmT_{\rmK} ) diag(\rmN_{\rmK}))} \times}( \rmT /\rmL_{ \O}) \ra \rmL/\rmL_{\O}.\]
Therefore, this reduces to the case of the fibers over the open orbit with the group $\rmG$ replaced 
by $\rmL$. 
\end{proof}

\vskip .2cm
In the remainder of this section, we will need to consider the derived direct images of maps between simplicial varieties
of the form $\EGx\rmX \ra \EGx\rmY$ where $\rmX$ and $\rmY$ are varieties provided with the action of a linear algebraic group $\rmG$ and
$f:\rmX \ra \rmY$ is a $\rmG$-equivariant map. Most of the time, this can be handled as in section ~\ref{sheaves.simp.2}, but occasionally it
is helpful to make use of the methods of \cite{Jo02}, which enables one to consider a single derived functor $Rf_*$.  Since this is technical issue,
we prefer to discuss it later in ~\ref{simpl.der.functors}, rather than here.
\vskip .2cm
\subsection{\bf Step 2 of the proof of Theorem ~\ref{mainthm.2}.  }
Throughout the rest of the paper, we will let 
\be \begin{equation}
\label{tildeN}
\tilde \rmN := (\rmT \times \rmT) \diagNT.  
\end{equation} \ee
Recall that 
$\tilde \rmN$ acts on $\bar {\rmT}$ via its quotient $\tilde{ \rmN }/{\rm {diag}}(\rmT) \simeq \rmWT$. Let
\be \begin{equation}
     \label{psi}
\psi: \rmE\tilde \rmN{\underset {\tilde \rmN} \times} \bar \rmT \ra \rmE \rmWT {\underset {\rmWT} \times}\bar \rmT
    \end{equation} \ee
denote the map induced by the identity on $\bar \rmT$ and the quotient map $\tilde \rmN \ra \rmWT$. 
Since ${\rm {diag}}(\rmT)$ 
acts trivially on $\bar \rmT$, the fibers of the simplicial map $\psi$ at every point can be identified with $\rmB {\rm {diag}}(\rmT)$. Recall ${\rmD}^{b,o}_{{\tilde \rmN},c}(\bar \rmT )$ denotes 
the full subcategory of the derived category ${\rmD}^{b}_{{\tilde \rmN}, c}(\bar \rmT )$ generated by 
the constant sheaves on each orbit of $\tilde \rmN$ on $\bar \rmT$.  Let $\rmG^{\bullet}$ denote the canonical Godement resolution. 
Then, clearly the functor $R\psi_*=\psi_*\rmG^{\bullet}$ sends complexes in ${\rmD}^{b,o}_{\tilde \rmN, c}(\bar \rmT )$ 
to complexes of dg-modules over the sheaf of dg-algebras
$R\psi_*(\Q)$, that is, to objects in the derived category $\rmD_{\rmWT}(\bar \rmT, R\psi_*(\Q))$. 
We let ${\rmD}^{+,o}_{\rmWT, c}(\bar \rmT, R\psi_*(Q))$ denote the full subcategory generated 
by the objects $R\psi_*(j_!\Q)$ where $j:\O \ra \bar \rmT$ denotes the immersion associated 
with an $\tilde \rmN$-orbit and we vary over such $\tilde \rmN$-orbits. Let 
\[L\psi^*: { \rmD}^{+,o}_{\rmWT,c}(\bar \rmT,R\psi_*(\Q)) \ra {\rmD}^{b,o}_{\tilde \rmN, _c}(\bar \rmT)\]
denote the functor defined by sending a dg-module $M$ to 
$\Q{\overset L {\underset {\psi^{-1}R\psi_*(\Q)} \otimes}} \psi^{-1}(M)$.
\vskip .1cm
\begin{remark}
\label{t.struct.remark1}
Observe that the derived category ${\rmD}^{b,o}_{\tilde \rmN, c}(\bar \rmT)$ has the following $t$-structures: 
 either the standard one, where the heart is the full subcategory of complexes whose cohomology sheaves are $\tilde \rmN$-equivariant and vanish in all degrees except  
$0$ (see, for example, \cite[1.3.2]{BBD82}), or the $t$-structure obtained by gluing  as in \cite[2.2]{BBD82}, where the heart is the full subcategory of $\tilde \rmN$-equivariant perverse sheaves.
\end{remark}
\begin{proposition}
\label{Rpsi*}
(i) The functor $R\psi_*: {\rmD}^{b,o}_{\tilde \rmN, c}(\bar \rmT) \ra {\rmD}^{+,o}_{\rmWT, c}(\bar \rmT, R\psi_*(\Q))$ 
is an equivalence of categories with inverse $L\psi^*$. (The derived category on the right is generated by the 
 dg-modules over the dga $R\psi_*(\Q)$ 
of the form $j_!j^*R\psi_*(\Q) = R\psi_*j_! j^*(\Q)$, as $j: \O \ra \bar \rmT$ varies over the 
$\tilde \rmN$-orbits on $\bar \rmT$.)
\vskip .1cm
(ii) As a consequence, the derived category ${\rmD}^{+,o}_{\rmWT, c}(\bar \rmT, R\psi_*(\Q))$ obtains
induced $t$-structures, induced from the $t$-structures on ${\rmD}^{b,o}_{\tilde \rmN, c}(\bar \rmT)$ by the functor $R\psi_*$.

\end{proposition}
\begin{proof}
 A key observation is that if $j: \O \ra \bar \rmT$ is the locally-closed immersion associated to an $\tilde \rmN$-orbit on $\bar \rmT$, then
$R\psi_*(j_!j^*(\Q)) = j_!j^*(R\psi_*(\Q))$. This follows readily since the map $\psi$ is a fibration with 
fibers $\rmB \diagT$ and one makes use of the simplicial model of classifying spaces. In the simplicial setting, the
map $\psi_n: (\rmE\tilde \rmN {\underset {\tilde \rmN} \times} \bar \rmT)_n \ra (\rmE \rmWT{\underset {\rmWT} \times} \bar \rmT)_n$ is of the
form $p_n \times id$, where $p_n: (\tilde \rmN)^{n-1} \ra (\rmWT)^{n-1}$ is the surjection induced by the 
surjection $\tilde \rmN \ra \rmWT$. Moreover, if $j_n: (\rmE\tilde \rmN {\underset {\tilde \rmN} \times} \O)_n =(\tilde \rmN)^{n-1} \times \O\ra
(\rmE\tilde \rmN {\underset {\tilde \rmN} \times} \bar \rmT)_n =(\tilde \rmN)^{n-1} \times \bar \rmT$ is the map induced by $j$, then $j_n=id ^{n-1} \times j$,
with a similar assertion holding for the map induced by $j$, $(\rmE \rmWT{\underset {\rmWT} \times} \O)_n \ra (\rmE \rmWT{\underset {\rmWT} \times} \bar \rmT)_n$, which will be also denoted $j_n$.
Therefore (denoting by $j$ also the map $\{j_n|n\}$ of simplicial varieties), 
\[R\psi_*(j_!j^*(\Q) = \{R\psi_{n*}(j_{n!}j_n^*(\Q)|n\} = \{Rp_{n*}(\Q) \boxtimes j_!(j^*(\Q))|n\} = \{j_{n!}j_n^*(Rp_{n*}(\Q) \boxtimes \Q)|n\} = j_!j^*R\psi_*(\Q). \]
\vskip .2cm \noindent
Therefore, one obtains the identifications 
\[R\psi_* \circ L\psi^*(j_!R\psi_*(j^*\Q)) = j_!j^*(R\psi_*(\Q)) \mbox{ and } L\psi^* \circ R\psi_*(j_!j^*\Q) = j_!j^*(\Q). \]
Here we have
denoted the restriction of the map $\psi$ to any of the orbits also by $\psi$. Since the derived category ${\rmD}^{b,o}_{\tilde \rmN, c}(\bar \rmT)$
(${\rmD}^{+,o}_{\rmWT, c}(\bar \rmT, R\psi_*(\Q))$) is generated by the
$j_!j^*(\Q)$ ($j_!(R\psi_*j^*(\Q))$, \res)  as one varies over the $\tilde \rmN$ -orbits, and $L\psi^*$ is left-adjoint to $R\psi_*$, 
the statement (i) in the proposition follows. 
\vskip .2cm
Proposition ~\ref{transf.t.struct} shows that when one has an equivalence of triangulated categories, one may transfer
the $t$-structure on one of the triangulated categories to the other. This proves the statement (ii) and completes the proof of
the Proposition.
\end{proof}
\begin{definition}
 \label{t.struct.gluing}
Henceforth we call the $t$-structure obtained by gluing the standard $t$-structures on a stratified variety, shifted  by the codimensions
of the strata (as in \cite[2.2]{BBD82}), {\it the $t$-structure obtained by gluing}.
\end{definition}

One starts with either the standard $t$-structures or the $t$-structures obtained by gluing on the derived categories $\rmD_{\rmG \times \rmG, c}^b(\bar \rmG)$ and  ${\rmD}_{(\rmT \times \rmT)\diag (N(T) ,c}^{b,o}(\bar \rmT)$. 

\begin{corollary}
\label{equiv.1}
 We obtain the equivalences of derived categories:
\[\rmD_{\rmG \times \rmG, c}^b(\bar \rmG) {\overset {res} \ra}  {\rmD}_{(\rmT \times \rmT)\diag (N(T) ,c}^{b,o}(\bar \rmT)  {\overset {R\psi_*} \ra} \rmD^{+,o}_{WT, c}(\bar \rmT, R\psi_*(\Q)). \]
The derived category ${\rmD}^{+,o}_{\rmWT, c}(\bar \rmT, R\psi_*(\Q))$ obtains
induced $t$-structures, induced from the $t$-structures on ${\rmD}^{b,o}_{\tilde \rmN, c}(\bar \rmT)$ by the functor $R\psi_*$.
\vskip .1cm
Moreover, if $\rmIC^{\tilde \rmN}(\O_{\tilde \rmN})$ denotes the $\tilde \rmN$
-equivariant intersection cohomology complex on the closure of the $\tilde \rmN$-orbit $\O_{\tilde \rmN}$ obtained from the constant local system
on the corresponding orbit, then 
$R\psi_*(\rmIC^{\tilde \rmN}(\O_{\tilde \rmN}) = \rmIC^{\rmWT}(\O_{\rmWT}) \otimes R\psi_*(\Q)$. Here
$\O_{\rmWT}$ is the same $\tilde \rmN$-orbit viewed as an orbit for the action of $\rmWT$ and $\rmIC^{\rmWT}(\O_{\rmWT})$ denotes
the $\rmWT$-equivariant intersection cohomology complex on the closure of the orbit $\O_{\rmWT}$ obtained from the
constant local system on $\O_{\rmWT}$.
\end{corollary}
\begin{proof} All conclusions but the last are  clear by combining Theorem ~\ref{eq.toroidal.1} and Proposition ~\ref{Rpsi*}. To see the last conclusion, first observe that
$\rmIC^{\tilde \rmN}(\O_{\tilde \rmN}) \simeq \psi^*(\rmIC^{\rmWT}(\O_{\rmWT})$ and the map $\psi$ is a locally trivial fibration. Therefore, the last conclusion 
 follows from the projection formula.
 \end{proof}
\vskip .2cm
%\begin{remark}
% The toric variety $\bar \rmT$, with the action of $\rm\rmT \times \rmT$, is a non-trivial example of a {\it toric stack}, since the diagonal torus acts
%trivially on $\bar \rmT$. Thus, perhaps surprisingly, the theory of toric stacks shows up in the analysis of the equivariant derived categories for 
%toroidal group imbeddings.
%\end{remark}

\vskip .2cm
\subsection{\bf Step 3 of the proof of Theorem ~\ref{mainthm.2}: proof of Theorem ~\ref{mainthm.2.1.0}.  }

It is important for us to be able to separate the $\rmW$ and $\rmT$ actions on $\bar \rmT$. We proceed to do this.
\vskip .2cm
First we define actions of finite groups on simplicial varieties. 
\begin{definition} ({\it Finite group actions on simplicial varieties})
\label{finite.grp.acts.simpl.spaces}
If $\rmX_{\bullet}$ is a simplicial variety, an action of the finite group $\rmW$ on $\rmX_{\bullet}$ corresponds
to a $\rmW$-action on each $\rmX_n$ so that the structure maps of the simplicial variety $\rmX_{\bullet}$ are all $\rmW$-equivariant.
The two results below, follow readily from the fact that $\rmWT$ is the semi-direct product of $\rmW$ and $\rmT$. 
\end{definition}
\begin{proposition} 
\label{equiv.tech.3}
(i) If $\rmX$ is a variety, then giving a $\rmWT$-action on $\rmX$ is equivalent to providing $\rmT$ and $\rmW$ actions on $\rmX$ so that the $\rmT$-action $\mu: \rmT \times \rmX \ra \rmX$
is $\rmW$-equivariant. Equivalently, the diagram
\be \begin{equation}
     \label{WT.1}
\xymatrix{{\rmW \times \rmT \times \rmX} \ar@<1ex>[r]^{id \times \mu_{\rmT}} \ar@<1ex>[d]^{\mu'_{\rmW}} & {\rmW \times \rmX} \ar@<1ex>[d]^{\mu_{\rmW}}\\
           {\rmT \times \rmX} \ar@<1ex>[r]^{\mu_{\rmT}} & {\rmX}}
    \end{equation} \ee
commutes, where $\mu_W$ ($\mu_{\rmT}$) denote the ${\rmW}$ ($\rmT$) action and $\mu'_{\rmW}$ denotes the diagonal action of ${\rmW}$ on $\rmT \times \rmX$.
Equivalently, the following relations hold (where $\circ$ denotes generically any of the actions):
\be \begin{equation}
     \label{WT.2}
 w\circ(t \circ x) = (w\circ t ) \circ (w \circ x), w \eps {\rmW}, t \eps \rmT, x \eps \rmX.
   \end{equation} \ee
(ii) Moreover, in this case the simplicial variety $\rmET{\underset {\rmT} \times}\rmX$ has an induced action by ${\rmW}$.
\end{proposition}
\begin{proof}  We skip the proofs as these may be deduced in a straightforward manner from the fact that $\rmWT$
is the semi-direct product of ${\rmW}$ and $\rmT$. 
\end{proof}

\vskip .2cm
\begin{lemma} 
\label{equiv.tech.4}
Assume the situation of Proposition ~\ref{equiv.tech.3}. Given $t_i \eps \rmT $ and $w_i \eps {\rmW}$, $i=1,2$ and $x \eps \rmX$,
 one obtains the following identifications:
\[t_1w_1 \circ (w \circ x) = t_1 \circ (w_1 \circ (w \circ x)), t_2w_2 t_1 w_1 \circ (w \circ x) = t_2w_2t_1 \circ( w_1 \circ (w \circ x))
 =t_2 \circ ( w_2 \circ (t_1 \circ (w_1 \circ w\circ x)))
\]
\end{lemma}
\begin{proof} Again the proof is skipped since this follows from a straight-forward calculation.
\end{proof}
One may now
consider the action of ${\rmW}$ on the simplicial variety $\ETx \bar \rmT$, where we let ${\rmW}$ act diagonally on the two factors $\rmET $  and $\bar \rmT$.
We let $\rmEW{\underset {\rmW} \times} (\ETx \bar \rmT)$ denote the simplicial variety which is the diagonal of the resulting bisimplicial variety. 
\begin{lemma}
 \label{equiv.tech.5}
Now we obtain an equivalence of categories:
\[\rmD^b_{cart,c}(E{\rmW}{\underset {\rmW} \times} (\ETx \bar \rmT), \Q) \simeq \rmD^b_{\rmWT,c}(\rmEWT{\underset {\rmWT} \times}\bar \rmT, \Q) =\rmD^b_{{\rmW}T,c}(\bar \rmT, \Q)\]
where the subscript $cart$ ($c$) denotes the complexes of sheaves having cartesian (constructible, \res)
 cohomology sheaves. (Recall that a sheaf $F=\{F_n|n\}$ on a simplicial variety $\rmX_{\bullet}$ is cartesian, if
all the maps $\phi: \alpha^*(F_m) \ra F_n$ are isomorphisms for each structure map $\alpha: \rmX_n \ra \rmX_m$ of the simplicial 
space $\rmX_{\bullet}$. See ~\ref{sheaves.simp.1} for further details.)
\vskip .2cm \noindent
The above equivalence extends to an equivalence of the corresponding bounded below derived categories:
\[\rmD^+_{cart,c}(\rmEW{\underset {\rmW} \times} (\ETx \bar \rmT), \Q) \simeq \rmD^+_{\rmWT,c}(\rmEWT{\underset {\rmWT} \times}\bar \rmT, \Q) =\rmD^+_{\rmWT,c}(\bar \rmT, \Q).\]
\vskip .2cm \noindent
The above equivalence of derived categories preserves both the standard $t$-structures and the $t$-structures obtained by gluing. 
\end{lemma}
\begin{proof}
 Since there is no direct map between the simplicial varieties $\rmEW{\underset {\rmW} \times} (\ETx \bar \rmT)$ and $\rmEWT{\underset {\rmWT} \times}\bar \rmT$,
we make use of the intermediate simplicial variety: $\rmEWT{\underset {\rmWT} \times} (\rmET \times \bar \rmT)$ which maps to both the above
simplicial varieties. Let 
\be \begin{align}
     \label{p.i}
p_1: \rmEWT{\underset {\rmWT } \times} (ET \times \bar \rmT) &\ra \rmEW{\underset {W} \times} (\ETx \bar \rmT) \mbox{ and}\\
p_2:\rmEWT{\underset {\rmWT } \times} (ET \times \bar \rmT) &\ra \rmEWT{\underset {\rmWT } \times} (\bar \rmT) \notag
\end{align} \ee
denote the  maps defined as follows. $p_1$ is induced by the maps
sending $\rmEWT$ to ${\rm {EW}}$ by taking the quotient by $\rmT$ and ${\rm {ET}} \times \bar \rmT$ to $\ETx \bar \rmT$. 
(Recall $(\rmEWT{\underset {\rmWT } \times} (ET \times \bar \rmT))_n = (\rmWT)^n \times \rmT^{n+1} \times \bar \rmT$, while
$(\rmEW{\underset {\rmW} \times} (\ETx \bar \rmT))_n = \rmW^n \times \rmT^n \times \bar \rmT$, so it should be clear what $p_{1,n}$ is.) The second map $p_2$ is 
induced by the map that drops the factor ${\rm {ET}}$.
\vskip .2cm
One may observe the need to work with the simplicial model of classifying spaces in being able to define the first map $p_1$.
The fibers of both these simplicial maps are $\rmET$ which is acyclic. Moreover, these maps are cohomologically proper and the projection 
formula holds for both of the maps, so that one concludes readily that the natural map $K \ra Rp_{1*}p_1^*(K)$ is a quasi-isomorphism for any
complex $ K \eps \rmD^+_{cart,c}(\rmEW{\underset {\rmW} \times} (\ETx \bar \rmT), \Q)$ and that the natural map
$K \ra Rp_{2*}p_2^*(K)$ is a quasi-isomorphism for any complex $K \eps \rmD^+_{\rmWT,c}(\rmEWT{\underset {\rmWT } \times}\bar \rmT, \Q)$. 
Therefore, it follows that both the functors $p_i^*$ are fully-faithful at the level of the bounded below derived categories. The heart of both 
the above bounded derived categories is the category of
$\rmWT$-equivariant sheaves on $\bar \rmT$. (This is clear for the derived category $\rmD^b_{WT,c}(\rmEWT{\underset {\rmWT } \times}\bar \rmT, \Q)$. To see the same for 
the derived category $\rmD^b_{cart,c}(\rmEW{\underset {\rmW} \times} (\ETx \bar \rmT), \Q)$, one
makes use of Proposition ~\ref{equiv.tech.3}.) Therefore, both the functors $p_i^*$ are equivalences of 
the bounded derived categories, thereby proving the first statement in the lemma. Since both the functors $p_i^*$ are exact, it is clear that they
preserve the standard $t$-structures. It is clear that the maps $p_i$ preserve the stratification on $\bar \rmT$ by $\rmWT$-orbits. Since the map
$(p_1)_n = \rmWT^n \times \rmT^{n+1} \times \bar  \rmT \ra \rmW^n \times \rmT^n \times \bar \rmT$ is given by the identity map on the factor $\bar \rmT$, one may readily 
see that the functor $p_1^*$ commutes with the functors $j_!$, $j^*$, $Rj_*$ and $Rj^!$ associated to the locally closed immersion 
$j:\rmU _i \ra \rmU _{i+1}$ of strata appearing in the stratification of $\bar \rmT$ by $\rmWT$-orbits. The same holds for the functor $p_2^*$. These
observations then show that the above equivalences preserve the $t$-structures obtained by gluing.
\vskip .2cm
To extend this equivalence for the corresponding bounded below derived categories, one proceeds as follows.
Let $\rmD = \rmD^+_{cart,c}(\rmEWT{\underset {\rmWT} \times} (\rmET \times \bar \rmT), \Q)$.
Observe that the  inclusion functor $\rmD^{\le n} \ra \rmD$ has a right adjoint given by the cohomology truncation functor
$\sigma _{\le n}$: here $\rmD^{\le n}$ denotes the full subcategory of $\rmD$ consisting of 
complexes that
have trivial cohomology above degree $n$. Next observe that the functors $Rp_{i*}=p_{i*}$, $i=1, 2$, and hence are exact functors. 
Similarly $p_i^*$, $i=1, 2$, are exact functors. Therefore all of the above functors  commute with
the functors $\sigma_{\le n}$ and  if $K \eps \rmD$, then $p_{i*}(\sigma_{\le n}K)$ are
bounded complexes for each $n$. Since $p_i^*$ is an equivalence at the level of the above bounded derived categories
with inverse $p_{i*}$, it follows that $K = \colimn \sigma _{\le n}(K) \simeq 
\colimn p_i^*(p_{i*} \sigma _{\le n}K) \simeq p_i^*(\colimn p_{i*}(\sigma_{\le n}K))$ and that
$K \simeq \colimn \sigma_{\le n}K {\overset  {\simeq} \ra} \colimn Rp_{i*}p_i^*(\sigma_{\le n} K) \simeq Rp_{i*}p_i^*(\colimn \sigma _{\le n}K) = Rp_{i*}p_i^*(K)$, thereby proving that 
the functors $p_i^*$ also 
induce equivalences of the corresponding bounded below derived categories.
   \end{proof}
Next we consider the subcategory $\rmD^+_{WT,c}(\bar \rmT, R\psi_*(\Q))$. Now $Rp_{1*}(p_2^*(R\psi_*(\Q))$ defines a sheaf of dgas on
$ \rmEW{\underset {\rmW} \times} (\ETx \bar \rmT)$. 
\begin{definition}
 \label{A}
Henceforth, we will denote the sheaf of dgas $Rp_{1*}p_2^*(R\psi_*(\Q))$ on
$ \rmEW{\underset {\rmW} \times} (\ETx \bar \rmT)$
 by $\A$. 
\end{definition}
Recall that $R\psi_*(\Q)$ is a complex of  sheaves on $\rmEWT{\underset {\rmWT } \times} (\bar \rmT)$ and that 
$\A= Rp_{1*}p_2^*(R\psi_*(\Q))$ is a complex of sheaves on 
$ \rmEW{\underset {\rmW} \times} (\ETx \bar \rmT)$. 
\begin{corollary}
\label{equiv.2}
 The equivalence of categories in Lemma ~\ref{equiv.tech.5} extends to an equivalence of categories:
\[\rmD^+_{cart,c}(\rmEW{\underset W \times} (\ETx \bar \rmT), \A) \simeq \rmD^+_{WT,c}(\rmEWT{\underset {\rmWT } \times}\bar \rmT, R\psi_*(\Q) = \rmD^+_{WT,c}(\bar \rmT, R\psi_*(\Q)) \mbox{ and }\]
\[{\rmD}^{+,o}_{cart,c}(\rmEW{\underset {\rmW} \times} (\ETx \bar \rmT), \A) \simeq {\rmD}^{+,o}_{\rmWT,c}(\rmEWT{\underset {\rmWT} \times}\bar \rmT, R\psi_*(\Q) ={\rmD}^{+,o}_{\rmWT,c}(\bar \rmT, R\psi_*(\Q))\]
where ${\rmD}^{+,o}_{cart,c}(\rmEW{\underset {\rmW} \times} (\ETx \bar \rmT), \A)$ is the full subcategory of 
${\rmD}(\rmEW{\underset {\rmW} \times} (\ETx \bar \rmT), \A)$
generated by the functor $Rp_{1*}p_2^*$ applied
to the generators $\{R\psi_*(j_!j^*(\Q_{|\O}))| \O \mbox{ a } \rmWT-\mbox{orbit on } \bar \rmT\}$ of the subcategory 
\newline \noindent
${\rmD}^{+,o}_{\rmWT,c}(\bar \rmT, R\psi_*(\Q))$.
\vskip .2cm
The $t$-structures on the derived categories on the left are obtained by transferring the $t$-structures on the corresponding
derived categories on the right. Moreover,
\[Rp_{1*}p_2^*(R\psi_*(\rmIC^{\tilde \rmN}(\O_{\tilde \rmN}))) = \rmIC^{\rmW, \rmT}(\O_{\rmWT}) \otimes Rp_{1*}p_2^*(R\psi_*(\Q)).\]
where $\rmIC^{\rmW, \rmT}(\O_{\rmWT})$
denotes the same equivariant intersection cohomology complex $\rmIC^{\rmWT}(\O_{\rmWT})$ for the action $\rmWT$, but viewed as an 
object on $\rmEW{\underset {\rmW} \times} (\ETx \bar \rmT)$. 
\end{corollary}
\begin{proof}
All but the last conclusions are clear in view of the observations above, and also from the observation that the composite functors $Rp_{i*} \circ p_i^*$  and $p_{i}^* \circ Rpi_{i*}$ are still
the identity at the level of the above derived categories of dg-modules. To see the last conclusion, first observe that
\be \begin{equation}
     \label{ICWT.1}
   \rmIC^{\rmWT}(\O_{\rmWT}) = \oplus_{w \eps {\rmW}/{\rmW}_{\O_{\rmT}}} \rmIC^{\rmT}(w\O_{\rmT}),
\end{equation} \ee
where $W_{\O_{\rmT}}$ denotes the stabilizer in $\rmW$ of the $\rmT$-orbit $\O_{\rmT}$, and $\rmIC^{\rmT}(w\O_{\rmT})$ denotes the intersection cohomology complex obtained by extending the constant local system on 
 $w\O_{\rmT}$. (This may be deduced from the fact that the 
equivariant intersection cohomology complexes $\rmIC^{\rmWT}(\O_{\rmWT})$ are  perverse 
extensions of the constant sheaf $\Q_{\O_{\rmWT }}= \oplus _{w \eps \rmW/\rmW_{\O_{\rmT}}} \Q_{\O_{\rmT}}$.) 
Therefore,
$p_2^*(\rmIC^{\rmWT}(\O_{\rmWT}))= p_1^*(\rmIC^{\rmW, \rmT}(\O_{\rmWT}))$: recall $\rmIC^{\rmW, \rmT}(\O_{\rmWT})$
denotes the same equivariant intersection cohomology complex $\rmIC^{\rmWT}(\O_{\rmWT})$ for the action $\rmWT$, but viewed as an 
object on $\rmEW{\underset {\rmW} \times} (\ETx \bar \rmT)$. This is possible, in view of ~\eqref{ICWT.1}.
\vskip .2cm
Next observe from Corollary ~\ref{equiv.1}
that $R\psi_*(\rmIC^{\tilde \rmN}(\O_{\tilde \rmN})) = \rmIC^{\rmWT}(\O_{\rmWT}) \otimes R\psi_*(\Q)$. 
Therefore, the last identification in the corollary follows
from a projection formula.
\end{proof}
\vskip .2cm
Let $\bar \rmT/ \rmT $ denote the space whose points correspond to the $\rmT$-orbits on $\bar \rmT$ equipped with the topology as in the
 introduction.
Let 
\be \begin{equation}
     \label{pi}
\pi: \rmET{\underset {\rmT} \times} \bar \rmT \ra \bar \rmT/\rmT
\end{equation} \ee
denote the obvious map. The fact that the action of $\rmT$ on $\bar \rmT$ is equivariant with respect to the $\rmW$-action shows that
one has an induced action of $\rmW$ on the set of $\rmT$-orbits on $\bar \rmT$.  Denoting this $\rmW$-action on $\bar \rmT/\rmT$ also by $\circ$, observe that
$w\circ [x] = [w\circ x]$, for any $ w\eps \rmW$ and $x \eps \bar \rmT$, where $[x]$ denotes the
$\rmT$-orbit of $x$. Now $\rmW$ acts on both $\rmET{\underset {\rmT} \times} \bar \rmT$ and $\bar \rmT /\rmT$.
\begin{lemma}
 \label{pi.is.equivariant}
Assume the above situation. Then,  with respect to the above actions, the map $\pi$ is $\rmW$-equivariant.
\end{lemma}
\begin{proof}
 Observe that the map $\pi$ sends $(t_1, t_2, \cdots, t_{n-1}, x) \eps (\rmET{\underset {\rmT} \times} \bar \rmT)_n$ to $[x]$. Now one verifies that
\[\pi(w,(t_1, \cdots, t_{n-1}, x)) = \pi( w \circ t_1, \cdots, w \circ t_{n-1}, w \circ x) = [w \circ x] = w \circ [x] = w \circ \pi(t_1, \cdots, t_{n-1}, x).\]
This proves the map $\pi$ is $\rmW$-equivariant. 
\end{proof}
\vskip .1cm
Therefore, we will also let the map 
$ \rmEW{\underset {\rmW} \times} (\ETx \bar \rmT) \ra \rmEW{\underset {\rmW} \times} (\bar \rmT/{\rmT})$ induced by $\pi$ be denoted
 $\pi$. 
\begin{definition}
\label{Rpi*A}
 We define ${\rmD}^{+,o}_{cart,c}(\rmEW{\underset {\rmW} \times} (\bar \rmT/\rmT), R\pi_*(\A))$ to be the full subcategory of
\newline \noindent
$\rmD (\rmEW{\underset {\rmW} \times} (\bar \rmT/{\rmT}), R\pi_*(\A))$ generated by $R\pi_*(G_{\alpha})$, as $G_{\alpha}$ varies
over the generators of  the derived category $\rmD^{+,o}_{cart,c}(\rmEW{\underset {\rmW} \times} (\ETx \bar \rmT), \A)$ as in Corollary
~\ref{equiv.2}.
\end{definition}
\vskip .2cm
One may now define 
\[R\pi_*: \rmD^{+,o}_{cart,c}(\rmEW{\underset {\rmW} \times} (\ETx \bar \rmT), \A) \ra \rmD^{+,o}_{cart,c}(\rmEW{\underset {\rmW} \times} (\bar \rmT/{\rmT}), R\pi_*(\A))\]
by once again making use of the canonical Godement resolutions. A left derived functor
\[L\pi^*: \rmD^{+,o}_{cart,c}(\rmEW{\underset {\rmW} \times} (\bar \rmT/{\rmT}), R\pi_*(\A)) \ra \rmD^{+,o}_{cart,c}(\rmEW{\underset {\rmW} \times} (\ETx \bar \rmT), \A)\]
may be defined by taking flat-resolutions of $M \eps \rmD^+_{cart,c}(\rmEW{\underset {\rmW} \times} (\bar \rmT/{\rmT}), R\pi_*(\A))$ and letting
\[L\pi^*(M) = \A {\overset L {\underset {\pi^{-1}R\pi_*(\A)} \otimes}} \pi^{-1}(M).\]
So defined,  $L\pi^*$ will send objects in $\rmD_{cart,c}(\rmEW{\underset {\rmW} \times} (\bar \rmT/{\rmT}), R\pi_*(\A))$ to
objects in $\rmD_{cart,c}(\rmEW{\underset {\rmW} \times} (\ETx \bar \rmT), \A)$. 
\begin{remark}
 \label{funct.resols}
 Observe that we define $R\pi_*$ making use of the canonical Godement resolution. Therefore $R\pi_*$ will be functorial
at the level of complexes. Similarly, by making use of functorial flat resolutions, one makes the functor $L\pi^*$ also functorial
 at the level of complexes.
\end{remark}
\vskip .1cm
Since $L\pi^*$ will be left adjoint to
a functor $R\pi_*$ (defined at the level of the corresponding unbounded derived categories), one obtains
natural transformations $L\pi^* \circ R\pi_* \ra id $ and $id \ra R\pi_* \circ L\pi^*$. To show these are quasi-isomorphisms,
it suffices to restrict to the hearts of the corresponding derived categories. For this, we
will resort to a variant of the arguments used for toric (and also horospherical)  varieties.
\vskip .2cm
In fact the ${\rmW}$-symmetric toric variety $\bar \rmT$ is a toric variety for the torus $\rmT$ provided with an extra action by the
finite group ${\rmW}$ that permutes the various $\rmT$-orbits of the same type. 
Therefore let $\rmY$ denote a $\rmT$-orbit in $\bar \rmT$
and let $\rmU_{\rmY}$ denote the (unique) affine open $\rmT$-stable neighborhood in $\bar \rmT$ so that $\rmY$ is the only closed $\rmT$-orbit in $\rmU_{\rmY}$.  Similarly, let ${\rmY '}$ denote a $\rmT$-orbit on 
$\bar \rmT$ and  let  $\rmU_{\rmY '}$ denote the  affine open $\rmT$-stable neighborhood of $\rm{\rmY '}$ in $\bar \rmT$ so that 
$\rm{\rmY '}$ is the only closed
$\rmT$-orbit in $\rmU_{\rmY '}$. Assume that either ${\overline \rmY '} - {\rmY '} \supseteq \rmY$ (in which case $\rmU_{\rmY '} \subseteq \rmU_{\rmY}$), or that $\rmY \nsubseteq \rmU_{\rmY '}$.  Observe that under the ${\rmW}$-action on $\bar \rmT$, the 
$\rmT$-orbit $\rmY$ 
($\rm{\rmY '}$)
is sent to another $\rmT$-orbit $w\rmY$ ($w\rm{\rmY '}$, \res) depending on $w \eps {\rmW}$. 
Therefore, as shown in \cite[Proof of Theorem 2.6]{Lu95} 
we observe that the restriction 
\be \begin{equation}
\label{contract.slices.WT.1}
 \rmH^*_{\rmT}(\rmU_{\rmY}, \A) \ra \rmH^*_{\rmT}(\rmU_{\rmY}-\rmU_{\rmY '}, \A) 
 \end{equation} \ee
 is an isomorphism: to see this one may observe that $\rmY$ is a deformation retract of both $\rmU_{\rmY}$ and $ \rmU_{\rmY}-\rmU'_{\rmY}$
 and  $R\psi_*(\Q)$, and therefore $\A$ have constant cohomology sheaves. (See ~\eqref{base.change.isoms} for the latter.)
  Then each $w \eps {\rmW}$ sends this to the restriction-isomorphism: 
  \be \begin{equation}
 \label{contract.slices.WT.2} 
  \rmH^*_T(\itw \rmU_{\rmY}, \A) \ra \rmH^*_T(\itw \rmU_{\rmY}-w\rmU_{\rmY}', \A). 
  \end{equation} \ee
\vskip .2cm
Next let ${\rmY '}$ be as before, but choose $\rmY$ to be another $\rmT$-orbit on $\bar \rmT$ so that $\rmY \nsubseteq \rmU_{\rmY '}$. Let
$j:\rmU_{\rmY '} \ra \bar \rmT$ denote the open immersion  and let $i:\bar \rmT - \rmU_{\rmY '} \ra \rmX$ denote the corresponding closed immersion. 
Let $\pi: \rmET{\underset {\rmT} \times} \bar \rmT \ra \bar \rmT/\rmT$ denote the  map considered in ~\eqref{pi}. We will consider the following commutative diagram
\be \begin{equation}
\label{big.diagm}
\xymatrix{{\rmET{\underset {\rmT} \times}\rmU_{\itw \rm{\rmY '}}} \ar@<1ex>[r]^{j_{\itw \rm{\rmY '}}} \ar@<1ex>[d]^{\pi_{\rmU_{\itw \rm{\rmY '}}}} &  {\rmET{\underset {\rmT} \times}{\bar \rmT}} \ar@<1ex>[d]^{\pi}\\
           {\rmU_{\itw \rm{\rmY '}}/\rmT} \ar@<1ex>[r]^{j_{\itw \rm{\rmY '}}/{\rmT}}  & {{\bar \rmT}/\rmT}}
    \end{equation} \ee
\vskip .2cm \noindent
where the maps $j_{w\rm{\rmY '}}$,  $j_{w\rm{\rmY '}/\rmT}$ and $\pi_{\rmU_{w\rm{\rmY '}}}$ are defined by the above diagram. 
\vskip .2cm
\begin{lemma}
 \label{ext.by.zero}
Assume the above situation. Then,
(i) $R\pi_* j_{w{\rmY '}!} { j}_{w{\rmY '}}^* (\A) \simeq { j}_{w{\rmY '}/{\rmT}!} { j}_{w{\rmY '}/{\rmT}}^*R \pi_*(\A)$.
\vskip .2cm
(ii) The equivariant derived category $\rmD^{+,o}_{cart,c}(\rmEW{\underset {\rmW} \times} (\ETx \bar \rmT), \A)$ is generated
by objects of the form $\oplus _{w \eps {\rmW}/{\rmW}_{{\rmY '}}} j_{w{\rmY '}!}j_{w{\rmY '}}^*(\A)$, where ${\rmW}_{{\rmY '}}$ denotes the stabilizer in
${\rmW}$ of the $\rmT$-orbit ${\rmY '}$.
\end{lemma}
\begin{proof}
(i) We obtain this from the isomorphism ~\eqref{contract.slices.WT.2}, by observing that, therefore, 
\newline \noindent
$\rmH^*_{T, \itw \rmU_{\rmY}-\itw \rmU_{\rmY}'} (\itw \rmU_{\rmY}, \A) =0. $
\vskip .2cm \noindent
(ii) Recall that the generators of the derived category $\rmD^{+,o}_{cart,c}(\rmEW{\underset {\rmW} \times} (\ETx \bar \rmT), \A)$
are obtained by applying the functor $Rp_{1*}p_2^*$ to the generators of the derived category $\rmD^{+,o}_{\rmWT,c}(\rmEWT{\underset {\rmWT} \times} \bar \rmT, R\psi_*(\Q))$.
The generators of the latter are  $R\psi_*(j_{\O_{\rmWT}!}(\Q_{|\O_{\rmWT}}))$ as one varies over the $\rmWT$-orbits $\O_{\rmWT}$ on 
$\bar \rmT$. But each $\rmWT$-orbit $\O_{\rmWT} = \sqcup_{w \eps {\rmW}/\rmW_{{\rmY '}}} {w{\rmY '}}$, where ${\rmY '}$ is a $\rmT$-orbit on $\bar \rmT$, with
$\rmW_{{\rmY '}}$ denoting the stabilizer of the $\rmT$-orbit ${\rmY '}$ for the action of $\rmW$ on $\bar \rmT/{\rmT}$.
\end{proof}
\begin{proposition}
\label{Rpi*.and.Lpi*}
 The natural transformations $K \ra R \pi_* L\pi^*(K) $, $K \eps  \rmD^{+,o}_{cart,c}(\rmEW{\underset {\rmW} \times} (\ETx \bar \rmT), \A)$ and
 $L\pi^*R\pi_*(L) \ra L$  for any $L \eps \rmD^{+,o}_{cart,c}(\rmEW{\underset {\rmW} \times} (\bar \rmT/{\rmT}), R\pi_*(\A))$ are quasi-isomorphisms. Therefore, the functors $R\pi_*$ and $L\pi^*$ 
 induce an equivalence of categories.
\end{proposition}
\begin{proof}
 The key observation (see Lemma ~\eqref{ext.by.zero} (ii)) is that the derived category $\rmD^{+,o}_{cart,c}(\rmEW{\underset {\rmW} \times} (\ETx \bar \rmT), \A)$ is generated by 
the objects $\oplus _{w \eps {\rmW}/{\rmW}_{{\rmY '}}} j_{w{\rmY '}!}j_{w{\rmY '}}^*(\A)$ while the derived category $\rmD^{+,o}_{cart,c}(\rmEW{\underset {\rmW} \times} (\bar \rmT/{\rmT}), R\pi_*(\A))$
is generated by the objects $\oplus _{w \eps {\rmW}/{\rmW}_{{\rmY '}}} j_{w{\rmY '}/{\rmT}!}j_{w{\rmY '}/{\rmT}}^*R\pi_*(\A)$ as ${\rmY '}$ varies among the $\rmT$-orbits on $\bar \rmT$. Lemma ~\eqref{ext.by.zero} shows that the functor $R\pi_*$
sends $\oplus _{w \eps {\rmW}/{\rmW}_{{\rmY '}}} j_{w{\rmY '}!}j_{w{\rmY '}}^*(\A)$ to $\oplus _{w \eps {\rmW}/{\rmW}_{{\rmY '}}} j_{w{\rmY '}/{\rmT}!}j_{w{\rmY '}/{\rmT}}^*R\pi_*(\A)$.  One may readily see that
$L\pi^*(\oplus _{w \eps {\rmW}/{\rmW}_{{\rmY '}}} j_{w{\rmY '}/{\rmT}!}j_{w{\rmY '}/{\rmT}}^*R\pi_*(\A)) = \oplus _{w \eps {\rmW}/{\rmW}_{{\rmY '}}} j_{w{\rmY '}!}j_{w{\rmY '}}^*L\pi^* R\pi_*(\A) = \oplus _{w \eps {\rmW}/{\rmW}_{{\rmY '}}} j_{w{\rmY '}!}j_{w{\rmY '}}^*(\A)$.
Therefore,
one readily sees that the natural transformations $L\pi^* \circ R\pi_* \ra id $ and $id \ra R\pi_* \circ L\pi^*$ are isomorphisms
proving the proposition.
\end{proof}
\vskip .2cm \noindent
{\bf Proof of Theorem ~\ref{mainthm.2.1.0}}. Combining Proposition ~\ref{Rpi*.and.Lpi*}, Corollaries ~\ref{equiv.1} and ~\ref{equiv.2}, we obtain Theorem ~\ref{mainthm.2.1.0}.
\vskip .2cm
\section{\bf Proof of Theorem ~\ref{mainthm.2}: Step 4.} 
The remainder of this section is devoted to a proof of the remaining aspects of Theorem ~\ref{mainthm.2}. 
We begin with the following result, which applies to any group action. 
\begin{lemma}
 \label{equiv.struct.maps} Let $\rmG$ denote a linear algebraic group acting on a variety $\rmX$. Then for the action of $\rmG$ on $(\EGx \rmX)_n$ given by $g \circ (g_0, \cdots, g_{n-1}, x) = (gg_0g^{-1}, \cdots, gg_{n_1}g^{-1}, g\circ x)$,
 with $(g_0, \cdots, g_{n-1}, x) \eps \rmG^n \times \rmX$, $g \eps \rmG$, the structure maps of the simplicial variety $\EGx X$ are all $\rmG$-equivariant.
\end{lemma}
\begin{proof} We first verify this for $n=1$. In this case, the face map $d_0= pr_{2} $, ( which is the projection to the second factor $\rmG \times X \ra \rmX$), and
 $d_1= \mu: \rmG \times \rmX \ra \rmX$ (which is the group action). Then $g\circ d_0(g_0, x) = g \circ x= d_0( gg_0g^{-1}, g\circ x) = 
 d_0(g \circ (g_0, x))$ and $g\circ d_1(g_0, x) = g \circ (g_0 \circ x) =  (g  g_0) \circ x = d_1( gg_0g^{-1}, g\circ x) = 
 d_1(g \circ (g_0, x))$. The face map $s_0$ sends $x \eps \rmX$ to $(e, x) \eps \rmG \times \rmX = (\EGx \rmX)_1$. Now, $g\circ s_0(x) = g\circ (e, x) = (geg^{-1}, g \circ x) = (e, g \circ x) = s_0(g\circ (e, x))$.  The proof that the remaining structure maps are all $\rmG$-equivariant is similar, and is therefore
 skipped.
\end{proof}
\vskip .1cm 

Next we show that 
the category of equivariant sheaves admits a simpler formulation when the group is a discrete group. This seems to be rather well-known: see 
\cite[Chapter I, section 8]{BL94} for related results.
Let $\rmW$ denote a discrete group acting on a variety $\rmX$ and let $\rmR$ denote a commutative
Noetherian ring. We will assume $\rmX$ is provided
 with a (Grothendieck) topology $\Top(\rmX)$. Let $\Sh^{\rmW}(\rmX, \rmR) =\Sh^{\rmW}(\rmX)$ denote the following category. Its objects are sheaves of $F$ of $\rmR$-modules
  on $\Top(\rmX)$ so that one is  provided with a homomorphism $\rmW \ra Aut(\itF)$, where ${\rm Aut(\itF)}$ denotes the automorphism group of $\itF$ as a sheaf of 
  $\rmR$-modules. Morphisms between two such sheaves $F'$ and $F$ are morphisms of sheaves compatible with the given extra structure.

\begin{lemma}
 \label{discrete.grp.action}
 Assume the above situation.
 \begin{enumerate}[\rm(i)]
  \item Then the categories $\Sh^{\rmW}(\rmX)$ and $\Sh^{\rmW}(\EWx \rmX)$ are equivalent, where the category $\Sh^{\rmW}(\EWx \rmX)$
   denotes the full subcategory of {\it cartesian} sheaves of $\rmR$-modules on the simplicial variety $\EWx \rmX$.
  \item If $F=\{F_n|n\} \eps \Sh^W(\EWx \rmX)$, then each
  $F_n \eps \Sh^{\rmW}((\EWx \rmX)_n)$, when $(\EWx \rmX)_n$ is provided with the $\rmW$-action as in Lemma ~\ref{equiv.struct.maps}.
\item
  If $X=\{x_o\}$ denotes a point, then we obtain the equivalence: $\Sh^{\rmW}(\rmBW) \simeq \Sh^{\rmW}(\{x_o\})$, and the latter category is
 the category of $\rmR$-modules provided with $\rmW$-actions.
 \end{enumerate}
\end{lemma}
\begin{proof} For each $w \eps \rmW$, we let $w: \rmX \ra \rmX$ denote the automorphism induced by $w$. 
Now, one may observe that giving a homomorphism $\rmW \ra Aut(F)$ corresponds to giving for each $w \eps \rmW$, an isomorphism $w^{-1}(F) \ra F$ (or equivalently
  an isomorphism $F \ra w_*(F)$) of sheaves of $\rmR$-modules, and which are compatible as $w$ varies in $\rmW$. 
  In view of the fact that the group $\rmW$ is discrete, one may now readily show that the category $\Sh^W(\rmX)$ is equivalent to the category of
  sheaves $F_0$ on $\Top(\rmX)$ provided with an isomorphism $\phi: \mu^*(F_0) \ra pr_2^*(F_0)$, satisfying a cocycle condition on further pull-back
  to $(\EWx \rmX)_2$, and so that $s_0^*(\phi)$ is the identity. (Here $\mu, pr_2: \rmG \times \rmX \ra \rmX$
  are the group action and projection to the second factor, respectively, while $s_0: \rmX \ra \rmG \times \rmX$
  is the map $x \mapsto (x, e)$, where $e$ denotes the identity of $\rmG$ and $x \eps \rmX$.)
  \vskip .1cm
  It is well-known that this latter category is equivalent to $\Sh^{\rmW}(\EWx \rmX)$.
  In fact, given an $F_0 \eps \Sh^W(\rmX)$, one defines $F = \{F_n|n\} \eps \Sh^{\rmW}(\EWx \rmX)$, by letting $F_n = ((d_0)^n)^*(F_0)$. The inverse of this functor sends
  any such $F=\{F_n|n\}$ to $F_0$. Finally, one may see from the above description that if $F = \{F_n|n\}$ belongs to $\Sh^{\rmW}(\EWx \rmX)$, then
  $F_0 \eps \Sh^{\rmW}(\rmX)$. Since all the structure maps of the simplicial variety $\EWx \rmX$ have been shown to be $\rmW$-equivariant (see
  Lemma ~\ref{equiv.struct.maps}), it follows that $F_n= ((d_0)^n)^*(F_0) \eps \Sh^{\rmW}((\EWx \rmX)_n)$. 
  The last statement is clear since $\EWx \{x_o\} = \rmBW$.
\end{proof}

{\bf Step 4.1}. The {\it first result}
 we proceed to establish in this section is that the dga $\A$, considered in Definition ~\ref{A} and the dga $R\pi_*(\A)$, (considered in Definition ~\ref{Rpi*A}) are formal
 as dgas.
\vskip .1cm

 \begin{definition}
  \label{W.equ.complex}
  A complex of sheaves $K$  of $\Q$-vector spaces on $\EWx\rmX$ is ${\rmW}$-equivariant, if each term $K^n$ in the complex
is a ${\rmW}$-equivariant sheaf and the differentials of $K$ are also ${\rmW}$-equivariant, that is, $K^{\bullet}$
 is a complex in the abelian category $\Sh^{\rmW}(\EWx\rmX, \Q)$.
 \end{definition}
\vskip .2cm
We next consider the commutative diagram:
\be \begin{equation}
\label{key.comm.diagm}
\xymatrix{{\rmE\tilde \rmN{\underset {\tilde \rmN} \times} \bar \rmT} \ar@<1ex>[d]^{\psi} \ar[r]^{\alpha} &  {\rmB \tilde \rmN}\ar@<1ex>[d]^{\bar \psi} \ar[r]^{\beta} &{\rmBW \diagT} \ar@<1ex>[d]^{\tilde \psi}\\
           {\rmEWT {\underset {\rmWT } \times}\bar \rmT}   \ar[r]^{\bar \alpha} &  {\rmBWT} \ar[r]^{\bar \beta} &{\rmBW }.} 
\end{equation} \ee
           
\vskip .1cm
Here  $\tilde \rmN = (\rmT \times \rmT)\diagNT$ and the torus $\rmT$ in $\rmWT$ denotes the quotient torus $  (\rmT \times \rmT)/\diagT$. One may identify this
torus with the sub-torus $(1 \times \rmT) \subseteq \rmT\times \rmT \subseteq (\rmT \times \rmT) {\rm {diag}}{\rm N}({\rm T})$. $\rmW\diagT$ denotes the quotient
$((\rmT \times \rmT)\diagNT/(1 \times T))$. The maps $\alpha$ and $\bar \alpha$ are the obvious maps. The map $\beta$ corresponds to taking the
quotient of $\tilde \rmN$ by $(1 \times \rmT)$ and the map $\bar \beta$ corresponds to taking the quotient of $\rmWT$ by the corresponding torus which also 
identifies with $(1 \times \rmT)$. 
\vskip .1cm
Next, one may observe that the composite map  $\bar \beta \circ \bar \alpha \circ p_2: \rmEWT{\underset {\rmWT } \times} (ET \times \bar \rmT)
\ra \rmEWT{\underset {\rmWT } \times} \bar \rmT \ra \rmBW$ factors
also as the composite map $\phi \circ \pi \circ p_1$, where $\phi: \rmEW{\underset \rmW \times} (\bar \rmT/{\rmT}) \ra \rmBW$ is the map induced sending all of
${\bar \rmT}/\rmT$ to a point, $p_i$, $i=1,2$ are the maps defined
in ~\eqref{p.i} and $\pi$ is the map in ~\eqref{pi}. Therefore,
\be \begin{align}
\label{A.1}
\A &= Rp_{1*}p_2^*(\bar \alpha^* \bar \beta ^*(R\tilde \psi_*(\Q)))\\
    &= Rp_{1*} p_1^* \pi^* \phi^*(R \tilde \psi_*(\Q)) \notag\\
    &= \pi^*\phi^*(R\tilde \psi_*(\Q)) \otimes _{\Q} Rp_{1*}(\Q) \notag\\
    &= \pi^* \phi^*(R\tilde \psi_*(\Q)). \notag
    \end{align}\ee
The last equality follows from the observation that $Rp_{1*}(\Q) \simeq \Q$.
\begin{proposition}\footnote{As pointed out by one of the referees, it is not essential in this Proposition that $\rmW$ be 
the Weyl group $\rmN(T)/\rmT$. The same arguments should work if $\rmW$ is replaced by a finite group (which
we will continue to denote by $\rmW$ for want of a better notation), so that $\rmW \diag(T)$ is the semi-direct
product of $\rmW$ and ${\rm diag(T)}$, with ${\rm diag(T)}$ normal in $\rmW\diag(T)$.}
\label{formal.1}
(i) The sheaf of commutative dgas, $R\tilde \psi_*(\Q)= \tilde \psi_*\rmG^{\bullet}(\Q)$, (where $\rmG^{\bullet}$ is the canonical Godement resolution), 
is $\rmW$-equivariant and formal as a sheaf of commutative dgas and whose dga-structure is compatible with the $\rmW$-action.
\vskip .1cm
(ii) The corresponding statements also hold for $R\psi_*(\Q)$ on ${\rmEWT {\underset {\rmWT } \times}\bar \rmT}$, $\A$ on $\rmEW{\underset {{\rmW}} \times} (\ETx \bar \rmT)$ and $R\pi_*(\A)$ on
%\newline \noindent
 $\rmEW{\underset {\rmW} \times} (\bar \rmT/{\rmT})$.
   \end{proposition}
\begin{proof}
\vskip .2cm
{\it Step 4.1.1}.  The two  squares in the diagram ~\eqref{key.comm.diagm} are clearly cartesian. Observe that the map 
\[\psi_n: (\rmE\tilde \rmN{\underset {\tilde \rmN} \times} \bar \rmT)_n = \tilde \rmN^{n-1} \times \bar \rmT \ra (\rmWT)^{n-1} \times \bar \rmT = (\rmEWT{\underset {\rmWT } \times}\bar \rmT)_n\]
is a principal bundle with fiber $(\diagT)^{n-1}$. The same holds for each of the maps $\bar\psi_n$ and $\tilde \psi_n$. 
 Therefore, one has base-change, and one obtains the quasi-isomorphisms:
\be \begin{equation}
\label{base.change.isoms}
\bar \alpha ^*\bar \beta^*R\tilde \psi_*(\Q){\overset {\simeq} \ra} \bar \alpha ^*R\bar\psi_*(\beta ^*\Q) = R\psi_*(\alpha ^*\beta^*(\Q)) = R\psi_*(\Q).
\end{equation} \ee
Moreover, the above isomorphisms are compatible with the structure of sheaves of dg-algebras. Therefore, the statements in (ii) for $R\psi_*(\Q)$
 follow from those in (i). Proof of (i) will occupy the remaining part of Step 4.1.1 and Step 4.1.2.  Steps 4.1.3 and 4.1.4
below provide a detailed proof of the statements in (ii) for $\A$ and $R\pi_*(\A)$.
\vskip .2cm
Next observe from Lemma ~\ref{equiv.struct.maps} 
that each $(\rmBW\diag(T))_n$ has a $\rmW$-action (induced from the action of the bigger group $\rmW diag(T)$) and that the structure maps of the simplicial variety $\rmBW\diag(T)$ are 
compatible with these $\rmW$-actions. Therefore, $\rmW$ acts on the simplicial variety $\rmBW\diag(T)$. In fact, the above action identifies
with the action of $\rmW$ induced by the conjugation action of $\rmW$ on $\rmW\diagT$. \footnote{In particular, this also shows
 that one gets an induced action of $\rmW$ on $\rmB\diagT$, (viewed as a sub-simplicial scheme of $\rmBW\diag(T)$), which is in fact induced by 
 the conjugation action of $\rmW$ on $\diagT$.} Moreover, the simplicial map $\tilde \psi: \rmBW\diag(T) \ra BW$ 
is a $\rmW$-equivariant map in each simplicial degree.  Therefore, $R\tilde \psi_*(\Q) 
= \tilde \psi_*\rmG^{\bullet}(\psi^*(\Q))$ is a complex of $\rmW$-equivariant sheaves on $\rmBW$.
\vskip .2cm
One may also observe that 
that the pairing $\tilde \psi^*(\Q) \boxtimes \tilde \psi^*(\Q) \ra \Delta_*(\tilde \psi^*(\Q))$ is compatible with the $\rmW$-actions, when
$\rmW$-acts diagonally on $\rmBW\diag(T) \times \rmBW\diag(T)$. 
\vskip .2cm

Next the commutative diagram
\[\xymatrix{{\rmBW\diag(T)} \ar@<1ex>[r]^(.35){\Delta} \ar@<1ex>[d] & {{\rmBW\diag(T) \times \rmBW\diag(T)}} \ar@<1ex>[d]\\
            {\rmBW} \ar@<1ex>[r]^{\Delta} &{ \rmBW \times \rmBW},}
\]
and the pairing $\tilde \psi^*(\Q) \boxtimes \tilde \psi^*(\Q) \ra \Delta_*(\tilde \psi_*(\Q))$, (along with the fact that 
the Godement resolution $\rmG^{\bullet}$ is
 functorial in its argument and preserves multiplicative pairings), shows that one obtains the pairings 
\be \begin{equation}
     \label{Rtilde.psi.pairing}
  R\tilde \psi_*(\Q) \boxtimes R\tilde \psi_*(\Q) \ra \Delta_*(R\tilde \psi_*(\Q)),   
    \end{equation} \ee
which are also compatible with the
 $\rmW$-actions when $\rmW$ acts diagonally on the left.
 Therefore, it follows that the dga-structure on $R\tilde \psi_*(\Q)$ is compatible with the $\rmW$-action. Since the maps $\bar \alpha$, $\bar \beta$,
 $\phi$ and $\pi$ are all $\rmW$-equivariant maps, it also follows that the dga-structures on
$R\psi_*(\Q)$, $\A$ and on $R\pi_*(\A)$ are also compatible with the $\rmW$-action.
\vskip .2cm
{\it Step 4.1.2}. Next we show the sheaf of commutative dg-algebras $R\psi_*(\Q)$ is formal. 
In view of the fact that the quasi-isomorphisms in ~\eqref{base.change.isoms} are compatible with the dga-structures, 
it suffices to show that $R\tilde \psi_*(\Q)$ is formal as a sheaf of dgas.
\vskip .2cm
 Let $\H(\rm\rmB diag(T)) = R\tilde \psi_*(\Q)_{{\it x}_o}$ where the right-hand-side denotes the stalk at the base point $x_o$ in $\rmBW$. (Observe that
this identification of the stalks is possible  because we are using the simplicial topology as in ~\ref{sheaves.simp.2} and because $(\rmBW)_0 = \{x_o\}$.)
This is a commutative dga provided with an action by $\rmW$ and whose cohomology is $\rmH^*(\rmB diag(T))$. {\it A key point to observe in view of 
 the equivalence of
 categories in Lemma ~\ref{discrete.grp.action}(iii) is that, in order to show the sheaf of dgas $R\tilde \psi_*(\Q)$ on $\rmBW$ is formal as a sheaf of dgas,
 and compatible with the given $\rmW$-action, it suffices to show that dga at the stalk $\H(\rm\rmB diag(T)) = R\tilde \psi_*(\Q)_{{\it x}_o}$ is formal, and compatible with the given $\rmW$-action.}
 \vskip .1cm
We adopt a rather standard argument to do this: see \cite[(3.5) Lemma]{Lu95}.  Key use is made of the fact that the cohomology of this dga is a polynomial ring generated
by elements in degree $2$. Let $Z^2$ denote the cycles in degree $2$ of the dga $\H(\rmB diag(T))$. Let
$K= d^{-1}(Z^2) \subseteq \H(\rmB diag(T))^1$ and $N= ker(d: K \ra Z^2)$. 
Since $\H(\rmB diag(T))$ has no cohomology in degree $1$, one can find a $\rmW$-stable subring $S$ of $\H(\rmB diag(T))^0$
so that $d(S) = N$. (In fact, let $S = d^{-1}(N)$.) Since the dga $\H(\rmB diag(T))$ has an action by $\rmW$, the differentials of the dga
$\H(\rmB diag(T))$ are compatible with the $\rmW$-action and therefore, $K$, $\rmN$ and $S$ all are stable with
respect to the given action of $\rmW$. 
\vskip .2cm
Now one lets $\B$ denote the free graded commutative (or super-commutative) algebra on
$S \oplus K \oplus Z^2$, with the differential $d:S \ra N \subseteq K$ and $d:K\ra Z^2$ defined to be induced by the
differential of $\H(\rm\rmB diag(T))$. Then the  map sending $S$ ($K$, $Z^2$) to $S$ ($K$, $Z^2$) defines a map
of dg-algebras $\B \ra \H(\rmB diag(T))$ which is a quasi-isomorphism. The quotient of $\B$ by the ideal
generated by $S$, $K$ and $d(K)$ will then map isomorphically to the cohomology algebra $\rmH^*(\rmB diag(T))$. Since
all the objects above are stable with respect to the action of $\rmW$, we observe that that the
quasi-isomorphisms $\H(\rmB diag(T)) \leftarrow \B \ra H^*(\rmB diag(T))$ are compatible with the $\rmW$-actions.
 This completes the proof that $R\tilde \psi_*(\Q)$ is 
formal as a dga on $\rmBW$, and that therefore $R\psi_*(\Q)$ is formal as a sheaf of dgas on ${\rmEWT {\underset {\rmWT } \times}\bar \rmT} $, thereby completing
 the proof of (i).
\vskip .2cm
{\it Step 4.1.3}. Recall that $\A$ is defined as $Rp_{1*}p_2^*(R\psi_*(\Q))$ (using the terminology as in the
proof of Lemma ~\ref{equiv.tech.5}). Next we consider the sheaf of commutative dg-algebras $p_2^*(R\psi_*(\Q))$ on $\rmEWT{\underset {\rmWT } \times} (ET \times \bar \rmT)$ which is
formal. Recall this means, there exists a sheaf of commutative dg-algebras  $K \eps \rmD(\rmEWT{\underset {\rmWT } \times} (ET \times \bar \rmT), \Q)$
so that there exists a diagram $p_2^*(R\psi_*(\Q) \leftarrow K \ra {\mathcal H}^*(p_2^*(R\psi_*(\Q)))$ of sheaves of commutative dgas, 
where the maps are all
quasi-isomorphisms of sheaves of dgas and ${\mathcal H}^*(p_2^*(R\psi_*(\Q)))$ is the cohomology algebra of $p_2^*(R\psi_*(\Q)$. Observe that $Rp_{1*}(K)$ is  a sheaf of commutative
dg-algebras on $\rmEW{\underset {W} \times} (\ETx \bar \rmT)$. 
Next we make use of the observation that the fibers of $p_1$ are all ${\rm {ET}}$ and hence acyclic, so that $R^np_{1*}(\Q) = 0$ for $n \ne 0$ and
$= \Q$ for $n=0$.
Since $R^np_{1*}(\Q)=0$ for all $n >0$, it follows that the natural map $p_{1*}({\mathcal H}^*(p_2^*(R\psi_*(\Q)))) \ra
Rp_{1*}({\mathcal H}^*(p_2^*(R\psi_*(\Q))))$ is a quasi-isomorphism.  It follows that we obtain the diagram of sheaves of commutative dgas on
$\rmEW{\underset {\rmW} \times} (\ETx \bar \rmT)$:
\[ Rp_{1*}(p_2^*(R\psi_*(\Q))) {\overset {\simeq} \leftarrow} Rp_{1*}(K) {\overset {\simeq} \rightarrow} Rp_{1*}({\mathcal H}^*(p_2^*(R\psi_*(\Q))))
 {\overset {\simeq} \leftarrow} p_{1*}({\mathcal H}^*(p_2^*(R\psi_*(\Q))))
\]
Clearly the last sheaf of dgas is formal. This completes the proof of the statement that the sheaf of commutative dgas $\A= Rp_{1*}(p_2^*(R\psi_*(\Q))) $
on $\rmEW{\underset {\rmW} \times} (\ETx \bar \rmT)$ is formal. Since $\A$ is formal it is quasi-isomorphic as a sheaf of commutative dgas to
its cohomology algebra, which is ${\mathcal H}^*(\A)$, which is the constant sheaf associated to $H^*(\rmB diag(T), \Q)$. 
\vskip .2cm
{\it Step 4.1.4}.  Since $\bar \rmT$ is a projective
toric variety for the torus $\rmT$, the arguments in \cite[Theorem 3.1]{Lu95} or \cite[section 5]{Gu05} apply verbatim to prove that
$R\pi_{*}({\mathcal H}^*(\A))$ is formal as a sheaf of commutative dgas. 
In fact $R\pi_{.*}({\mathcal H}^*(\A)) = R\pi_{*}(\Q) \otimes _{\Q}{\mathcal H}^*(\bar \A)$, where $\bar \A$ is defined
in the line following ~\eqref{A.2}. Therefore, it suffices to observe that the arguments in \cite[Theorem 3.1]{Lu95} or \cite[section 5]{Gu05} are compatible with the action of $\rmW$. 
Recall these arguments are essentially a sheafified variant of the arguments discussed above in Step 2, sheafified
on the space $\bar \rmT/{\rmT}$. As discussed in ~\eqref{contract.slices.WT.1}, each $\rmT$-orbit $\rmY$ on $\bar \rmT$ has an open $\rmT$-stable neighborhood $\rmU_{\rmY}$ of which
$\rmY$ is a deformation retract and the $\rmW$-action sends the $\rmT$-orbit $\rmY$ to another $\rmT$-orbit $w \rmY$ and
the neighborhood $\rmU_{\rmY}$ to a neighborhood $\rmU_w \rmY$ of $w\rmY$. The sheafified variant of the argument in Step 2 
is carried out by applying the constructions in Step 2 to the sheaf of dg-algebras $R\pi_*(\Q_{|\rmU_{\rmY}})$, which produces
a free graded commutative dga $\B_{\rmU_{\rmY}}$ on the image of $\rmU_{\rmY}$ in $\bar \rmT$ and then by showing that for
a smaller open neighborhood $\rmU_{\rmY '} \subseteq \rmU_{\rmY}$ associated to another $\rmT$-orbit ${\rmY '}$, the
dg-algebra $\B_{\rmU_{\rmY}}$ restricts to $\B_{\rmU_{\rmY '}}$. Since the $\rmW$-action preserves the type of the 
$\rmT$-orbits, it preserves the closure relations among these $\rmT$-orbits, sending the neighborhood
$\rmU_{\rmY '}$ to $\rmU_{w{\rmY '}}$ and $\rmU_{\rmY}$ to $\rmU_{w \rmY}$. Therefore it is clear that the above restrictions
are compatible with the $\rmW$-action. This completes the proof of the proposition.
\end{proof}
\vskip .2cm
\noindent
{\bf Step 4.2}.
As observed in Proposition ~\ref{formal.1}, $R\psi_*(\Q)$ is obtained as the pull-back $\bar \alpha^* (\bar \beta ^*(R\tilde \psi_*(\Q))$.
 Since $R\tilde \psi_*(\Q))$ is formal, we may in fact replace
this by its cohomology. Therefore, by applying the projection formula to ~\eqref{A.1}, one may also observe that 
\be \begin{equation}
     \label{A.2}
R\pi_*(\A) = \phi^* (R\tilde \psi_*(\Q))) \otimes _{\Q} R \pi_*(\Q).
    \end{equation} \ee
\vskip .1cm 
We will denote 
\be \begin{equation}
     \label{A.3}
\phi^* (R\tilde \psi_*(\Q))) \mbox{ by }\bar \A.
\end{equation} \ee
\begin{lemma}
 \label{correspondence.1}
$R\pi_*(\rmIC^{\rmW, \rmT}(\O_{\rmW,T }) \otimes \A) = R\pi_*(\rmIC^{\rmW, \rmT})(\O_{\rmW,T }) \otimes  \bar \A$.
\end{lemma}
\begin{proof} Observe that $\A$ being formal, we may replace $\A$ by its cohomology sheaves. The arguments in  ~\eqref{A.1} show that
$\A$ is also constant, that is, we may replace $\A$ by a graded $\Q$-vector space. Then the conclusion follows readily.
\end{proof}
\subsubsection{}
\label{ident.dg.algs} 
Recall that we showed in Step 4 (see ~\ref{step4}), that the dgas $\B_{\rmG \times \rmG}(\bar \rmG)$ and ${ {\B}}_{\rmW, \rmT}(\bar \rmT)$ 
are quasi-isomorphic, where the dga ${ {\B}}_{\rmW, \rmT}(\bar \rmT)$ 
is defined in Step 4 (iii) as 
\[\RHom(\oplus _{\O}R\pi_*(\rmIC^{\rmW, \rmT})(\O_{\rmW,T } \otimes \A), \oplus _{\O}R\pi_*(\rmIC^{\rmW, \rmT})(\O_{\rmW,T } \otimes \A).\]
In view of Lemma ~\ref{correspondence.1}, this identifies with 
\[\RHom(\oplus _{\O}R\pi_*(\rmIC^{\rmW, \rmT})(\O_{\rmW,T }) \otimes  \bar \A, \oplus _{\O}R\pi_*(\rmIC^{\rmW, \rmT})(\O_{\rmW,T }) \otimes  \bar \A).\]
\vskip .2cm \noindent
{\bf Step 4.3}. {\it Formality of the dg-algebra ${ {\B}}_{\rmW, \rmT}(\bar \rmT)$}. In order to establish this formality,  we consider projective resolutions
 in Lemma ~\ref{projectives} and Proposition ~\ref{proj.res}, discussed below.
\begin{lemma}
 \label{projectives} Let $\rmC^{+,o}_{cart,c}(\rmEW{\underset \rmW \times} (\ETx \bar \rmT), \A)$ 
 ($\rmC^{+,o}_{cart,c}(\rmEW{\underset \rmW \times} ( \bar \rmT/{\rmT}), R\pi_*(\A))$ denote the category of
 complexes whose associated derived category is $\rmD^{+,o}_{cart,c}(\rmEW{\underset \rmW \times} (\ETx \bar \rmT), \A)$ 
 ($\rmD^{+,o}_{cart,c}(\rmEW{\underset \rmW \times} ( \bar \rmT/{\rmT}), R\pi_*(\A))$, \res). 
 \vskip .2cm 
(i) Let ${\rmY '}$ denote a $\rmT$-orbit on $\bar \rmT$. Let $\rmW_{{\rmY '}}$ denote the stabilizer of ${\rmY '}$ in $\rmW$. Then
\newline \noindent
$P=\oplus _{w \eps \rmW/\rmW_{{\rmY '}}} j_{w{\rmY '}!}j_{w{\rmY '}}^*(\A)$ is a projective object in $\rmC^{+,o}_{cart,c}(\rmEW{\underset \rmW \times} (\ETx \bar \rmT), \A)$ in the
sense that 
\newline \noindent
$Hom_{\A}(P, \quad)$ preserves quasi-isomorphisms in the second argument, where $Hom_{\A}$ denotes the Hom in the category 
$\rmC^{+,o}_{cart,c}(\rmEW{\underset \rmW \times} (\ETx \bar \rmT), \A)$ of
sheaves of $dg$-modules over $\A$. Similarly,
\vskip .1cm 
$Q=\oplus _{w \eps \rmW/\rmW_{{\rmY '}}} j_{w{\rmY '}/{\rmT}!}j_{w{\rmY '}/{\rmT}}^*(R\pi_*(\A))$ is a projective object in $\rmC^{+,o}_{cart,c}(\rmEW{\underset \rmW \times} ( \bar \rmT/{\rmT}), 
R\pi_*(\A))$.
\vskip .2cm
(ii) Every sheaf of dg-modules over ${\mathcal H}^*(\A)$ has a bounded above resolution by projectives as in (i). The same holds for
sheaves of dg-modules over ${\mathcal H}^*(R\pi_*(\A))$.
\end{lemma}
\begin{proof}
In view of the equivalence of categories provided by Proposition ~\ref{Rpi*.and.Lpi*} and Remark ~\ref{funct.resols}, the first assertion in (i) follows from the second assertion in (i).
For the second assertion in (i),  we simply observe that 
\[Hom_{R\pi_*(\A)}(\oplus _{w \eps \rmW/\rmW_{{\rmY '}}} j_{w{\rmY '}/{\rmT}!}j_{w{\rmY '}/{\rmT}}^*(R\pi_*(\A)), K) = \oplus_{w \eps \rmW/\rmW_{{\rmY '}/{\rmT}}} \Hom_{R\pi_*(A)}(j_{w{\rmY '}/{\rmT}}^*(R\pi_*(\A)), j_{w{\rmY '}/{\rmT}}^*(K)). \]
But $j_{w{\rmY '}/{\rmT}}^*(R\pi_*(\A))$ ($j_{w{\rmY '}/{\rmT}}^*(K)$) is the stalk of $R\pi_*(\A)$ ($K$, \res) at the orbit $w{\rmY '}$. This proves the second assertion in (i).  We skip the proof of the assertions in (ii). 
\end{proof}
\vskip .2cm
Let $\O_{\rmWT }$ denote a 
$\rmWT$-orbit on $\bar \rmT$. Then $\O_{\rmWT }$ is a disjoint union of $\rmT$-orbits on $\bar \rmT$ permuted under the action of $\rmW$: therefore,
 we  denote $\O_{\rmWT } = \rmW.\O_{\rmT}$. Observe also that the dg-algebra $R\pi_*(\A)$ being formal can be viewed as a dg-algebra over 
$\rmH^*(BT)$. Recall that $\rmIC^{\rmW, \rmT}(\O_{\rmWT })$ denotes the equivariant intersection cohomology complex
on the $\rmWT$-orbit $\O_{\rmWT}$ for the action of the group $\rmWT$ (and extending the constant sheaf $\Q$ on $\O_{\rmWT}$), but viewed as a complex on
$\rmEW{\underset \rmW \times} (\ETx \bar \rmT)$
\begin{proposition}
 \label{proj.res}
(i)  Let $x \eps \bar \rmT^{\rmT}$ denote a fixed point. Let $\rmW_{\it x}$ denote
the stabilizer of ${\it x}$ in $\rmW$ and $\rmW {\it x}= \rmW/\rmW_{\it x}$ the corresponding $\rmW$-orbit. Then
denoting the sum $\oplus_{w \eps W/W_{\it x}} R\pi_*(\rmIC^{\rmW, \rmT}(\O_{\rmWT })\otimes \A)_{{\it wx}}$ by $R\pi_*(\rmIC^{\rmW, \rmT}(\O_{\rmWT })\otimes \A)_{\rmW {\it x}}$, we obtain:
\be \begin{equation}
\label{stalk.ident.2}
R\pi_*(\rmIC^{\rmW, \rmT}(\O_{\rmWT }) \otimes \A)_{\rmW {\it x}} = \oplus_{ w \eps \rmW /\rmW_{\it x}} R\pi_*(\rmW \rmIC^{\rmT}(\O_{\rmT}))_{{\it wx}} \otimes _{\Q} {\bar \A}
% &= \oplus_{ w \eps \rmW /\rmW_{\it x}} \rmW \rmIC^{\rmT}(\O_{\rmT})_{{\it wx}} \otimes _{\Q} R\pi_*(\Q)_{{\it wx}} \otimes_{\Q} \bar \A \notag
\end{equation} \ee
\vskip .2cm \noindent
where $\rmW\rmIC^{\rmT}(\O_{\rmT})$ denotes a sum of intersection cohomology complexes $\oplus \rmIC^{\rmT}(\O_{\rmT})$ where the sum varies over
all the disjoint $\rmW$-translates of a given $\rmT$-orbit $\O_{\rmT}$. Therefore, the cohomology 
\newline \noindent
$\rmH^*(R\pi_*(\rmIC^{\rmW, \rmT}(\O_{\rmWT })) \otimes \bar \A)_{\rmW {\it x}}$ forms
a projective module over $\rmH^* (\rmB \rmT ) \otimes \rmH^*(\bar \A)$.
\vskip .2cm
(ii) Every object $R\pi_*(\rmIC^{\rmW, \rmT}(\O_{\rmWT })) \otimes \bar \A$ has a projective resolution $\{\cdots \ra P^{-n} \ra P^{-n+1} \ra \cdots P^0\}$
in $\rmD_W(\bar \rmT/{\rmT}, H^*(R\pi_*(\Q) \otimes \bar \A))$ so that the given augmentation $P^0 \ra R\pi_*(\rmIC^{\rmW, \rmT}(\O_{\rmWT })) \otimes \bar \A$ is a quasi-isomorphism
at each stalk of the form $\rmW {\it x}$, $x \eps \bar \rmT^{\rmT}$, it induces a surjection at each stalk $Wx$, $x \eps \bar \rmT$ and each $P^{-i}$ is of the form 
\[\oplus_{ w \eps \rmW /\rmW_{\itx}}j_{\rmU_{{\it wx}}!}j_{\rmU_{{\it wx}}}^*(H^*(R\pi_*(\Q) \otimes \bar \A))[n_{\rmU_{\it wx}}] \]
as $\rmU_{\it wx}$ varies over neighborhoods of points  $wx \eps \bar \rmT/{\rmT}$ and where $n_{U_{\it wx}}$ are integers.
\vskip .2cm
(iii)  The  complexes $P^{i}$, for $i<0$, are supported at points in $\bar \rmT -\bar \rmT^{\rmT}$ and hence
viewed as modules over $\H^*(\rmB \rmT, \Q) \otimes _{\Q} \rmH^*(\bar \A)$  are torsion.
\end{proposition}
\begin{proof}
 (i) is a straight-forward calculation making use Proposition ~\ref{correspondence.1}(iii) and  the following observation. If 
$\O_{\rmT}$ denotes an orbit of $\rmT$ on $\bar \rmT$, and $\rmW.\O_{\rmT}  =\O_{\rmWT }$ denotes the orbit for the corresponding 
$\rmWT$-action,
then as observed in ~\eqref{ICWT.1},
\be \begin{equation}
   \label{decompICWT}
 \rmIC^{\rmW, \rmT}(\O_{\rmWT }) = \oplus_{w \eps \rmW/\rmW_{\O_T}} \rmIC^{\rmT}(\O_{\rmT}), 
\end{equation} \ee 
\newline \noindent
 where $\rmW_{\O_T}$ denotes the stabilizer of the $\O_T$-orbit in 
$\rmW$. Applying $R\pi_*$, and making use of ~\eqref{A.1}, we therefore obtain:
\[R\pi_*(\rmIC^{\rmW, \rmT}(\O_{\rmWT}) \otimes \A) = R\pi_*(\rmW\rmIC^{\rmT}(\O_{\T})) \otimes_{\Q} \bar \A.\]
Then take the stalks at $wx$ and the sum $\oplus_{ w \eps \rmW /\rmW_{\it x}}$  to obtain ~\eqref{stalk.ident.2}.
\vskip .2cm 
Next use the fact that the
global equivariant intersection cohomology of a projective toric variety is a free module over the cohomology
ring of the classifying space of the torus, and the stalk cohomology of the intersection cohomology complex at a
$\rmT$-fixed point on $\bar \rmT$ is isomorphic to the global intersection cohomology of the link at that point: see \cite[(4.0.4) Theorem and 
 (4.2.2)]{Lu95} and also \cite[Theorem 1.1]{BJ04}. This completes the proof of (i).
\vskip .2cm
 Then (ii) is an immediate consequence of (i), Lemma ~\ref{projectives} and the formality of the dgas $R\pi_*(\A) = R\pi_*(\Q) \otimes \bar \A$. 
 (ii) shows that the augmentation $P^0 \ra R\pi_*(\rmIC^{\rmW, \rmT}(\O_{\rmWT })) \otimes \bar \A$
 induces a quasi-isomorphism at every $\rmW {\it x}$, 
for every  $x \eps \bar \rmT^{\rmT}$.
Therefore, it follows that each ${\mathcal H}^i(P^{\bullet})$, for $i<0$, are torsion modules over $\rmH^*(BT, \Q) \otimes _{\Q} \rmH^*(\bar \A)$. 
\end{proof}
Recall that $\bar \rmT$ is the closure in the projective variety $\bar \rmG$ of the torus $\rmT$: therefore, $\bar \rmT$ is projective.
Next recall from ~\eqref{A.1} that $\phi$ denotes the map $\phi: \rmEW{\underset W \times} (\bar \rmT/{\rmT} )\ra \rmBW$ sending $\bar \rmT/\rmT$ to the base point in $\rmBW$.
\begin{lemma}
   \label{key.stalk.id}
The stalk of 
\[R\phi_*(\RHom_{R\pi_*(\Q)}(\oplus R\pi_*(\rmIC^{\rmW, \rmT}(\O_{\rmWT }), \oplus R\pi_*(\rmIC^{\rmW, \rmT}(\O_{\rmWT } ) )  ))\] 
at the base point in $\rmBW$, where the sum varies over all the $\rmWT$-orbits in $\bar \rmT$,  identifies with the dg-algebra $\B_{\rmT}(\bar \rmT)$, which is the dg-algebra
 considered in Theorem ~\ref{mainthm.2} for the projective toric variety $\bar \rmT$, viewed as an imbedding of the torus $\rmT$. In particular,
 its cohomology is a torsion-free module over $H^*(\rmB\rmT, \Q)$.
\end{lemma}
\begin{proof}
 The decomposition of the $\rmWT$-orbits into the corresponding $\rmT$-orbits and therefore, the induced
 decomposition as in ~\eqref{decompICWT} shows that the sum $\oplus R\pi_*(\rmIC^{\rmW, \rmT}(\O_{\rmWT })$  in fact, runs over all the $\rmT$-equivariant 
 intersection cohomology complexes on the $\rmT$-orbits on $\bar \rmT$. Then the stalk considered
 above (using the simplicial topology as in ~\ref{sheaves.simp.2}) identifies with the dg-algebra $RHom_{R\pi_{o*}(\Q)}(\oplus R\pi_{o*}(\rmIC^{\rmT}(\O_{\rmT })), \oplus R\pi_{o*}(\rmIC^{\rmT}(\O_{\rmT } )))$,
 where $\pi_o: \ETx \bar \rmT \ra \bar \rmT /\rmT$ is the quotient map. By \cite[the proof of (0.1.1) Theorem]{Lu95} this
 identifies with the dg-algebra $\B_{\rmT }(\bar \rmT)$. The last statement is then deduced as in
 \cite[(4.0.3) Theorem]{Lu95} from the 
 observation that the global equivariant intersection cohomology complex of a projective toric variety is a free 
 module over the cohomology ring of the classifying space of the torus as observed earlier. One applies this observation to
 the projective toric varieties forming the closures of the  $\rmT$-orbits on $\bar \rmT$.
\end{proof}
\begin{remark} In keeping with the terminology used in Theorem ~\ref{mainthm.2}, the dg -algebra $\B_{\rmT}(\bar \rmT)$ should be denoted $\B_{\rmT \times \rmT}(\bar \rmT)$. However, since in this case $\rmT$ is commutative, we adopt the simpler notation as above.
\end{remark}
\begin{theorem}
 \label{main.thm.5} Assume next that the toroidal imbedding $\bar \rmG$ and therefore, the toric variety $\bar \rmT$ is projective. 
 Then (i) the dg algebra ${ {\B}}_{\rmW, \rmT}(\bar \rmT)$ is formal and (ii)   $\B_{\rmW, \rmT}(\bar \rmT) = (\B_{\rmT }(\bar \rmT) \otimes R\tilde\psi _*(\Q))^{\rmW}$.
\end{theorem}
\begin{proof}
Let $P^{\bullet} \ra \oplus  R\pi_*(\rmIC^{\rmW, \rmT}(\O_{\rmWT })) \otimes \bar \A$ denote a projective resolution, where each $P^{-i}$ is a sum of terms of the
form as in Proposition ~\ref{proj.res}(ii) and where the sum varies over all
the $\rmWT$-orbits in $\bar \rmT$. Therefore $P = \oplus_iP^{-i}[i] \ra \oplus  R\pi_*(\rmIC^{\rmW, \rmT}(\O_{\rmWT })) \otimes \bar \A$
is a quasi-isomorphism with $P$ a projective object in $\rmC_W(\bar \rmT/{\rmT}, H^*(R\pi_*(\Q) \otimes \bar \A))$.
\vskip .1cm
Observe next that the dga $\bar \A$ is formal, and since each $P^{-i}$ is of the form given in
Proposition ~\ref{proj.res}(ii),  the differentials of each $P^{-i}$ are in fact trivial. Therefore, 
 the spectral sequence for the total complex for ${\mathcal H}om(P^{\bullet}, P^{\bullet})$ degenerates. Therefore, now ${ \B}_{\rmW, \rmT}(\bar \rmT)$ identifies with the total complex of $Hom(P^{\bullet}, P^{\bullet})$, where the differentials
of $P^{\bullet}$ (that is, $\{d:P^n \ra P^{n+1}|n\}$) provide the structure of a 
chain-complex on $Hom(P^{\bullet}, P^{\bullet})$. 
\vskip .2cm
We will presently provide two somewhat different proofs to show that the dg algebra ${ \B}_{\rmW, \rmT}(\bar \rmT)$ is formal.
The first starts with the  observation that the dg-algebra ${ \B}_{\rmW, \rmT}(\bar \rmT)= R\Gamma (\rmBW, \widetilde \B_{\rmW, \rmT}(\bar \rmT)) = 
(\widetilde \B_{\rmW, \rmT}(\bar \rmT))^W$ for the sheaf of dgas $\widetilde \B_{\rmW, \rmT}(\bar \rmT) = {\mathcal H}om(P^{\bullet}, P^{\bullet})$ on $\rmBW$. 
\vskip .2cm
A {\it key observation} we make is that the stalk of ${\mathcal H}^*(\widetilde \B_{\rmW, \rmT}(\bar \rmT))$ (at the base point in $\rmBW$) 
is a torsion-free
module over $\rmH^*(\rmB \rmT, \Q) \otimes_{\Q} \rmH^*(R\tilde \psi_*(\Q))$. When $\bar \rmG$ and hence $\bar \rmT$ are projective, this is clear in view of the identification 
of the stalks of 
\[R\phi_*(\RHom_{R\pi_*(\Q)}(\oplus R\pi_*(\rmIC^{\rmW, \rmT}(\O_{\rmWT })), \oplus R\pi_*(\rmIC^{\rmW, \rmT}(\O_{\rmWT } ) )  ))\] 
at the
 base point in $\rmBW$ with the dg-algebra $\B_{\rmT}(\bar \rmT)$  as shown in Lemma ~\ref{key.stalk.id}, along with the
  observation that $\bar \A= \phi^*(R\tilde \psi_*(\Q))$. (In fact, see ~\eqref{BWT} below.)
  Therefore, Proposition ~\ref{proj.res} shows that the cohomology of the sheaf of dg-algebras
$\widetilde \B_{\rmW, \rmT}(\bar \rmT) $, which identifies with the cohomology of the complex ${\mathcal H}om(P^{\bullet}, P^{\bullet})$, vanishes
in all degrees except $ 0$. (Here, we again make use of the observation that the dga $\bar \A$ is formal, and since each $P^{-i}$ is of the form given in
Proposition ~\ref{proj.res}(ii),  the differentials of each $P^{-i}$ are in fact trivial.)
\vskip .1cm
Now the diagram of dgas
\[{\mathcal H}^0(\widetilde \B_{\rmW, \rmT}(\bar \rmT)) \leftarrow \sigma_{\le 0}(\widetilde \B_{\rmW, \rmT}(\bar \rmT)) \ra \widetilde \B_{\rmW, \rmT}(\bar \rmT)\]
(where $\sigma_{\le 0}$ is the functor that kills the above cohomology in negative degrees) shows that the sheaf of 
dg-algebras $\widetilde \B_{\rmW, \rmT}(\bar \rmT)$ is formal on $\rmBW$. Finally one simply takes $\rmW$-invariants, observing that taking $\rmW$-invariants is
an exact functor since we are working with rational coefficients. This proves (i).
\vskip .2cm
We will next provide a proof (ii). Now one may observe that with $\bar \A= \phi^*(R\tilde \psi_*(\Q))$,
\be \begin{align}
     \label{BWT}
{\B}_{\rmW, \rmT}(\bar \rmT) &= R\Gamma (\rmBW, R\phi_*(\RHom_{R\pi_*(\Q) \otimes \bar \A }(\oplus R\pi_*(\rmIC^{\rmW, \rmT}(\O_{\rmWT }))\otimes _{\Q}\bar \A, \oplus R\pi_*(\rmIC^{\rmW, \rmT}(\O_{\rmWT })) \otimes _{\Q}\bar \A )))\\
         &= R\Gamma (\rmBW, R\phi_*(\RHom_{R\pi_*(\Q)}(\oplus R\pi_*(\rmIC^{\rmW, \rmT}(\O_{\rmWT })), \oplus R\pi_*(\rmIC^{\rmW, \rmT}(\O_{\rmWT })) \otimes_{\Q} \bar \A ))) \notag\\
         &= R\Gamma (\rmBW, R\phi_*(\RHom_{R\pi_*(\Q)}(\oplus R\pi_*(\rmIC^{\rmW, \rmT}(\O_{\rmWT }) ), \oplus R\pi_*(\rmIC^{\rmW, \rmT}(\O_{\rmWT } ) ) ) \otimes_{\Q} R\tilde \psi_*(\Q))). \notag
    \end{align} \ee
The last equality follows from the projection formula as well as the observation that $\bar \A$ is formal and constant (see 
~\eqref{A.2}), while the one above follows from standard properties of $\RHom$.
Here the sum $\oplus$ varies over all $\rmWT$-orbits in $\bar \rmT$.
Now observe that 
\be \begin{equation}
     \label{toric.stalks}
R\phi_*(\RHom_{R\pi_*(\Q)}(\oplus R\pi_*(\rmIC^{\rmW, \rmT}(\O_{\rmWT }) ), \oplus R\pi_*(\rmIC^{\rmW, \rmT}(\O_{\rmWT } ) ) ))
\end{equation} \ee
 is a complex of $\rmW$-equivariant  
sheaves on $\rmBW$. Lemma ~\ref{key.stalk.id} shows that  its stalk at the base point identifies with the dga 
$\B_{\rmT}(\bar \rmT)$ associated to $\bar \rmT$, viewed as a toric variety for $\rmT$. Therefore, we obtain (ii):
\[\mbox{ that is, the dga }  \B_{\rmW, \rmT}(\bar \rmT) = (\B_{\rmT }(\bar \rmT) \otimes R\tilde\psi _*(\Q))^W.\]
\vskip .2cm
Finally we will provide a second proof of (i). 
Since $R\tilde \psi_*(\Q)$ was already observed to be formal as a dga on $\rmBW$, (that is, as a dga with $\rmW$-action), it suffices to show that 
$ \B_{\rmT}(\bar \rmT)$ is formal as a dga with $\rmW$-action. The arguments for the formality of $\B_{\rmT }(\bar \rmT)$ (see for example, the proof of \cite[Proposition 4.1.2]{Lu95}) show that $H^i(\B_{\rmT}(\bar \rmT)) =0$
for all $i \ne 0$, where $H^i$ denotes the cohomology of the corresponding complex $\Hom(P^{\bullet}, P^{\bullet})$.
\vskip .2cm
Next let $\sigma_{\le 0}  \B_{\rmT }(\bar \rmT)$ denote the truncation functor that kills the cohomology above degree $0$. Then it is clear that we
obtain maps of  dgas
\[\rmH^0( \B_{ \rmT }(\bar \rmT)) \leftarrow \sigma_{\le 0}( \B_{\rmT }(\bar \rmT)) \ra  \B_{ \rmT }(\bar \rmT)\]
which are both quasi-isomorphisms and compatible with the given $\rmW$-actions. (The compatibility with the
$\rmW$-action should be clear from ~\eqref{decompICWT}, which shows that the $\rmT$-orbits on $\bar \rmT$ may be
grouped into $\rmWT$-orbits.) Therefore, the induced maps of dgas
\[\rmH^0( \B_{ \rmT }(\bar \rmT)) \otimes_{\Q} \rmH^*(R\tilde \psi_*(\Q))\leftarrow \sigma_{\le 0}( \B_{\rmT }(\bar \rmT))\otimes_{\Q} \rmH^*(R\tilde \psi_*(\Q)) \ra \B_{ \rmT }(\bar \rmT)\otimes_{\Q} H^*(R\tilde \psi_*(\Q))\]
are also quasi-isomorphisms and compatible with the $\rmW$-actions. Since $R\tilde \psi_*(\Q)$ was observed to be formal as a dga on $\rmBW$, one may now
replace $\rmH^*(R\tilde \psi_*(\Q))$ by $R\tilde \psi_*(\Q))$ and obtain the same conclusions. Therefore, we see that ${ {\B}}_{\rmW, \rmT}(\bar \rmT) = ( \B_{\rmT } (\bar \rmT)\otimes_{\Q} R\tilde \psi_*(\Q))^{\rmW}$
is formal as a dga.
\end{proof}
In view of the arguments in ~\ref{ident.dg.algs}, this completes the proof of Theorem ~\ref{mainthm.2}.

\section{\bf A General  Obstruction Theory for Formality of the Dg-algebra $\B$ and Conclusions for scs varieties}
We start with the observation that the dg-algebra $\B_{\rmG}(\rmX)$ in Theorem ~\ref{mainthm.1} is only an $A_{\infty}$-dg-algebra (or only an associative dg-algebra), because the
multiplicative structure is given by composition and hence not commutative in general. (Recall that an $A_{\infty}$-dg algebra means
a dg-algebra where the multiplication is coherently homotopy associative.)
A well-known result of Kadeishvili (see \cite{Ka80})
shows that, nevertheless for any $A_{\infty}$ dg-algebra $\B$ (over $\Q$) there exists an $A_{\infty}$-structure on the cohomology algebra $H^*(\B)$, so
 that $\B$ is quasi-isomorphic as an $A_{\infty}$-dg-algebra to $H^*(\B)$ with the above $A^{\infty}$-structure.
We follow the exposition of \cite[section 2]{LPWZ09} for this. Since we are working over $\Q$, we obtain a (non-canonical) decomposition
of each $\Q$-vector space $\B^n$ as
\begin{equation}
\label{obs.1}
\B^n=Z^n\oplus L^n=C^n\oplus H^n\oplus L^n.
\end{equation}
Here $C^n$ denotes the co-boundaries, $Z^n$ the co-cycles and $H^n = Z^n/C^n$. We will identify $\oplus_{n} H^n(\B)$ as imbedded in  $\B$ 
by co-cycle-sections $H^n\subset \B^n$. {\it Clearly there are many different 
choices of $H^n$ and $L^n$.} Let $p: \B \to \B$ be a projection to $H:=\oplus_n H^n$, and 
let $G: \B \to \B$ be a homotopy from $id_{\B}$ to $p$. Hence 
we have $id_{\B}-p=\partial G+G\partial$.  We will define the map $\rmG$  as follows:
for every $n$, $G^n: \B^n\to \B^{n-1}$ is the map which satisfies
\begin{itemize}
\item 
$G^n=0$ when restricted to $L^n$ and $H^n$, and 
\item
$G^n=(\partial^{n-1}|_{L^{n-1}})^{-1}$ when restricted to $C^n$.
\end{itemize}
(Observe that $\partial^{n-1}_{|L^{n-1}}: L^{n-1} \ra C^n$ is a bijection.) Therefore, the image of $G^n$ is $L^{n-1}$. 
It follows that $G^{n+1}\partial^n=Pr_{L^n}$ and 
$\partial^{n-1}G^n=Pr_{C^n}$, where $Pr_{L^n}: \B^n \ra L^n$ and $Pr_{C^n}: \B^n \ra C^n$ are the projections.
\vskip .2cm
Next we define a sequence of linear maps $\lambda_n: \B^{\otimes n}\to \B$ of
degree $2-n$ as follows.  There is no map $\lambda_{1}$, but we
formally define the ``composite'' $G \lambda_{1}$ by $G \lambda_{1} =
-id_{\B}$.  $\lambda_2$ is the multiplication of $\B$, namely, 
$\lambda_2(a_1\otimes a_2)=a_1\cdot a_2$.
For $n\geq 3$,
$\lambda_n$ is defined by the recursive formula
\begin{equation}
\label{obs.2}
\lambda_n=\sum_{\substack{s+t=n, \\ s,t\geq 1}} (-1)^{s+1} 
\lambda_2 \circ(G\lambda_s\otimes G\lambda_t).
\end{equation}
Using $p$ to denote both the map $\B
\rightarrow \B$ and also (since the image of $p$ is $H^*(\B)$) the map
$\B \rightarrow H^*(\B)$; we also use $\lambda_{i}$ both for the map
$\B^{\otimes i} \rightarrow \B$ and for its restriction $(H^*(\B))^{\otimes
i} \rightarrow H^*(\B) \ra \B$. Then the above mentioned result of Kadeishvili can be stated as follows:

\begin{theorem} (See \cite[Theorem 2.2, Proposition 2.3]{LPWZ09}.)
\label{thm2.2} 
(i) Let $m_i=p\lambda_i$. Then $(H^*(\B), m_2,m_3, \cdots)$ is an 
$A_\infty$-algebra.
\vskip .2cm
(ii) Let $\{\lambda_n\}$ be defined as above. For $i \geq 1$ let $f_i=-G 
\lambda_i: (H^*(\B))^{\otimes i}\to \B$. Then $f:=\{f_i\}$ is a map of $A_{\infty}$ dg-algebras so that $f_1$ is a quasi-isomorphism.
\end{theorem}

We define the $A_{\infty}$-dg-algebra $\B$ to be  {\it formal}, if all the $m_i=0$ for $i>2$. Now we 
conclude this discussion with the 
%\vfill \eject \noindent
\vskip .2cm \noindent
{\bf Proof of Theorem ~\ref{mainthm.4}}.
\vskip .2cm \noindent
 (i) follows readily from the fact that the $m_i =p \lambda _i$ and that the map $\lambda _i : \B^{\otimes i} \ra \B$ has degree $2-i$ which is
odd if $i$ is odd. Now the map $m_i$ is the composition 
$H^*(\B) ^{\otimes i} \ra \B^{\otimes i} {\overset {\lambda _i} \ra} \B {\overset p \ra} \H^*(\B)$. Therefore, when $H^i(\B)=0$ for all odd $i$,
one may readily see that the  map $m_i$ is  trivial for all odd $i$.
\vskip .2cm \noindent
(ii) Let $\rmX$ denote a projective $\rmG$-spherical variety as in (ii). Then it was shown in \cite[Lemma 3.6]{BJ04} and \cite[Theorem 4]{BJ01} that the only $\rmG$-equivariant local systems on
the $\rmG$-orbits $\O$ on $\rmX$ are constant, that the odd dimensional intersection cohomology sheaves ${\mathcal H}^i(\rmIC(\Q_{| \O}))$, as well as 
the odd dimensional intersection cohomology groups ${\rm {IH}}^i(\bar \O)$ vanish. (Here
$\Q_{|\O}$ denotes the constant sheaf $\Q$ on the $\rmG$-orbit $\O$ in $\rmX$.) A spectral sequence argument now shows that $H^i(\B_{\rmG}(\rmX))=0$ for all odd $i$, for
the corresponding dg-algebra $\B_{\rmG}(\rmX)$. Therefore, the result in (i) applies. \qed
\begin{remark}
 One may interpret the conclusion of the last theorem as saying that {\it half the obstructions for the formality of the dg-algebra $\B_{\rmG}$ vanish}
for the spherical varieties considered there. Clearly, this includes a large class of spherical varieties. 
\end{remark}

\section{\bf  Comparison of Equivariant Derived Categories}
\label{equiv.der.cats}
The equivariant derived categories associated to the action of a group $\rmG$ on a space are usually defined as certain full subcategories of the
derived category on the Borel construction associated to the group action.  Different models for the Borel construction, therefore provide different models for the equivariant 
derived categories. 
The geometric model which had been introduced in \cite{BL94} in the topological framework and in the scheme-theoretic framework in \cite{To99} and \cite{MV99} complements the simplicial model
which was discussed in \cite{De74},  \cite{Fr83} and \cite{Jo93}, each with its own advantages and dis-advantages. For example, the discussion of
the Weyl-group action in section 3 (see Lemma ~\ref{equiv.tech.5}) as well as the need to handle modules over sheaves of dgas that are only bounded below,
  seems to require the simplicial model. In fact the situation considered in 
Proposition ~\ref{Rpsi*} and  Lemma ~\ref{equiv.tech.5} is a special case of a more general situation where one needs to consider actions of two groups $\rmG$ and $\rmH$ provided with
a surjective map from $\rmG$ to $\rmH$ acting on spaces compatibly and then relate the corresponding equivariant derived categories. In this case, the
simplicial model seems to be able to handle the situation easily, while the approach making use of 
geometric classifying spaces simply does not work, since one has to find representations of both $\rmG$ and $\rmH$
and relate them suitably.
We encounter several instances of such situations in section 2, which is the reason we have chosen to work with the simplicial model. 
\vskip .2cm
However, since the geometric model  is perhaps more suited for
handling properties like the weight filtration (used in \cite{Jo17}), and also more commonly used in the literature dealing with equivariant derived categories, 
we felt it important to provide a comparison between the two models, which is what is done in the rest of this section. 
   In view of the various applications, we have decided to make the discussion in 
this section general enough
so that it applies to actions of linear algebraic groups defined over fields that are perfect and of finite $\ell$ cohomological dimension for some $\ell \ne char (k)$.

\subsection{Equivariant Derived Categories: Version I}
Presently we proceed to define a  model for the equivariant derived category that is valid in all characteristics, making use
of an algebraic model for the classifying space $\BG$ for a linear algebraic group as in \cite{MV99} and \cite{To99}. (The construction in
 \cite{BL94} is similar in spirit, and may have served as a motivation for the constructions in \cite{MV99} and \cite{To99}, but applies mostly to actions of topological groups on topological spaces.)

\begin{definition}
\label{geom.class}
We will 
often use $\rmEG^{\gm, {\it m}}$ to denote the $m$-th term of an admissible gadget $\{\rmU_m|m \}$: the superscript ${gm}$ stands for {\it geometric}. This is 
discussed in more detail in Definition ~\ref{defn:Adm-Gad}. (Recall this means $\rmU_{\it m}=\rmEG ^{\gm, {\it m}}$  is an open $\rmG$-stable
subvariety of a representation $\rmW_{\it m}$ of $\rmG$, so that (i) $\rmG$ acts freely on $\rmU_{\it m}$ and a geometric quotient
$\rmU_{\it m}/\rmG$ exists as a variety and
(ii) so that in the family $\{(\rmW_{\it m}, \rmU_{\it m})|{\it m} \eps {\mathbb N}\}$, the codimension of $\rmW_{\it m} - \rmU_{\it m}$ in $\rmW_{\it m}$ goes to $\infty$ as 
$m$ approaches $\infty$.)
\end{definition}
Making use of $\rmEG^{\gm, {\it m}}$ we may now
define a characteristic free algebraic model for the equivariant derived category.

\vskip .2cm
\subsubsection{\bf Convention} Henceforth, we will adopt the following conventions. 
If $\rmX$ is a variety defined over $k$, we will consider $\ell$-adic sheaves on $\rmX_{et}$, which denotes 
the \'etale site of $\rmX$. (One may also consider sheaves of ${\mathbb Z}/\ell^{\nu}$-modules, or sheaves of $R$-modules for a commutative 
Noetherian ring that is torsion with torsion prime to $char(k)$. This way, it becomes possible to
 consider $\ell$-adic sheaves where $\Q_{\ell}$ is replaced by $\bar \Q_{\ell}$, which is the algebraic closure
  of $\Q_{\ell}$: see \cite[2.2.18]{BBD82}.)
If $\rmX$ is a variety defined over the
complex numbers, we may consider sheaves of $\Q$-vector spaces on the transcendental site of $\rmX(\Cl)$ or $\ell$-adic sheaves on $\rmX_{et}$. We will denote by $\rmD(\rmX)$ ($\rmD^+(\rmX)$, $\rmD^b(\rmX)$) the unbounded derived category  
(the bounded below derived category, the bounded derived category, \res) of complexes of sheaves of $\Q$-vector spaces on $\rmX(\Cl)$ or $\ell$-adic sheaves on $\rmX_{et}$ depending on the context.
\vskip .1cm
 Observe that, if $k$ is algebraically closed, 
\[\rmH^i_{et}(EG^{\gm, {\it m}}, \Q_{\ell}) =0 \mbox{ for all } 0<i \le 2m-2 \mbox{ and } \rmH^0_{et}(EG^{\gm, {\it m}}, \Q_{\ell}) =\Q_{\ell}. \]
This follows from the fact that  $\rmEG^{\gm, {\it m}} = \rmU_m$ which is an open $\rmG$-stable subvariety of codimension at least $m>1$.
(It may be deduced from the hypothesis in the definition of the admissible gadgets in Definition ~\eqref{defn:Adm-Gad} that 
$\codim_{\rmW_m}\left(\rmW_m \setminus \rmU_m\right) =
m (\codim_{\rmW}(\rmZ))$.)  The corresponding results also hold with ${\mathbb Z}$ or $\Q$-coefficients over an algebraically closed field of characteristic $0$.
(Here we apply Lemma ~\ref{acyclic.n} with $c=m$.) Therefore, for each fixed finite interval $\rmI=[a, b]$ of integers, $a \le b$, and each integer $m \ge 0$, 
we now define
\[\rmD^{\rmI }(\rmEG^{\gm, {\it m}}{\underset {G} \times \rmX}) = \{K \eps \rmD(\rmEG^{\gm, {\it m}}{\underset {G} \times \rmX}))| {\mathcal H}^i(K) =0, i \notin \rmI \}.\]
Now we let, for each $\rmI$, with $2m-2  \ge |\rmI|=b-a$,
\be \begin{multline}
     \begin{split}
\label{eqdercat.1}
\rmD^{\rmI, gm}_{\rmG, {\it m}}(\rmX) = \mbox{ the full subcategory of } \rmD^{\rmI}(\rmEG^{\gm, {\it m}}{\underset {\rmG} \times \rmX}) \mbox { consisting of those } K\\ 
\mbox{ such that there exists 
an } L \eps \rmD(\rmX) \mbox{ so that } \pi_m^*(K) {\overset {\simeq} \leftarrow} p_{2,m}^*(L).
\end{split}
\end{multline} \ee
Here $\pi_m: \rmEG^{\gm, {\it m}} \times \rmX \ra \rmEG^{\gm, {\it m}} {\underset {G} \times} \rmX$ is the quotient map and
 $\rmp_{2,m}: \rmEG^{\gm, {\it m}} \times \rmX  \ra \rmX$ is the projection.
In case we need to clarify the choice of the geometric classifying spaces, we will denote $\rmD^{\rmI, gm}_{G, {\it m}}(\rmX)$ by
$\rmD^{\rmI}_{\rmG}(E\rmG^{\gm, {\it m}} {\underset {\rmG} \times}\rmX)$. One observes that if $\rmI \subseteq \rmJ$, then one obtains a fully faithful
imbedding $\rmD^{\rmI,gm}_{\rmG, {\it m}}(\rmX) \ra \rmD^{\rmJ, gm}_{\rmG, {\it m}}(\rmX)$, so that varying $\rmI$, one obtains a filtration of 
$\rmD^{b,gm}_{\rmG, m}(\rmX)$, which is defined as above except the vanishing of the cohomology sheaves ${\mathcal H}^i(K)$ is for
all $i $ outside of some finite interval $\rmI$ depending on $K$.
\vskip .2cm
Next let $i_{\it m}^{\rmG} =i_{\it m} {\underset {\rmG} \times} id_{\rmX}: \rmEG^{\gm, {\it m}} {\underset {\rmG} \times} \rmX \ra \rmEG^{\gm, {\it m+1}} {\underset {\rmG} \times} \rmX$
denote the map induced by the regular immersion $i_{\it m}:\rmEG^{\gm, {\it m }} \ra \rmEG^{\gm, {\it m+1}}$, which is the map
defined in Definition ~\ref{defn:Adm-Gad}(5). Pullback along this map defines the inverse system 
$\{\rmD^{\rmI, gm}_{G,m}(\rmX)|m\}$. We proceed to take the limit of this inverse system of categories: we discuss
this in a more general setting as follows.
 Given an inverse system of categories $\{\bC_i|i \eps {\mathbb N}\}$ and functors
$\F _{i-1,i} : \bC_i \ra \bC_{i-1}$ for each $ i \eps {\mathbb N}$, we define the inverse limit category $\limi \bC_i$ to be the
following category: 
\begin{enumerate}
 \item The objects are pairs $(\{C_i|i \} , \{\phi_{ i-1,i}|i \} )$ where $C_i \eps \bC_i$ for each $i \eps {\mathbb N}$ and
$\phi_{i-1,i} : \F _{i-1,i} (C_i ) {\overset {\cong} \ra }C_{ i-1}$  for any $ i \ge 2$.
\item
 A morphism $f$ between two objects $(\{C_i| i \}, \{\phi_{i-1,i }|i \})$, $(\{D_i|i  \} , \{\psi_{ i-1,i } |i\} )$
is a set of arrows $\{f_i : C_i \ra D_i|i \}$  so that
the  diagram
\vskip .2cm \indent
\xymatrix{{\F_{i-1, i}(C_i)} \ar@<1ex>[r]^{\phi_{i-1,i}} \ar@<1ex>[d]^{\bF_{i-1, i}(f_i)} & {C_{i-1}} \ar@<1ex>[d]^{f_{i-1}}\\
          {\F_{i-1, i}(D_i)} \ar@<1ex>[r]^{\psi_{i-1,i}}  & {D_{i-1}}}          
\vskip .2cm \noindent
commutes.  Composition of morphisms is component-wise.
\end{enumerate}
\vskip .2cm
Similarly if $\{ \F_{\alpha, \beta}: \bC_{\alpha} \ra \bC_{\beta}| \alpha, \beta \eps \rmJ \}$ is a sequence of functors indexed by the filtered category $\rmJ$, the category $\colimJ \bC_{\alpha}$ may be described as follows: objects have the form $(C,\alpha)$ for some $\alpha \eps \rmJ$ and $C \eps Ob(\bC_{\alpha})$. 
The set of morphisms from $(C, \alpha)$ to $(D,  \beta)$ is given by
${\underset {\gamma \ge \alpha, \beta} {\rm {colim}}} Hom_{\bC_{\gamma}}(\F_{\alpha, \gamma}(C), \F_{\beta, \gamma}(D))$.
\vskip .2cm

\vskip .2cm
Let $\rmI =[a, b]$, $a <b$ denote a finite interval of integers.  Now one defines
\be \begin{equation}
     \label{der.cat.1}
 \rmD^{\rmI, gm}_{\rmG}(X) = \limm \{\rmD^{\rmI, gm}_{G, {\it m}}(\rmX)|m\} \mbox{ and } \rmD^{b, gm}_{\rmG}(X)  = \colimI  \rmD^{\rmI, gm}_{\rmG}(X) 
   \end{equation} \ee
where the last colimit is over all finite intervals $\rmI=[a, b]$. Making use of the above discussion, we define for each prime $\ell$ different from the 
residue characteristics, and each 
finite interval $\rmI=[a, b]$ of integers
\be \begin{equation}
     \label{eqdercat.2}
\rmD^{\rmI, gm}_{\rmG}(\rmX, {\mathbb Z}_{\ell }) = \limnu \{\rmD^{\rmI,gm}_{\rmG}(\rmX, {\mathbb Z}/\ell^{\nu})|\nu\} \mbox{ and } \rmD^{\rmI, gm}_{\rmG}(\rmX, \Q_{\ell}) = \rmD^{\rmI, gm}_{\rmG}(\rmX, {\mathbb Z}_{\ell}) \otimes \Q_{\ell}.
\end{equation} \ee
One also defines the $\ell$-adic derived categories
\be \begin{equation}
\label{eqdercat.2}
\rmD^{b, gm}_{\rmG, {\it m}}(\rmX, \Q_{\ell}) = \colimI \{\rmD^{\rmI,gm}_{\rmG, {\it m} }(\rmX, \Q_{\ell})|\rmI\}. 
\end{equation} \ee
\vskip .2cm \noindent
When everything is defined over the complex numbers, one may define $\rmD^{\rmI, gm}_{\rmG}(\rmX, \Q)$ and $\rmD^{b, gm}_{\rmG, m}(\rmX, \Q)$ similarly.

\begin{proposition} 
\label{stability.1}
Let $R$ denote a commutative Noetherian ring that is torsion, with torsion prime to $char (k)$.
\vskip .2cm
(i) Let $\{{\widetilde {\rmE \rmG }}^{\gm, {\it m}}|m\}$ denote the geometric classifying space defined with respect to the choice of another admissible 
gadget. Then, for each fixed $\rmI=[a, b]$ with $2m-2 \ge |\rmI | =b-a$, the projections 
\[pr_1:(\rmE\rmG ^{\gm, {\it m}} \times {\widetilde {\rmE\rmG }}^{\gm, {\it m}}){\underset {\rmG} \times}\rmX \ra \rmE\rmG^{gm,m} {\underset {\rmG} \times}\rmX \mbox{ and } {\it pr}_2:(\rmE\rmG ^{\gm, {\it m}} \times {\widetilde {\rmE\rmG }}^{\gm, {\it m}}){\underset {\rmG} \times}\rmX \ra {\widetilde {\rmE\rmG }}^{\gm, {\it m}} {\underset {\rmG} \times}\rmX \] 
induce equivalences 
\[pr_1^*:\rmD^{\rmI}_{\rmG}(\rmE\rmG^{\gm, {\it m}} {\underset {\rmG} \times} \rmX, R) \ra \rmD^{\rmI}_{\rmG}((\rmE\rmG ^{gm,m} \times {\widetilde {\rmE\rmG }}^{\gm, {\it m}}){\underset {\rmG} \times}\rmX, R) \mbox{ and }\]
\[pr_2^*:\rmD^{\rmI}_{\rmG}({\widetilde {\rmE\rmG}}^{\gm, {\it m}} {\underset {\rmG} \times} \rmX, R) \ra \rmD^{\rmI}_{\rmG}((\rmE\rmG ^{\gm, {\it m}} \times {\widetilde {\rmE\rmG }}^{\gm, {\it m}}){\underset {\rmG} \times}\rmX, R).\]
\vskip .2cm
(ii)  For a fixed $\rmI =[a,b]$, with $2m-2 \ge |\rmI| =b-a$, the induced map $i_m^{G,*}: \rmD^{\rmI,gm}_{G,{\it m+1}}(\rmX, R) \ra 
\rmD^{\rmI, gm}_{G, {\it m}}(\rmX, R)$ is an equivalence of categories.
\end{proposition}
\begin{proof}
 (i) follows readily since both $\rmE G^{\gm, {\it m}}$ and ${\widetilde {\rmE\rmG}}^{\gm, {\it m}}$ are $2m-2$-connected. In view of this,
the functor $pr_i^*$ (for $i=1,2$) has an inverse given by $\tau_{\le b}Rpr_{i*}$, where $\tau_{\le b}$ denotes the cohomology truncation to degrees
$\le b$. Therefore, the functors $pr_i^*$ are fully-faithful. Now to see these functors are   equivalences of derived categories, it suffices to observe that all
complexes in the above equivariant derived categories are quasi-isomorphic to complexes obtained by pull-back from $\rmX$.
\vskip .2cm
Next we consider (ii). One  may take
${\widetilde {{\rmE\rmG}}}^{\gm, {\it m}} = \rmE\rmG^{\gm, {\it m+1}}$ which show the projections
$pr_1:(\rmE\rmG^{\gm, {\it m}} \times ({ {\rmE\rmG}})^{\gm, {\it m+1}}){\underset {\rmG} \times} \rmX \ra \rmE\rmG^{\gm, {\it m}}{\underset {\rmG} \times} \rmX$ and 
${pr}_2: (\rmE\rmG^{\gm, {\it m}} \times ({ {\rmE\rmG}})^{\gm, {\it m+1}}){\underset {\rmG} \times} \rmX \ra \rmE\rmG^{\gm, {\it m+1}}{\underset {\rmG} \times} \rmX$
both induce equivalences:
 \[pr_1^*:\rmD^{\rmI}_{\rmG}(\rmE\rmG^{\gm, {\it m}} {\underset {\rmG} \times} \rmX, R) \ra \rmD^{\rmI}_{\rmG}((\rmE\rmG ^{\gm, {\it m}} \times { {\rmE\rmG }}^{\gm, {\it m+1}}){\underset {\rmG} \times}\rmX, R ) \mbox{ and}  \]
\[pr_2^*:\rmD^{\rmI}_{\rmG}({ {\rmE\rmG}^{\gm, {\it m+1}}} {\underset {\rmG} \times} \rmX, R) \ra \rmD^{\rmI}_{\rmG}((\rmE\rmG ^{\gm, {\it m}} \times { {\rmE\rmG }}^{\gm, {\it m+1}}){\underset {\rmG} \times}\rmX, R).\]
If $\Delta: \rmE\rmG^{\gm, {\it m}}{\underset {\rmG} \times} \rmX \ra (\rmE\rmG^{\gm, {\it m}} \times ({ {\rmE\rmG}})^{\gm, {\it m+1}}){\underset {\rmG} \times} \rmX $ is the diagonal
imbedding, then $\Delta$ is a section to $pr_1$ and $ i_m^{\rmG} = pr_2 \circ \Delta $. Since $pr_1^*$ is an equivalence, so is $\Delta^*$; since
$\Delta^*$ and $pr_2^*$ are equivalences, it follows so is $i_m^{G*}$. 
(Here, it may be worthwhile to observe that for every complex $K \eps \rmD^{\rmI}_{\rmG}(\rmE\rmG^{\gm, {\it m+1}}{\underset {\rmG} \times}\rmX, R)$ 
 there exists a complex $L \eps
\rmD^{\rmI}(\rmX, R)$ so that $\pi_{m+1}^*(K)  {\overset {\simeq} \leftarrow} p_{2, m+1}^*(L)$ following the notations as in ~\eqref{eqdercat.1}.)  
\end{proof}
\begin{remark} Observe that, for each fixed $m$, and each fixed finite interval $\rmI$, one obtains the equivalences:
\[\rmD^{\rmI,gm}_{\rmG}(\rmX, \Q_{\ell}) = \limm \rmD^{\rmI, gm}_{G,{\it m} }(\rmX, \Q_{\ell}) = \limm (\limnu \rmD^{\rmI,gm}_{G, {\it m}}(\rmX, {\mathbb Z}/\ell^{\nu}))\otimes \Q_{\ell})\]
\[ \cong (\limm  \limnu \rmD^{\rmI, gm}_{G, {\it m}}(\rmX, {\mathbb Z}/\ell^{\nu})) \otimes \Q_{\ell} \cong (  \limnu \limm \rmD^{\rmI, gm}_{G, {\it m}}(\rmX, {\mathbb Z}/\ell^{\nu})) \otimes \Q_{\ell}.\]
All  but the next-to-last equivalence are clear from the definition, and this equivalence follows from 
Proposition ~\ref{stability.1}(ii).
\end{remark}

\subsection{The role of stratifications}
\label{strats}
Next we will consider the role of stratifications. A $\rmG$-stratification of the variety $\rmX$ is a decomposition of $\rmX$ into finitely many disjoint 
locally closed smooth equi-dimensional and $\rmG$-stable subvarieties called {\it strata}. Let the stratification  be denoted $\S=\{S_{\alpha}| \alpha \}$. Since the Borel construction is functorial, 
and $\rmE\rmG^{\gm, {\it m}}/\rmG$ is smooth, such a stratification of $\rmX$ defines a stratification $\{\rmE\rmG^{\gm, {\it m}}{\underset {\rmG} \times}S_{\alpha}| \alpha \}$ of the Borel construction
$\rmE\rmG^{\gm, {\it m}} {\underset {\rmG} \times}\rmX$ for each $m$. This stratification of $\rmE\rmG^{\gm, {\it m}}{\underset {\rmG} \times}\rmX$ will be
denoted $\S_{\rmG,m}$. These are evidently compatible as $m$ varies. 
\vskip .2cm
Given  an interval $\rmI$ of the integers, we let $\rmD^{\rmI,gm}_{\rmG, {\it m}}(\rmX, \S_{\rmG, m})$ denote the full subcategory of 
\newline \noindent
$\rmD^{\rmI, gm}(E\rmG^{\gm, {\it m}}{\underset {\rmG} \times}\rmX, \S_{\rmG, {\it m}})$ consisting
 of complexes $K$ whose cohomology sheaves are $\rmG$-equivariant. 
\newline \noindent
$\rmD^{b, gm}_{\rmG, {\it m}}(\rmX, \S_{\rmG, {\it m}})$ will denote the full
subcategory of complexes $K$ that belong to $\rmD^{\rmI, gm}(E\rmG^{\gm, {\it m}}{\underset {\rmG} \times}\rmX, \S_{\rmG, {\it m}})$ for some $I$. 
One  takes the $2$-limit of the  categories $\{\rmD^{\rmI,gm}_{\rmG, {\it m}}(E\rmG^{\gm, {\it m}}{\underset {\rmG} \times}\rmX, \S_{\rmG, {\it m}})|m\}$ as $m \ra \infty$ to define $\rmD^{\rmI, gm}_{\rmG}(\rmX, \S)$. 
\vskip .2cm
Next we discuss how the $t$-structures on  the equivariant derived category $\rm \rmD_{\rmG, m}^{b,gm}(\rmX, \S_{\rmG, m})$ behave as one varies $m$ and also  the stratifications.
\vskip .2cm
First observe that the map $i_{\it m}:\rmEG^{\gm, {\it m}} \ra \rmEG^{\gm, {\it m+1}}$ is a regular closed immersion, for each $m$. Therefore so is the induced map
$i_m^{\rmG}: \rmEG^{\gm, {\it m}}{\underset G \times} \rmX \ra \rmEG^{\gm, {\it m+1}}{\underset G \times} \rmX $. Given a $\rmG$-stratification $\S =\{S_{\alpha}|\alpha \}$ of $\rmX$, 
let $i^{(m)}_{S_{\alpha}}= id {\underset {\rmG} \times} i: {{\rmE{\rm G}^{\gm, {\it m}}{\underset {\rm G} \times}}} S_{\alpha} \ra {{\rmE{\rm G}^{\gm, {\it m}}{\underset {\rm G} \times}}} \rmX$ denote the induced closed immersion. 
\vskip .2cm
A perversity function $p$ defined on a stratified variety $\rmY$ will be defined as a non-decreasing function on codimension of the strata, so 
that the value on the open stratum will be $0$. We will view $p$ as defined on the strata themselves. 
(For the most part, we will only consider the middle perversity, which is defined by $m(\rmS ) =$ the codimension of $\rmS$ in $\rmY$.)  
Recall that the standard $t$-structure on a derived category of complexes with bounded cohomology is one whose heart consists of complexes that
have non-trivial cohomology only in degree $0$.
Then one  may start with standard $t$-structures defined on each of the strata $\rmS$ shifted by the perversity $p(\rmS)$ and obtain a non-standard $t$-structure
on the bounded derived category, $\rmD^b( \rmY)$ by {\it gluing} as in \cite[Definition 2.1.2]{BBD82}.
\begin{proposition}
\label{key.props.eq.der.cats}
%\begin{enumerate}
 Let $\S_{\rmG}$ denote a fixed $\rmG$-stable stratification of the $\rmG$-variety $\rmX$ and let $\S_{\rmG, m}$ denote the
induced stratification on $\rmEG^{\gm, {\it m}}{\underset G \times}\rmX$ for the action of $\rmG$ on $\rmX$.
\item [(i)] Let the derived categories $\rmD^{b,gm}_{G,{\it m+1}}(\rmX, \S_{\rmG, m+1})$ and  $ \rmD^{b, gm}_{G, {\it m}}(\rmX, \S_{\rmG, m})$ be provided with 
the $t$-structures obtained by gluing with respect to a fixed perversity function and with respect to
a $\rmG$-stable stratification of $\rmX$. Then the functor $i_m^{G,*}$ preserves the above $t$-structures.
\item[(ii)] Let $\sT_{\rmG}= \{\sT_{\beta}|\beta\}$ denote a $\rmG$-stable stratification that is a refinement of the stratification 
$\S_{\rmG}$. Then
any complex in $\rmD^{b, gm}_{G, {\it m}}(\rmX)$ whose cohomology sheaves are local systems on each stratum of $\S_{\rmG, m}$ clearly belongs $\rmD^{b,gm}(\rmEG^{\gm, {\it m}}{\underset G \times}\rmX, \sT_{\rmG,m})$. This induces the inclusion functor
\[\rmD^{b,gm}_{G, {\it m} }(\rmX, \S_{\rmG, {\it m}}) \ra \rmD^{b, gm}_{G, {\it m} }(\rmX, \sT_{\rmG, {\it m}}).\]
This functor preserves the $t$-structures on either side obtained by gluing.
%\end{enumerate}
\end{proposition}
\begin{proof} 
 We will first consider (i). Let $i$ denote generically the closed immersions ${{\rmE{\rm G}^{\gm, {\it m}}{\underset {\rm G} \times}}} \rmS \ra \EGm+1 \rmS$ and   ${{\rmE{\rm G}^{\gm, {\it m}}{\underset {\rm G} \times}}} \rmX \ra \EGm+1 \rmX$.
Recall the $t$-structure obtained by gluing on $ \rmD^{b, gm}_{G, {\it m}}(\rmX)$ is such that
\be \begin{align}
     {\rmD^{b, gm}_{G, {\it m}}(\rmX, \S_{\rmG, {\it m}})} ^{\le 0} &=\{K \eps \rmD^{b, gm}_{G, {\it m}}(\rmX, \S_{\rmG, {\it m}})| {\mathcal H}^i(i^{(m)}_S)^*(K)) =0, i > p(S)\} \mbox{ and } \\
     {\rmD^{b, gm}_{G, {\it m}}(\rmX, \S_{\rmG, {\it m}})} ^{\ge 0} &=\{K \eps \rmD^{b, gm}_{G, {\it m}}(\rmX, \S_{\rmG, {\it m}})| {\mathcal H}^i(Ri^{(m)}_S)^!(K)) =0, i < p(S)\}
    \end{align} \ee
Here $i^{(m)}_S: \rmE\rmG^{\gm, {\it m}}{\underset {\rmG} \times} \rmS \ra \rmE\rmG^{\gm, {\it m}}{\underset {\rmG} \times} \rmX$ is the closed immersion corresponding to
a stratum $\rmS$. 
Since $i^{(m)*}_{\rmS} i^* = i^* i^{(m+1)*}_{\rmS}$, it follows readily that $i^*$ sends  $\rmD^{I, gm}_{G, {\it m+1}}(\rmX, \S_{\rmG, {\it m}}) ^{\le 0}$ to $ \rmD^{I, gm}_{G,{\it m}}(\rmX, \S_{\rmG, {\it m}}) ^{\le 0}$.
To see that $i^*$ sends $\rmD^{I, gm}_{G, {\it m+1}}(\rmX, \S_{\rmG, m}) ^{\ge 0}$ to $ \rmD^{I, gm}_{G, {\it m}}(\rmX, \S_{\rmG, {\it m}}) ^{\ge 0}$, one needs to observe that $i^{(m)}_S= id_{{\rmE\rmG}^{\gm, {\it m}}} \times_{\rm G} i_{\rmS}$,
$i^{(m+1)}_S= id_{{E\rmG}^{gm, {\it m+1}}} \times_{\rm G} i_S$ and that $i= i' \times _{\rmG} id$, where $i':\rmE\rmG^{\gm, {\it m}} \ra \rmE\rmG^{\gm, {\it m+1}}$ is the closed
immersion. One  also needs the fact that for every complex $K \eps \rmD^{\rmI, gm}_{G, {\it m}}(\rmX, \S_{\rmG, {\it m}})$ 
 there exists a 
complex $L \eps \rmD^{\rmI}(\rmX)$ so that $p_2^*(L) \simeq \pi_m^*(K)$. 
\vskip .2cm
(ii) follows from \cite[Proposition 2.1.14]{BBD82}.
\end{proof}
\begin{proposition}
 Assume in addition to the above hypotheses that $\rmG$ acts with finitely many orbits on the variety $\rmX$. Let $\S$ denote the
stratification of $\rmX$ by the $\rmG$-orbits. Then $\rm \rmD_{\rmG, {\it m}}^{b, gm}(\rmX, \S_{\rmG, {\it m}}) = \rmD^{b, gm}_{\rmG, {\it m}}(\rmX)$, for every $m$.
\end{proposition}
\begin{proof} This is clear since any $\rmG$  equivariant sheaf is a local system  on each orbit.
\end{proof}

\vskip .2cm
\subsection{\bf Equivariant Derived Categories: Version II} This is based on the simplicial model of the Borel construction and
we discussed this already in section ~\ref{simpl.der.cat}. Therefore,  we proceed to discuss the proof of Theorem ~\ref{mainthm.5}, next. 
\vskip .2cm \noindent
{\bf Proof of Theorem ~\ref{mainthm.5}.}
\label{equiv.dercat.3}
 (i) The fact that $p_1^*$  is fully faithful follows from the observation that the geometric fibers of the map $p_1$ are
$\rmEG^{\gm, {\it m}}$ (base extended to the algebraic closure of the base field) which is acyclic in degrees $1$ through $2m-2$. (See Lemma ~\ref{acyclic.n}.) Its inverse is the functor $\tau_{\le b}Rp_{1*}$, where
$\tau_{\le b}$ is the functor killing cohomology above degree $b$. Now the fact that $p_1^*$
is an equivalence follows from observing that it induces an  equivalence  on the hearts of the corresponding derived categories provided with the
usual $t$-structures, where the heart consists of complexes with trivial cohomology in all degrees except $0$. Both $p_1$ and $p_2$ are simplicial maps which are smooth in each degree: therefore, both $p_1^*$ and $p_2^*$ send
complexes that are mixed and pure to complexes that are mixed and pure.
\vskip .2cm
The geometric fibers of the map $p_2$ identify with the simplicial variety $\rmEG$ (base extended to the algebraic closure of the base field) which is acyclic. The equivalence in 
~\eqref{eq.cats.1} and 
\cite[Theorem 4.2]{Jo02}, show readily that the functor $p_2^*$ is fully-faithful. That it is also an 
equivalence now follows by observing
that $p_2^*$ induces an equivalence on the hearts of the corresponding derived categories provided with the
usual $t$-structures, where the heart consists of complexes with trivial cohomology in all degrees 
except $0$. These observations complete the proof of the statements in (i). 
\vskip .2cm
(ii) We make use of the construction of the geometric classifying spaces through the admissible gadgets discussed in ~\eqref{adm.gadget.1}, but first over
 the algebraic closure of the base field. 
In the  terminology there, if we are working over the complex numbers, then with the
complex topology, we now observe that  ${\rm U}_{m+1}$ is in fact the {\it join} of ${\rm U}_m$ with ${\rm U}$, where the join ${\rm U}_m* {\rm U}$ denotes the
homotopy pushout 
\be \begin{equation}
\label{join.1}
     \xymatrix{{{\rm U}_m \times {\rm U}} \ar@<1ex>[r] \ar@<1ex>[d] & {{\rm U}} \ar@<1ex>[d]\\
                {{\rm U}_m} \ar@<1ex>[r] & {{\rm U}_m*{\rm U}}}
    \end{equation} \ee
\vskip .2cm \noindent
In the \'etale framework, the corresponding statement holds when ${\rm U}_m$ (${\rm U}$) is replaced by the completed \'etale topological types, with the
completion at a prime $\ell$ different from the characteristic. Thus, in case the base field is the complex numbers, ${\rm U}_m$ is the iterated join 
of ${\rm U}$
$m$-times and in positive characteristics, $({\rm U}_m)_{et}\compl_{\ell} \simeq ({\rm U}_{et} \compl_{\ell}) * \cdots * ({\rm U}_{et} \compl_{\ell})$ (the $m$-fold join).
Therefore $E\rmG^{\gm, {\it m}}$ ($E\rmG^{\gm, {\it m}}_{et} \compl_{\ell}$) is highly  connected in the first case 
(the second case respectively) in the sense that the homotopy groups are all trivial through a sufficiently
 high degree. 
(A point to observe here is that the join of any two connected simplicial sets is simply connected. Completion preserves a simplicial set being 
connected. Therefore, since the join is already simply connected, one may check if the join $({\rm U}_m)_{et}\compl_{\ell} * ({\rm U}_{et})\compl_{\ell}$ is $\ell$-complete 
 on homology with $Z/l^{\nu}$-coefficients. Therefore, the join of two 
$\ell$-complete simplicial sets  is $\ell$-complete.)
\vskip .2cm
Next observe that the map $p_1$ is a map of simplicial varieties, which in degree $n$ is given by the map
\[p_{1,n}: \rmG^{n-1}\times (\rmEG^{gm.m} \times \rmX) \ra \rmG^{n-1} \times \rmX.\]
Therefore, taking the completed \'etale topological types provides an inverse system of maps of simplicial sets for each fixed $n$:
\be \begin{equation}
\label{fib.n}
p_{1,n,et}: {(\rmG)_{et}\compl_{\ell}}^{n-1} \times (\rmEG^{\gm, {\it m}}_{et}\compl_{\ell}) \times \rmX_{et}\compl_{\ell} \ra {(\rmG)_{et}\compl_{\ell}}^{n-1} \times \rmX_{et}\compl_{\ell}
\end{equation}\ee
This is simply the projection to all factors except $\rmEG^{\gm, {\it m}}_{et}\compl_{\ell}$, so that the fibers identify with 
$\rmEG^{\gm, {\it m}}_{et}\compl_{\ell}$.
Varying  $n$, this defines an inverse system of maps of bisimplicial sets 
\be \begin{equation}
 \label{fibrations}
\{\rmE_{\bullet, \bullet}^{\alpha} \ra \rmB_{\bullet, \bullet}^{\alpha}|\alpha\}
\end{equation} \ee
\vskip .2cm \noindent
with the second index corresponding to the $n$ in ~\eqref{fib.n}. 
Since all the varieties are connected, it follows that for each fixed $n$ and $\alpha$, the simplicial set  $\rmB_{\bullet, n}^{\alpha}$
is connected. Therefore, one may invoke \cite[Lemma 5.2]{Wa78} to conclude that on taking the diagonals, we obtain 
an inverse system of fibrations $\{\Delta (\rmE_{\bullet, \bullet})^{\alpha} \ra \Delta (\rmB_{\bullet, \bullet})^{\alpha}|\alpha\}$. The fibers
are clearly $\rmEG^{\gm, {\it m}}_{et}\compl_{\ell}$, which we have observed above are highly connected if $m$ is sufficiently large. Therefore
it follows that the map $p_1$ induces an isomorphism on the fundamental groups completed at $\ell$, when everything is defined over the
algebraic closure of $k$.
\vskip .2cm
Observe that the map $\rmEG^{\gm, {\it m}} \times \rmX \ra \rmEG^{\gm, {\it m}}{\underset {G} \times}\rmX$ is a locally trivial principal fibration in the \'etale
topology. Therefore, a similar analysis as for the case of  the map $p_1$, will prove that the map $p_2$ will induce an isomorphism
on the fundamental groups completed at $\ell$, again when everything is defined over the algebraic closure of $k$.
\vskip .2cm
Next we consider the general case, where the base field $k$ is no longer required to be algebraically closed. 
In this case we will let $\bar k$ denote the algebraic closure of $k$.
For this discussion we will denote all objects defined over the base field $k$ with a subscript $o$, while objects over $\bar k$ will be denoted
as before. It suffices
to show that local systems on $\rmE {\rmG}_o{\underset {\rmG_o} \times}\rmX_o$ and $\rmEG_o^{\gm, {\it m}} {\underset {\rmG_o} \times}\rmX_o$ correspond $1-1$ if $m>>0$.
Observe a local system on either of these defines, by pull-back a local system on $\rmEG_o {\underset {\rmG_o} \times}(\rmEG_o^{\gm, {\it m}} \times \rmX_o)$
and by further pull-back a local system on $\EGx (\rmEG^{\gm, {\it m}} \times \rmX)$. Observe that the group of automorphisms of the latter over
$\rmEG_o {\underset {G_o} \times}(\rmEG_o^{\gm, {\it m}} \times \rmX_o)$ is the Galois group $Gal_k(\bar k)$. By what we just proved in the paragraphs
above,  local systems on $\EGx (\rmEG^{\gm, {\it m}} \times \rmX)$ correspond $1-1$ with local systems on $\EGx \rmX$ as well as local systems 
on $\rmEG^{\gm, {\it m}} {\underset {G} \times}\rmX$ which are also
equivariant for the action of $Gal_{k}(\bar k)$. Therefore, by $Gal_k(\bar k)$-equivariance, these local systems descend to local systems
on $\rmE{\rmG}_o{\underset {\rmG_o} \times}\rmX_o$ and $\rmEG_o^{\gm, {\it m}} {\underset {\rmG_o} \times}\rmX_o$.
 This completes the proof of the second statement and hence of the theorem.
 \vskip .2cm
 \subsection{\bf Derived functors for maps of simplicial varieties}
 \label{simpl.der.functors}
 \vskip .2cm
 It is shown in \cite{Jo02} and also (\cite{Jo93}), 
 that associated to any given Grothendieck topology on varieties, there is another Grothendieck topology
that one can put on any simplicial variety $\rmX_{\bullet}$. As we show below, this topology plays a key role in
being able to define derived functors of the direct image functor for maps between simplicial varieties, and therefore often comes in handy.
 Therefore, we proceed to summarize the main features of this construction. Let $\Top$ denote
a Grothendieck topology defined on algebraic varieties over the given field $k$. Let $\rmX_{\bullet}$ denote a simplicial variety over $k$. 
Then we define the objects of the topology, $\STop( \rmX_{\bullet})$, to consist of all maps 
$u_{\bullet}:\rmU_{\bullet} \ra \rmX_{\bullet}$ of simplicial varieties so that 
 each $u_n: \rmU_n \ra \rmX_n$ belongs to the topology 
$\Top(\rmX_n)$. Morphisms between two such objects will be defined to be
commutative triangles. One defines
a family of maps $\{{v_{\bullet}}_{\alpha}:{\rmV_{\bullet}}_{\alpha} \ra \rmU_{\bullet}|\alpha\}$ to be a covering if each 
$\{v_{n_{\alpha}}:\rmV_{n_\alpha} \ra \rmU_n| \alpha \}$ is a {\it covering} in $\Top(\rmX_n)$. 
We will often call this {\it the simplicial topology  associated to the
given topology $\Top$}. It is shown in \cite[section 1]{Jo02} that the topology $\STop$ and the associated topos of sheaves on it 
inherits all the good properties from the given topology $\Top$: for example, the category underlying $\STop(\rmX_{\bullet})$ is closed
under finite inverse limits and  there
are enough points on the site $\STop(\rmX_{\bullet})$ if  each $\Top(\rmX_n)$ has the corresponding property.
\vskip .2cm
Observe that for each fixed integer $n \ge 0$, there is a map of sites $\eta_n:\Top(\rmX_n) \ra \STop(\rmX_{\bullet})$ defined by sending the
(simplicial) object $\rmU_{\bullet}$ in $\STop(\rmX_{\bullet})$ to $\rmU_n$. 
Let $\rmR$ denote a commutative Noetherian ring and let $\rmC^+(\Sh(\Top(\rmX_{\bullet}), \rmR))$ ($\rmC^+(\Sh(\STop(\rmX_{\bullet}), \rmR))$) denote the category of all
complexes of sheaves of $\rmR$-modules on $\Top(\rmX_{\bullet})$ ($\STop(\rmX_{\bullet})$, \res),  that are bounded below.
Then one defines a functor $\eta: \rmC^+(\Sh(\Top(\rmX_{\bullet}), \rmR)) \ra \rmC^+(\Sh(\STop(\rmX_{\bullet}), \rmR))$ by sending
a complex of sheaves $\{K_n|n\}$ on $\Top(\rmX_{\bullet})$ to the total complex of the double complex $\{\eta_{n*}(K_n)|n\}$. One verifies
readily the following properties of this functor (see \cite[sections 1 and 3]{Jo02}):
\begin{enumerate}[\rm(i)]
\label{sheaves.simp.3}
 \item If $f: \rmX_{\bullet} \ra \rmY_{\bullet}$ is a map of simplicial varieties, then $\eta(f_*(K)) = {}_sf_*(\eta (K))$, $K \eps \rmC^+(\Sh(\STop(\rmX_{\bullet}), \rmR))$
and $\eta(f^*(L)) = {}_sf^*(\eta (L))$, $L \eps \rmC^+(\Sh(\STop(\rmY_{\bullet}), \rmR))$.  Here ${}_sf_*$ (${}_sf^*$) denotes the
push-forward functor (the pull-back functor) associated to $f$ on the simplicial topology.
Both of these statements may be checked at
the stalks and may be deduced from the basic properties of stalks of sheaves computed on the simplicial topology as in \cite[section 3]{Jo02}.
\item If $j: \rmU_{\bullet} \ra \rmX_{\bullet}$ is an open immersion in each degree, then $\eta(j_!(K)) = {}_sj_!(\eta(K))$, \\$K \eps
\rmC^+(\Sh(\STop(\rmU_{\bullet}), \rmR))$ and ${}_sj_{!}$ is the corresponding functor on the simplicial topology.
\item If $f: \rmX_{\bullet} \ra \rmY_{\bullet}$ is a map of simplicial varieties and $K \eps \rmC^+(\Sh(\STop(\rmX_{\bullet}), \rmR))$
so that $(K_n, f_n)$ is cohomologically proper, that is, satisfies the conclusions of the proper base-change theorem, then the cohomology of
the stalks of $Rf_*(K)$ computed on the simplicial topology $\STop(\rmY_{\bullet})$ identify with the cohomology of the
fibers of the simplicial map $f$ with respect to $K$. (See \cite[(4.2) Theorem]{Jo02}.)
\end{enumerate}
\vskip .2cm
 Next let $\rmG$ denote a linear algebraic group acting on a variety $\rmX$. Then it is shown in \cite[(3.8.2) Definition, (3.10) Corollary]{Jo02}
that one may define the notion of sheaves with descent (or cartesian sheaves)  on the site $\STop(\EGx \rmX)$ and that if 
$ \rmD_{\rmG}(\STop(\EGx \rmX))$ denotes the full subcategory of
$\rmD(\STop(\EGx \rmX))$ consisting of complexes whose cohomology sheaves have descent, then the functor $\eta$ induces a  {\it fully-faithful}
functor (where the superscript $+$ denotes
 the bounded below derived categories):
\be \begin{equation}
     \label{eq.cats.1}
\rmD_{\rmG}^+(\rmX) \ra \rmD_{\rmG}^+(\STop(\EGx \rmX)).
    \end{equation} \ee
\vskip .2cm
\begin{remark}
 \label{simpl.eq.der.cat}
Therefore, if $X$ is a variety for which the derived category $\rmD_{\rmG}^+(\rmX)$ is  generated by a collection of
complexes $\{K_{\alpha}|\alpha\}$, by taking the full subcategory of $\rmD_{\rmG}^+(\STop(\EGx \rmX))$ generated by the images
of these complexes, one obtains an equivalence of categories. Therefore, we will henceforth denote this full subcategory by $\rmD_{\rmG}^+(\STop(\EGx \rmX))$
which enables us to work with the simplicial topology throughout.
\end{remark}
\vskip .3cm

\vskip .3cm
\section{\bf Appendix:  Background material}
Next we consider the following background material needed for the definition of the geometric classifying spaces
in section 2.
We start with the following lemma.
\begin{lemma}
 \label{acyclic.n}
 Let $\rmV$ denote a representation of the linear algebraic group $\rmG$, all defined over a perfect field $k$ of finite $\ell$-cohomological dimension for some
prime $\ell \ne char(k)$.
Let ${\rm U} \subseteq \rmV$ denote an open 
$\rmG$-stable subvariety so that the complement $\rmV-{\rm U}$ has codimension $c>1$ in $\rmV$. 
\vskip .1cm 
(i) Then denoting by $\bar k$ the algebraic closure of $k$, 
\[ \rmH^n_{et}({\rm U}{\underset {Spec \, k} \times} Spec \,\bar k, {\mathbb Z}/\ell^{\nu}) =0 \mbox{ for all } 0<n <2c-1 \mbox{ and } \rmH^0_{et}({\rm U}{\underset {Spec \, k} \times} Spec \,\bar k, {\mathbb Z}/\ell^{\nu}) = {\mathbb Z}/\ell^{\nu}.\]
\vskip .1cm
(ii) For any variety $\rmX$, 
\[R^nf_*({\mathbb Z}/\ell^{\nu}) =0 \mbox{ for all } 0< n <2c-1,\mbox{ and } R^0f_*({\mathbb Z}/\ell^{\nu}) ={\mathbb Z}/\ell^{\nu}\]
 where
$f: {\rm U}\times \rmX \ra \rmX$ denotes the  projection.
\vskip .1cm
(iii) In case the field $k= \Cl$, the corresponding results also hold for ${\mathbb Z}$ and $\Q$ in the place of ${\mathbb Z}/\ell^{\nu}$.
\end{lemma}
\begin{proof} (i) It suffices to consider the case  $k$ is algebraically closed. Then (i) follows from the long-exact sequence
 \[ \cdots \ra \rmH^n_{et, \rmV-{\rm U}}(\rmV, {\mathbb Z}/\ell^{\nu}) \ra \rmH^n_{et}(\rmV, {\mathbb Z}/\ell^{\nu}) \ra \rmH^n_{et}({\rm U}, {\mathbb Z}/\ell^{\nu}) \ra \rmH^{n+1}_{et, \rmV-{\rm U}}(\rmV, {\mathbb Z}/\ell^{\nu}) \cdots \ra \]
and the fact that $\rmH^i_{et, \rmV-{\rm U}}(\rmV, {\mathbb Z}/\ell^{\nu}) =0$ for all $i <2c$ while $\rmH^i_{et}(\rmV, {\mathbb Z}/\ell^{\nu})=0$ for all $i>0$,
$\rmH^0_{et}(\rmV, {\mathbb Z}/\ell^{\nu})={\mathbb Z}/\ell^{\nu}$.  
 These complete the
proof of (i).
\vskip .2cm
The assertion $\rmH^i_{et, \rmV-{\rm U}}(\rmV, {\mathbb Z}/\ell^{\nu}) =0$ for all $i <2c$ is a 
cohomological semi-purity statement. We provide a short proof of this statement due to the lack of an adequate reference. Since the base field is assumed to be perfect, one may find an  open subvariety $\rmV_0$ of $\rmV$ so that
$\rmY_0 = (\rmV-{\rm U})\cap \rmV_0$ is smooth and nonempty. Now one has a long-exact sequence in cohomology:
\[\cdots \ra \rmH^i_{et, Y_1}(\rmV, {\mathbb Z}/\ell^{\nu}) \ra \rmH^i_{et, Y}(\rmV, {\mathbb Z}/\ell^{\nu}) \ra \rmH^i_{et, Y_0}(\rmV_0,{\mathbb Z}/\ell^{\nu}) \ra \rmH^{i+1}_{et, Y_1}(\rmV, {\mathbb Z}/\ell^{\nu}) \ra \cdots\]
where $\rmY= \rmV-{\rm U}$, $\rmY_1 = Y - Y_0$. We may also assume without loss of generality that $\rmY$ is  irreducible. 
By cohomological purity, $\rmH^i_{et, Y_0}(\rmV_0, {\mathbb Z}/\ell^{\nu}) =0$ for all $ i <2c$ 
(in fact, for all $i \ne 2c$). Since $\rmY_1$ is of dimension strictly less than the dimension of
$\rmY$, an ascending induction on the dimension of $\rmY$ enables one to assume $\rmH^i_{et, Y_1}(\rmV, {\mathbb Z}/\ell^{\nu}) =0$ for all $i <2\codim _{Y_1}(\rmV)$. (One may start the induction when
$\dim (Y)=0$, since in that case $\rmY$ is smooth.) Since
$\codim _Y(\rmV) < \codim _{Y_1}(\rmV)$, the long exact sequence above now proves $ \rmH^i_{et, Y}(\rmV, {\mathbb Z}/\ell^{\nu}) =0$ for all $i < 2\codim _Y(\rmV)$.
\vskip .2cm
(ii) follows readily from (i). We skip the proof of (iii) which follows
 along the same lines as the proofs of (i) and (ii).
\end{proof}
Since different choices are possible for such geometric classifying spaces, we proceed to consider this in the
more general framework of {\it admissible gadgets} as defined in \cite[section 4.2]{MV99}. The following definition is a variation
of the above definition in \cite{MV99}.
\subsection{Admissible gadgets associated to a given $\rmG$-variety}
\label{subsubsection:adm.gadgets}
 We shall say that a pair $(\rmW,{\rm U})$ of smooth varieties over $k$
is a {\sl good pair} for $\rmG$ if $\rmW$ is a $k$-rational representation of $\rmG$ and
${\rm U} \subsetneq \rmW$ is a $\rmG$-invariant non-empty open subset on which $\rmG$ acts freely and so that ${\rm U}/\rmG$ is a variety. 
It is known ({\sl cf.} \cite[Remark~1.4]{To99}) that a 
good pair for $\rmG$ always exists.
\begin{definition}
\label{defn:Adm-Gad}
A sequence of pairs $ \{ (\rmW_m, {\rm U}_m)|m \ge 1\}$ of smooth varieties
over $k$ is called an {\sl admissible gadget} for $\rmG$, if there exists a
good pair $(\rmW,{\rm U})$ for $\rmG$ such that $\rmW_m = \rmW^{\times^ m}$ and ${\rm U}_m \subsetneq \rmW_m$
is a $\rmG$-invariant open subset such that the following hold for each $m \ge 1$.
\begin{enumerate}
\item
$\left({\rm U}_m \times W\right) \cup \left(W \times {\rm U}_m\right)
\subseteq {\rm U}_{m+1}$ as $\rmG$-invariant open subvarieties.
\item
$\{\codim _{{\rm U}_{m+1}}\left({\rm U}_{m+1} \setminus \left({\rm U}_{m} \times W\right)\right)|m\}$ is 
a strictly increasing sequence,  
\newline \noindent
that is, $\codim_{{\rm U}_{m+2}}\left({\rm U}_{m+2} \setminus 
\left({\rm U}_{m+1} \times W\right)\right) > 
\codim_{{\rm U}_{m+1}}\left({\rm U}_{m+1} \setminus \left({\rm U}_{m} \times W\right)\right)$.
\item
$\{\codim_{W_m}\left(W_m \setminus {\rm U}_m\right)|m\}$ is a strictly increasing sequence, that is, $\codim_{W_{m+1}}\left(W_{m+1} \setminus {\rm U}_{m+1}\right)
> \codim_{W_m}\left(W_m \setminus {\rm U}_m\right)$.
\item
${\rm U}_m$ has a free $\rmG$-action, the quotient ${\rm U}_m/G$ is  a smooth quasi-projective
variety over $k$ and ${\rm U}_m \ra {\rm U}_m/G$ is a principal $\rmG$-bundle.
%\item We let $i_{\it m}: \rmU_{\it m} \ra \rmU_{{\it {m+1}}}$ denote the regular immersion defined by sending $u \mapsto (u, 0) \eps \rmU_m \times W$.
\end{enumerate}
\end{definition}
\begin{lemma}
\label{smooth.quotients.1} Let ${\rm U}$ denote a smooth quasi-projective variety over a field $\oK$ with a free action by the linear algebraic group 
$\rmG$ so that the quotient ${\rm U}/\rmG$ exists as a smooth quasi-projective variety over $\oK$. Then if $\rmX$ is any smooth quasi-projective  variety
over $\oK$, the quotient ${\rm U}{\underset {\rmG} \times}\rmX \cong ({\rm U} {\underset {\oSpec \, \oK} \times}\rmX)/\rmG$ (for the diagonal action of $\rmG$) exists  as a  scheme
over $\oK$. 
\end{lemma}
\begin{proof} 
This follows, for
example, from \cite[Proposition 7.1]{MFK94}.
\end{proof}
\vskip .2cm
 An {\it example} of an admissible gadget for $\rmG$ can be constructed as follows: start  with a good pair $(\rmW, {\rm U})$ for $\rmG$. 
The choice of such a good pair will vary depending on $\rmG$. 
 Choose a faithful $k$-rational representation ${\rm R}$ of $\rmG$ of dimension $n$, that is, $\rmG$ admits a closed
immersion into $\GL(R)$. Then 
$\rmG$ acts freely on an open subset ${\rm U}$ of ${\rm W}= {\rm R}^{\oplus n} = {\rm {End}}({\rm R})$ so that ${\rm U}/\rmG$ is a variety. (For e.g.  ${\rm U}=\GL(R)$.)
Let $\rmZ = \rmW \setminus {\rm U}$.
\vskip .2cm
Given a good pair $(\rmW, {\rm U})$, we now let
\be \begin{equation}
\label{adm.gadget.1}
 \rmW_m = \rmW^{\times m}, {\rm U}_1 = {\rm U} \mbox{ and } {\rm U}_{m+1} = \left({\rm U}_m \times \rmW \right) \cup
\left(\rmW \times {\rm U}_m\right) \mbox{ for }m \ge 1.
\end{equation} \ee
Setting 
$\rmZ_1 = \rmZ$ and $\rmZ_{m+1} = {\rm U}_{m+1} \setminus \left({\rm U}_m \times \rmW\right)$ for 
$m \ge 1$, one checks that $W_m \setminus {\rm U}_m = \rmZ^m$ and
$\rmZ_{m+1} = \rmZ^m \times {\rm U}$.
In particular, $\codim _{W_m}\left(W_m \setminus {\rm U}_m\right) =
m (\codim_{\rmW}(\rmZ))$ and
$\codim _{{\rm U}_{m+1}}\left(\rmZ_{m+1}\right) = (m+1)d - m(\dim(\rmZ))- d =  m (\codim _{\rmW}(\rmZ))$,
where $d = \dim (\rmW)$. Moreover, ${\rm U}_m \to {{\rm U}_m}/\rmG$ is a principal $\rmG$-bundle and the quotient $\rmV_m={\rm U}_m/\rmG$ exists 
 as a smooth quasi-projective scheme. 
\vskip .2cm
We conclude with the following Proposition that shows how $t$-structures may be transferred under an equivalence of derived categories.
Though this result is fairly well-known, we could not find a suitable reference: that is the reason for including it here. But we skip its proof.
\begin{proposition}
 \label{transf.t.struct}
 Let $\T$ and $\T'$ denote two triangulated categories, with $F: \T' \ra \T$ an equivalence 
 of categories together with a left-adjoint $G: \T \ra \T'$, so that the natural transformations $id_{\T} \ra F \circ G$ and
 $G \circ F \ra id_{\T'}$ are natural isomorphisms. Let $(\T'^{\le 0}, \T'^{\ge 0})$ denote a $t$-structure on
 $\bT'$. 
 \vskip .2cm
 Then $(\T^{\le 0}, \T^{\ge 0})$ is a $t$-structure on $\bT$ where $\T^{\le 0}$ ($ \T^{\ge 0}$)
 is defined as the strict full subcategory of $\T$ whose objects are isomorphic to objects of the form $F(A)$, $A$ an object of $\T'^{\le 0}$ ($A$ an
 object of $\T'^{\ge 0}$, \res).
\end{proposition}

\end{document}